\documentclass[11pt]{article}

\usepackage{lineno,hyperref}
\usepackage{fullpage}
\usepackage{authblk}

\usepackage{amsmath}
\usepackage{amsthm}
\usepackage{times}
\usepackage{bm}
\usepackage{natbib}

\usepackage{color}
\definecolor{myColor}{RGB}{0,0,0}
\usepackage[plain,noend]{algorithm2e}

\renewcommand{\Pr}{\mathrm{pr}\, }
\usepackage{galois} % composition function \comp
\usepackage{bm}
\usepackage{amsmath}
\usepackage{amssymb}
\usepackage{mathrsfs}
\usepackage{natbib}
\usepackage{graphicx,caption}
\usepackage{color}
\usepackage{booktabs}
\usepackage[page,title]{appendix}
\usepackage{enumerate}
\usepackage[FIGTOPCAP]{subfigure}
\usepackage{changepage}
\usepackage{datetime}
\newdate{date}{9}{1}{2017}

\DeclareMathOperator{\mytr}{tr}
\DeclareMathOperator{\mydiag}{diag}

\DeclareMathOperator{\myE}{E}
\DeclareMathOperator{\myVar}{var}

\newcommand{\Ba}{\mathbf{a}}    \newcommand{\Bb}{\mathbf{b}}                                        
                                            \newcommand{\Bx}{\mathbf{x}}
        
\newcommand{\BA}{\mathbf{A}}    \newcommand{\BB}{\mathbf{B}}        \newcommand{\BD}{\mathbf{D}}            \newcommand{\BG}{\mathbf{G}}        \newcommand{\BI}{\mathbf{I}}            
        \newcommand{\BO}{\mathbf{O}}    \newcommand{\BP}{\mathbf{P}}            \newcommand{\BS}{\mathbf{S}}        \newcommand{\BU}{\mathbf{U}}    \newcommand{\BV}{\mathbf{V}}        \newcommand{\BX}{\mathbf{X}}
\newcommand{\BY}{\mathbf{Y}}        

\newcommand{\myT}{\intercal}

\newcommand{\bfsym}[1]{\ensuremath{\boldsymbol{#1}}}

              \def\bGamma{\bfsym \Gamma}

              \def\bSigma{\bfsym \Sigma}
         \def\bLambda {\bfsym {\Lambda}}
           \def\bOmega {\bfsym {\Omega}}

 \def\bxi{\boldsymbol {\xi}}
 \def\bPsi{\boldsymbol {\Psi}}
 \def\bzeta{\boldsymbol {\zeta}}
 \def\bZeta{\boldsymbol {\zeta}}
\theoremstyle{plain}
\newtheorem{theorem}{Theorem}

\newtheorem{corollary}{Corollary}
\newtheorem{lemma}{Lemma}

\newtheorem{assumption}{Assumption}

\theoremstyle{definition}

\newtheorem{remark}{Remark}
\theoremstyle{remark}

\begin{document}

\title{
An approximate randomization test for high-dimensional two-sample Behrens-Fisher problem under arbitrary covariances
%in arbitrary dimensions
%in high-dimensional settings
%via pairing within group
}

%\author[1]{Rui Wang}
%\author[1]{Wangli Xu\thanks{Corresponding author\\Email address: wlxu@ruc.edu.cn}}
%\affil[1]{Center for Applied Statistics and School of Statistics, Renmin University of China, Beijing 100872, China
%}
\author{Rui Wang}
\author{Wangli Xu\thanks{Corresponding author\\Email address: wlxu@ruc.edu.cn}}
\affil{Center for Applied Statistics and School of Statistics, Renmin University of China, Beijing 100872, China
}

\maketitle

\begin{abstract}
    This paper is concerned with the problem of comparing the population means of two groups of independent observations.
   % Observations within groups are not assumed to be identically distributed.
    An approximate randomization test procedure based on the test statistic of \cite{Chen2010ATwo} is proposed.
    %The asymptotic properties of the proposed test procedure is studied under fairly general conditions.
    The asymptotic behavior of the test statistic as well as the randomized statistic is studied under weak conditions.
    In our theoretical framework, 
    observations are not assumed to be identically distributed even within groups.
    No condition on the eigenstructure of the covariance matrices is imposed.
    And the sample sizes of the two groups are allowed to be unbalanced.
    Under general conditions,
    all possible asymptotic distributions of the test statistic are obtained.
    We derive the asymptotic level and local power of the approximate randomization test procedure.
    Our theoretical results show that the proposed test procedure
    can adapt to all possible asymptotic distributions of the test statistic and always has correct test level asymptotically.
    Also, the proposed test procedure has good power behavior.
    Our numerical experiments 
    show that the proposed test procedure has favorable performance compared with several alternative test procedures.
\end{abstract}

\noindent {\it Key words}: 
Behrens-Fisher problem;
High-dimensional data;
Randomization test;
Lindeberg principle.

\section{Introduction}
Two-sample mean testing is a fundamental problem in statistics with an enormous range of applications.
In modern statistical applications,
high-dimensional data, where the data dimension may be much larger than the sample size, is ubiquitous.
%including genomics, medical imaging, signal processing.
However, most classical two-sample mean tests are designed for low-dimensional data,
and
may not be feasible, or may have suboptimal power, for high-dimensional data; see, e.g., \cite{Bai1996Effect}.
In recent years, the study of high-dimensional two-sample mean tests has attracted increasing attention.

Suppose that $X_{k,i}$, $i = 1, \ldots, n_k$, $k=1,2$, are independent $p$-dimensional random vectors with $\myE (X_{k,i}) = \mu_k $, $\myVar (X_{k,i}) = \bSigma_{k,i}$.
The hypothesis of interest is
\begin{align}\label{hypothesis}
    \mathcal H_0 :  \mu_1 = \mu_2
    \quad \text{v.s.} \quad
    \mathcal H_1 :  \mu_1 \neq \mu_2.
\end{align}
Denote by
$\bar \bSigma_k = n_k^{-1} \sum_{i=1}^{n_k} \bSigma_{k,i}$ the average covariance matrix within group $k$, $k = 1,2$.
Most existing methods on two-sample mean tests assumed that the observations within groups are identically distributed.
In this case,
$\bSigma_{k,i} = \bar \bSigma_{k}$, $i = 1, \ldots, n_k$, $k = 1, 2$,
and the considered testing problem reduces to the well-known Behrens-Fisher problem.
In this paper, we consider the more general setting where the observations are allowed to have different distributions even within groups.
Also, we consider the case where the data is possibly high-dimensional, that is, the data dimension $p$ may be much larger than the sample size $n_k$, $k = 1, 2$.

Let $\bar X_k = n_k^{-1} \sum_{i=1}^{n_k}X_{k,i}$ and $\BS_k = (n_k - 1)^{-1} \sum_{i=1}^{n_k} (X_{k,i} - \bar X_k)(X_{k,i} - \bar X_k)^\myT $ denote the sample mean vector and the sample covariance matrix of group $k$, respectively, 
$k = 1,2$.
Denote $n = n_1 + n_2$.
A classical test statistic for hypothesis \eqref{hypothesis} is Hotelling's $T^2$ statistic, defined as
\begin{align*}
    \frac{n_1 n_2}{n}  (\bar X_1 - \bar X_2)^\myT 
    \BS^{-1}
    (\bar X_1 - \bar X_2),
\end{align*}
where $\BS = ( n -2)^{-1} \{ (n_1 - 1) \BS_1 + (n_2 - 1) \BS_2 \}$ is the pooled sample covariance matrix.
When $p > n - 2$, the matrix $\BS$ is not invertible, and consequently Hotelling's $T^2$ statistic is not well-defined.
In a seminal work, \cite{Bai1996Effect} 
%removed $\BS_{\mathrm{pooled}}^{-1}$ from Hotelling's $T^2$ statistic and 
proposed the test statistic 
\begin{align*}
    T_{\mathrm{BS}} (\BX_1, \BX_2) = \|\bar X_1 - \bar X_2 \|^2 
    {- \frac{n}{n_1 n_2} \mytr( \BS )}
    ,
\end{align*}
where $\BX_k = (X_{k,1}, \ldots, X_{k,n_k})^\myT$ is the data matrix of group $k$, $k = 1,2$.
The test statistic $T_{\mathrm{BS}}(\BX_1, \BX_2)$ is well-defined for arbitrary $p$.
\cite{Bai1996Effect} assumed that the observations within groups are identically distributed and $\bar \bSigma_1 = \bar \bSigma_2$. % and derived the asymptotic normality of $T_{\mathrm{BS}}$.
The main term of $T_{\mathrm{BS}} (\BX_1, \BX_2)$ is $\|\bar X_1 - \bar X_2\|^2$.
\cite{Chen2010ATwo} removed terms $X_{k,i}^\myT X_{k,i}$ from $\|\bar X_1 - \bar X_2\|^2$ and proposed the test statistic
\begin{align*}
    T_{\text{CQ}} (\BX_1, \BX_2) =
    \sum_{k=1}^2
    \frac{2
    \sum_{i=1}^{n_k}
    \sum_{j=i+1}^{n_k}
    X_{k,i}^\myT X_{k,j}
    }{n_k (n_k -1)}
    -
    \frac{2
\sum_{ i=1}^{n_1}
    \sum_{ j = 1}^{n_2}
    X_{1,i}^\myT X_{2,j}
    }{n_1 n_2}
    .
\end{align*}
The test method of \cite{Chen2010ATwo} can be applied when $\bar \bSigma_1 \neq \bar \bSigma_2$.
Both \cite{Bai1996Effect} and \cite{Chen2010ATwo} used the asymptotical normality of the statistics to determine the critical value of the test.
However, the asymptotical normality only holds for a restricted class of covariance structures.
For example, \cite{Chen2010ATwo} assumed that
\begin{align}\label{eq:condition_chen}
\mytr(\bar \bSigma_k^4 )
=o\left[[\mytr \{(\bar \bSigma_1 + \bar \bSigma_2)^2\}]^2\right]
,\quad
k=1,2
.
\end{align}
%If the condition \eqref{eq:condition_chen} does not hold, then $T_{\mathrm{BS}}$ and $T_{\mathrm{CQ}}$ may not be asymptotically normally distributed; see, e.g., \cite{Aoshima2018Two} and \cite{WANG2018OnTwo}.
%Unfortunately, it is not an easy task to verify the condition \eqref{eq:condition_chen} in practice.
However, 
\eqref{eq:condition_chen} may not hold when $\bar \bSigma_k$ has a few eigenvalues significantly larger than the others, excluding some important senarios in practice. %; see, e.g., \cite{WANG2018OnTwo}.
For example, \eqref{eq:condition_chen} is often violated when the variables are affected by some common factors; see, e.g., \cite{Fan2021Robust} and the references therein.
%\cite{Katayama2013Asymptotic}.
%However, \eqref{eq:condition_chen} is violated for some important models;
%see, e.g., \cite{Aoshima2018Two} and \cite{WANG2018OnTwo}.
If the condition \eqref{eq:condition_chen} is violated, the test procedures of \cite{Bai1996Effect} and \cite{Chen2010ATwo} may have incorrect test level; see, e.g., \cite{Wang2019A}.
For years,
how to construct test procedures that are valid under general covariances has been an important open problem in the field of high-dimensional hypothesis testing;
see the review paper of \cite{Zhidong2015AReview}.
%There is another line of research which utilize extreme value statistics.
%\cite{Tony2013Two} proposed a 
%How to construct test procedures based on $T_{\mathrm{BS}}$ such that the test level is preserved for general covariance matrices.
%and $T_{\mathrm{CQ}}$ was an open problem for years.
%As noted in the review paper \cite{Zhidong2015AReview}, there is a strong need to develop some data based procedures which can be applied in more general settings.
%However, this task is highly nontrivial.
%In fact, 
%permutation test may not work for the Behrens-Fisher problem; see, e.g., \cite{}.
%Also, the classical bootstrap method may not produce correct test level for $T_{\mathrm{CQ}}$; see \cite{Wang2019A}.
%\cite{Wang2019A} showed that for one-sample version of $T_{\mathrm{CQ}}$ statistic,
%In Section ???, we shall see that for $T_{\mathrm{CQ}}$, wild bootstrap method does not work well either.
%Since then, two-sample Behrens-Fisher problem in the high-dimensional setting has been intensively studied by researchers; see \cite{Feng2014Two,Xue2020Distribution,Zhang2021Two-sample} for some recent work.
%\cite{Zhidong2015AReview} for a review up to that time

Intuitively, one may resort to the bootstrap method to control the test level under general covariances.
Surprisingly, 
as shown by our numerical experiments, the empirical bootstrap method does not work well for $T_{\mathrm{CQ}}(\BX_1, \BX_2)$.
Also,
the wild bootstrap method, which is popular in high-dimensional statistics, does not work well either.
We will give heuristic arguments to understand this phenomenon in Section \ref{sec:test_procedure}.
Recently, several methods were proposed to give a better control of the test level of $T_{\mathrm{CQ}}(\BX_1, \BX_2)$ and related test statistics.
%\cite{Xu2019Pearson} derived the universality of quadratic forms for i.i.d.\ random vectors, which can not be used for our purpose.
In the setting of $\bar \bSigma_1 = \bar \bSigma_2$,
\cite{Zhang2019ASimple} considered the statistic $\|\bar X_1 - \bar X_2\|^2$
and proposed to use the Welch-Satterthwaite $\chi^2$-approximation to determine the critical value.
Subsequently,
\cite{Zhang2021Two-sample} extended the test procedure of  \cite{Zhang2019ASimple} to the two-sample Behrens-Fisher problem,
{and \cite{Zhang2020AFurtherStudy} considered a $\chi^2$-approximation of $T_{\mathrm{CQ}}(\BX_1, \BX_2)$.}
However, the test methods of \cite{Zhang2019ASimple}, \cite{Zhang2021Two-sample} and \cite{Zhang2020AFurtherStudy} still require certain conditions on the eigenstructure of $\bar \bSigma_{k}$, $k = 1, 2$.
%Consequently, their method is not applicable to t
In another work,
\cite{Wu2018HypothesisTesting} investigated the general distributional behavior of $\|\bar X_1\|^2$ for one-sample mean testing problem.
%They used the theoretical results in \cite{Xu2019Pearson}.
%Their theoretical results
They proposed a half-sampling procedure to determine the critical value of the test statistic.
A generalization of the method of \cite{Wu2018HypothesisTesting} to the statistc $\|\bar X_1 - \bar X_2\|^2$ for the two-sample Behrens-Fisher problem when both $n_1$ and $n_2$ are even numbers was presented in Chapter 2 of \cite{Lou2020HighDimensional}.
Unfortunately, 
\cite{Wu2018HypothesisTesting} and \cite{Lou2020HighDimensional} did not include detailed proofs of their theoretical results.
Recently, \cite{Wang2019A} considered a randomization test based on the test statistic of \cite{Chen2010ATwo} for the one-sample mean testing problem.
%They investigated the distributional behavior of the corresponding $T_{\mathrm{CQ}}$ statistic and proposed to use a randomization test method to determine the critical value of $T_{\mathrm{CQ}}$.
Their randomization test has exact test level if the distributions of the observations are symmetric and has asymptotically correct level in the general setting.
%Even if symmetric assumption is violated, their test still has correct test level asymptotically.
%To the best of our knowledge,
Although many efforts have been made,
%there is still no existing method that can
%to the best of our knowledge,
no existing test procedure based on
$T_{\mathrm{CQ}}(\BX_1, \BX_2)$ is valid for arbitrary $\bar \bSigma_k$ with rigorous theoretical guarantees.
%how to construct a valid test procedure based on
%$T_{\mathrm{CQ}}(\BX_1, \BX_2)$ for arbitrary $\bar \bSigma_k$ with rigorous theoretical guarantees remains an open problem.
%There is another line of research which considere,
%\cite{Xue2020Distribution} use multiplier bootstrap to control the test level of the extreme value type statistics.

The statistics $T_{\mathrm{BS}}(\BX_1, \BX_2)$ and $T_{\mathrm{CQ}}(\BX_1, \BX_2)$ are based on sum-of-squares.
%These statistics are not scale invariant for individual variables.
{Some researchers investigated variants of sum-of-squares statistics that are scalar-invariant; see, e.g., \cite{Srivastava2008A}, \cite{SRIVASTAVA2013349} and \cite{Feng2014Two}.
These test methods also impose some nontrivial assumptions on the eigenstructure of the covariance matrices.
}
There is another line of research
initiated by \cite{Tony2013Two}
which utilizes extreme value type statistics to test hypothesis \eqref{hypothesis}.
\cite{Chang2017Sim} considered a parametric bootstrap method to determine the critical value of the extreme value type statistics.
The method of \cite{Chang2017Sim} allows for general covariance structures.
Recently, \cite{Xue2020Distribution} investigated a wild bootstrap method. % which can be applied for 
The test procedure in \cite{Xue2020Distribution} does not require that the observations within groups are identically distributed.
The theoretical analyses in \cite{Chang2017Sim} and \cite{Xue2020Distribution} are based on recent results about high-dimensional bootstrap methods developed by \cite{Chernozhukov2013Gaussian} and \cite{Chernozhukov2017Central}.
While their theoretical framework
%of \cite{Chernozhukov2013Gaussian} and \cite{Chernozhukov2017Central}
can be applied to extreme value type statistics,
it can not be applied to sum-of-squares type statistics like $T_{\mathrm{CQ}}(\BX_1, \BX_2)$.
Indeed, we shall see that the wild bootstrap method %used in \cite{Xue2020Distribution}
does not work well for $T_{\mathrm{CQ}}(\BX_1, \BX_2)$.
%In comparison, all other work mentioned above assumed that the observations are identically distributed within groups.
% to determine the critical value of the maximum-type statistics.

%The goal of the present paper is to propose a test procedure  based on $T_{\mathrm{CQ}}$
%which can control the test level under weak conditions.
%Motivated by the randomization test in \cite{Wang2019A},
In the present paper, we propose a new test procedure based on $T_{\mathrm{CQ}}(\BX_1, \BX_2)$.
%The proposed test procedure is based on a new randomization method.
{
In the proposed test procedure,
we use a new randomization method to determine the critical value of the test statistic.
The proposed randomization method is motivated by the randomization method proposed in \cite{Fisher1935experiments}, Section 21.
%However, they are different.
While this randomization method is widely used in one-sample mean testing problem,
it can not be directly applied to two-sample Behrens-Fisher problem.
In comparison,
the proposed randomization method has satisfactory performance for two-sample Behrens-Fisher problem.
}
We investigate the asymptotic properties of the proposed test procedure
and rigorously prove that it has correct test level asymptotically under fairly weak conditions.
Compared with existing test procedures based on sum-of-squares, the present work has the following favorable features.
First, 
the proposed test procedure has correct test level asymptotically
without any restriction on $\bar \bSigma_k$, $k = 1, 2$.
%provided $\bSigma_{k,i} = \bar \bSigma_k$, $i = 1, \ldots, n_k$, $k = 1, 2$.
As a consequence, the proposed test procedure serves as a rigorous solution to the open problem that how to construct a valid test based on $T_{\mathrm{CQ}} (\BX_1, \BX_2)$ under general covariances.
%Also, as in \cite{Xue2020Distribution}, 
Second, 
the proposed test procedure is valid even if the observations within groups are not identically distributed.
To the best of our knowledge, the only existing method that can work in such a general setting
is the test procedure in \cite{Xue2020Distribution}.
However, the test procedure in \cite{Xue2020Distribution} is based on extreme value statistic, which is different from the present work.
%The proposed test procedure works for unbalanced data.
%That is, we do not assume that $n_1$ and $n_2$ are of the same order.
Third,
our theoretical results are valid for arbitrary $n_1$ and $n_2$ as long as $\min(n_1, n_2) \to \infty$.
In comparison, all the above mentioned methods only considered the balanced data, that is, the sample sizes $n_1$ and $n_2$ are of the same order.
We also derive the asymptotic local power of the proposed method.
It shows that 
the asymptotic local power of
the proposed test procedure 
is the same as that of the oracle test procedure 
where the distribution of the test statistic is known.
A key tool for the proofs of our theoretical results is a universality property of the generalized quadratic forms.
This result may be of independent interest.
We also conduct numerical experiments to compare the proposed test procedure with some existing methods.
Our theoretical and numerical results show that the proposed test procedure has good performance in both level and power.

%Although new randomization method does not have exact test level in general, it has correct test level asymptotically under fairly weak conditions.
%However, they did not provide proofs of their results.
%They assumed i.i.d.\ data.
%They required that $\myE (\|X_{k,i}\|^{4+2\delta}) < \infty$ where $0 < \delta \leq 1$.
%However, they did not provide rigorous proofs of their results.
%Compared with existing methods, we do not assume i.i.d.\ data, we do not assume balanced data.
%We give a fully rigorous proof.
%The difficulty of the two sample problem lies in that under the null hypothesis, the common mean vector $\mu_1 = \mu_2$ is still a nuisance paramter.
%A well known method is paired test.
%However, this method can only work for $n_1 = n_2$.
%To overcome this difficulty, we proposed a new strategy based on pairing within group.

\section{Asymptotics of $T_{\mathrm{CQ}}(\BX_1, \BX_2)$}\label{sec:background}
In this section, we investigate the asymptotic behavior of $T_{\mathrm{CQ}}(\BX_1, \BX_2)$ under general conditions.
We begin with some notations.
For two random variables $\xi$ and $\zeta$, let $\mathcal L (\xi)$ denote the distribution of $\xi$, and $\mathcal L (\xi \mid \zeta)$ denote the conditional distribution of $\xi$ given $\zeta$.
    For two probability measures $\nu_1$ and $\nu_2$ on $\mathbb R$, let $\nu_1 - \nu_2$ be the signed measure such that for any Borel set $\mathcal A$, $(\nu_1 - \nu_2)(\mathcal A) = \nu_1(\mathcal A) - \nu_2 (\mathcal A)$.
    Let $\mathscr C_b^3(\mathbb R)$ denote the class of bounded functions with bounded and continuous derivatives up to order $3$.
    It is known that a sequence of random variables $\{\xi_n\}_{n=1}^\infty$ converges weakly to a random variable $\xi$ if and only if for every $f \in \mathscr C_b^3 (\mathbb R)$, $\myE f(\xi_n) \to \myE f(\xi)$; see, e.g., \cite{pollard1984convergence}, Chapter III, Theorem 12.
    We use this property to give a metrization of the weak convergence in $\mathbb R$.
For a function $f\in \mathscr C_b^3(\mathbb R)$, let $f^{(i)}$ denote the $i$th derivative of $f$, $i =1 , 2, 3$.
    For a finite signed measure $ \nu $ on $\mathbb R$, we define the norm $\|\nu\|_3$ as
    $
        \sup_f
            \left|
            \int_{\mathbb R^d} f(\Bx) \, \nu(\mathrm d \Bx)
            \right|
    $,
    where the supremum is taken over all 
    $f \in \mathscr C_b^3 (\mathbb R)$
    such that
    $
          \sup_{\Bx\in \mathbb R}
          |f (x)| \leq 1
    $
    and
    $
            \sup_{\Bx\in \mathbb R}
            |f^{(\ell)} (x)| \leq 1
    $, $\ell = 1, 2, 3$.
    It is straightforward to verify that $\|\cdot\|_3$ is indeed a norm for finite signed measures.
    %For two $m$-dimensional probability measures $\nu_1$ and $\nu_2$, 
    %we define their distance as $\|\nu_1 - \nu_2\|_3$.
    Also, 
    %the metric induced by the norm $\|\cdot\|_3$ metrizes the weak convergence in $\mathbb R$.
    %That is,
    a sequence of probability measures $\{\nu_n\}_{n=1}^\infty$ converges weakly to a probability measure $\nu$ if and only if $\| \nu_n - \nu \|_3 \to 0$; see, e.g., \cite{dudleyProbability}, Corollary 11.3.4.

The test procedure of \cite{Chen2010ATwo}
determines the critical value
based on the asymptotic normality of $T_{\mathrm{CQ}}(\BX_1, \BX_2)$ under the null hypothesis.
%This strategy is used by many test methods for high-dimensional data; see, e.g., \cite{Bai1996Effect}, \cite{chen2010tests} and \cite{Feng2015Multivariate}.
Define $ Y_{k, i} = X_{k,i} - \mu_k$, $i = 1, \ldots, n_k$, $k = 1, 2$.
Then $\myE (Y_{k,i}) = \mathbf{0}_p$.
Under the null hypothesis, we have
    $T_{\text{CQ}} (\BX_1, \BX_2) =
    T_{\text{CQ}} (\BY_1, \BY_2)
    $
where $\BY_k = (Y_{k,1}, \ldots, Y_{k, n_k})^\myT$, $k = 1,2$.
Denote by $\sigma_{T,n}^2$ the variance of $T_{\text{CQ}}(\BY_1, \BY_2)$.
Under certain conditions, \cite{Chen2010ATwo} proved that $T_{\mathrm{CQ}} (\BY_1, \BY_2) / \sigma_{T,n}$ converges weakly to $\mathcal N (0, 1)$.
They reject the null hypothesis if $T_{\mathrm{CQ}} (\BX_1, \BX_2) / \hat \sigma_{T,n}$ is greater than the $1-\alpha$ quantile of $\mathcal N( 0, 1 )$ where $\hat \sigma_{T,n}$ is an estimate of $\sigma_{T,n}$ and $ \alpha \in (0, 1)$ is the test level.
%For ease of reference, we list the conditions of \cite{Chen2010ATwo}.
However, we shall see that normal distribution is just one of the possible asymptotic distributions of $T_{\mathrm{CQ}}(\BY_1, \BY_2)$ in the general setting.

Now we list the main conditions imposed by \cite{Chen2010ATwo} to prove the asymptotic normality of $T_{\mathrm{CQ}}(\BY_1, \BY_2)$.
First, \cite{Chen2010ATwo} imposed the condition \eqref{eq:condition_chen} on the eigenstructure of $\bar \bSigma_k$, $k =1, 2$.
Second, \cite{Chen2010ATwo} assumed that as $n \to \infty$,
\begin{align}\label{chen:cond}
     {n_1}/{n} \to b \in (0, 1).
\end{align}
That is, \cite{Chen2010ATwo} assumed the sample sizes in two groups are balanced.
This condition is commonly adopted by existing test procedures for hypothesis \eqref{hypothesis}.
Third, \cite{Chen2010ATwo} assumed the general multivariate model
\begin{align}\label{chen:cond1}
    Y_{k,i} = \bGamma_k Z_{k,i}, \quad i = 1, \ldots, n_k, \quad k = 1,2,
\end{align}
where $\bGamma_k$ is a $p \times m$ matrix for some $m \geq p$, $k =1, 2$, and $\{Z_{k,i}\}_{i=1}^{n_k}$ are $m$-dimensional independent and identically distributed random vectors such that 
$\myE(Z_{k,i}) = \mathbf 0_m $,
$\myVar(Z_{k,i}) = \BI_m$,
and the elements of $Z_{k,i} = (z_{k,i,1}, \ldots, z_{k,i,m})^\myT$ satisfy
$\myE (z_{k,i,j}^4) =3 + \Delta < \infty$,
and $\myE (z_{k,i,\ell_1}^{\alpha_1} z_{k,i,\ell_2}^{\alpha_2} \cdots z_{k,i,\ell_q}^{\alpha_q}) = \myE (z_{k,i,\ell_1}^{\alpha_1}) \myE( z_{k,i,\ell_2}^{\alpha_2}) \cdots \myE( z_{k,i,\ell_q}^{\alpha_q})$
for a positive integer $q$ such that $\sum_{\ell = 1}^q \alpha_{\ell}
\leq 8$ and distinct $\ell_1, \ldots, \ell_q \in \{1, \ldots, m\}$.
We note that the general multivariate model \eqref{chen:cond1} assumes that the observations within group $k$ are identically distributed, $k = 1, 2$.
Also, the $m$ elements $z_{k,i,1}, \ldots, z_{k,i,m}$ of $Z_{k,i}$ have finite moments up to order $8$ and behave as if they are independent.

As we have noted, the condition \eqref{eq:condition_chen} can be violated when the variables are affected by some common factors.
%This condition is weakened by some recent methods {\color{red} haha}.
%On the other hand,
The condition \eqref{chen:cond} is assumed by most existing test methods.
%commonly adopted in the literature.
However, this condition is not reasonable if the sample sizes are unbalanced.
The general multivariate model
\eqref{chen:cond1} is commonly adopted by many high-dimensional test methods; see, e.g., \cite{Chen2010ATwo}, \cite{Zhang2019ASimple} and \cite{Zhang2021Two-sample}.
However, the conditions on $Z_{k,i}$ may be difficult to verify
{and are not generally satisfied by the elliptical symmetric distributions}.
%in practice.
Also, the model \eqref{chen:cond1} is not valid if the observations within groups are not identically distributed.
%In fact, the conditions \eqref{chen:cond} and \eqref{chen:cond1} can also be easily violated.
%The above three conditions are not always reasonable in practice.
Thus, we would like to investigate the asymptotic behavior of $T_{\mathrm{CQ}}(\BX_1, \BX_2)$
beyond the conditions \eqref{eq:condition_chen}, \eqref{chen:cond} and \eqref{chen:cond1}.
%We would like to investigate the asymptotic behavior of $T_{\mathrm{CQ}}$ under general conditions.
%We would like to 
We consider the asymptotic setting where $n \to \infty$ and
all quantities except for absolute constants are indexed by $n$, a subscript we often suppress.
We make the following assumption on $n_1$ and $n_2$.
\begin{assumption}\label{assumption:n}
    Suppose as $n \to \infty$, $\min(n_1, n_2) \to \infty$.
\end{assumption}
Assumption \ref{assumption:n} only requires that both $n_1$ and $n_2$ tend to infinity, which allows for the unbalanced sample sizes.
This relaxes the condition \eqref{chen:cond}.
We make the following assumption on the distributions of $Y_{k,i}$, $i = 1, \ldots, n_k$, $k =1, 2$.
\begin{assumption}\label{assumption:wangwang}
Assume there exists an absolute constant $\tau \geq 3$ such that for any $p\times p$ positive semi-definite matrix $\BB$,
\begin{align*}
    \myE \big\{ (Y_{k,i}^\myT \BB Y_{k,i})^2 \big\}
    \leq 
    \tau 
     \big\{ \myE (Y_{k,i}^\myT \BB Y_{k,i}) \big\}^2
     < \infty
    %\left\{\mytr (\BB \bSigma_{k,i})\right\}^2
    ,
    \quad
 i = 1, \ldots, n_k,
 \quad
 k =1, 2
    .
\end{align*}
\end{assumption}
%Note that $\myE (Y_{k,i}^\myT \BB Y_{k,i}) = \mytr (\BB \bSigma_{k,i})$.
Intuitively,
Assumption \ref{assumption:wangwang} requires that the fourth moments of $Y_{k,i}$ are of the same order as the squared second moments of $Y_{k,i}$.
%It prevents the cases in which the tail of the distribution of $Y_{k,i}$ is too heavy.
%In particular, it is required that $\myE (\|Y_{k,i}\|^4)< \infty$.
%If $Y_{k,i} \sim \mathcal N (\mathbf 0_p, \bSigma_{k,i})$, then Assumption \ref{assumption:wangwang} is satisfied with $\tau = 3$.
%Also,
%if the general multivariate model \eqref{chen:cond1} holds,
%then Assumption \ref{assumption:wangwang} holds with $\tau = 3 + \max(\Delta, 0)$; see, e.g., \cite{chen2010tests}, Proposition A.1.(i).
%Note that Assumption \ref{assumption:wangwang} does not impose a model of the observations.
%Also, Assumption \ref{assumption:wangwang} does not require that the observations are identically distributed.
We shall see that this assumption is fairly weak.
However, 
%Assumption \ref{assumption:wangwang}
the above inequality
is required to be satisfied for all positive semi-definite matrix $\BB$, which may not be straightforward to check in some cases.
The following lemma gives a sufficient condition of
Assumption \ref{assumption:wangwang}.
%holds for a multivariate model which is much more general than the model \eqref{chen:cond1}.

\begin{lemma}\label{lemma:11_30}
    Suppose 
    $Y_{k,i} = \bGamma_{k,i} Z_{k,i}$, $i = 1, \ldots, n_k$, $k = 1, 2$,
    where $\bGamma_{k,i}$ is an arbitrary $p \times m_{k,i}$ matrix and $m_{k,i}$ is an arbitrary positive integer.
    Suppose
    $\myE(Z_{k,i}) = \mathbf 0_{m_{k,i}} $,
    $\myVar(Z_{k,i}) = \BI_{m_{k,i}}$,
    and the elements of $Z_{k,i} = (z_{k,i,1}, \ldots, z_{k,i,m_{k,i}})^\myT$ satisfy
$\myE (z_{k,i,j}^4) \leq C < \infty$ where $C$ is an absolute constant.
Suppose for any distinct $\ell_1, \ell_2, \ell_3, \ell_4 \in \{1, \ldots, m_{k,i}\}$,
\begin{align}\label{eq:11_30a}
    &
 \myE (
z_{k,i,\ell_1}
z_{k,i,\ell_2} z_{k,i,\ell_3} z_{k,i,\ell_4}) = 0 
,
\quad
 \myE (
z_{k,i,\ell_1}
z_{k,i,\ell_2}
z_{k,i,\ell_3}^2
) = 0,
\quad
 \myE (
z_{k,i,\ell_1}
z_{k,i,\ell_2}^3
) = 0
.
\end{align}
Then Assumption \ref{assumption:wangwang} holds with $\tau = 3C$.
\end{lemma}

It can be seen that the conditions of Lemma \ref{lemma:11_30} is strictly weaker than the multivariate model \eqref{chen:cond1}.
In fact, Lemma \ref{lemma:11_30} does not require $m_{k,i} \geq p$, nor does it require the finiteness of $8$th moments of $z_{k,i,j}$.
Also, the moment conditions in \eqref{eq:11_30a} are much weaker than that required by the multivariate model \eqref{chen:cond1}.
%Thus, the conditions of Lemma \ref{lemma:11_30} are satisfied the multivariate model \eqref{chen:cond1}.
In addition to the multivariate model \eqref{chen:cond1},
the conditions of Lemma \ref{lemma:11_30} also allow for an important class of elliptical symmetric distributions.
%Suppose $Y_{k,i}$ has an elliptical symmetric distribution.
In fact, if $Y_{k,i}$ has an elliptical symmetric distribution,
then $Y_{k,i}$ can be written as $Y_{k,i} = \eta_{k,i} \bGamma_{k,i} U_{k,i}$, 
where $\bGamma_{k,i}$ is a $p \times m_{k,i}$ matrix,
$U_{k,i}$ is a random vector distributed uniformly on the unit sphere in  $\mathbb R^{m_{k,i}}$, and $\eta_{k,i}$ is a nonnegative random variable which is independent of $U_{k,i}$; see, e.g., \cite{Fang1990Symmetric}, Theorem 2.5.
%In this case, we have $Y_{k,i} = \bGamma_{k,i} Z_{k,i}$ where $Z_{k,i} = \eta_{k,i} U_{k,i}$.
Suppose there is an absolute constant $C$ such that $\myE(\eta_{k,i}^4) \leq C \{\myE(\eta_{k,i}^2)\}^2 < \infty$.
Then from the symmetric property of $U_{k,i}$ and the independence of $\eta_{k,i}$ and $U_{k,i}$,
the conditions of Lemma \ref{lemma:11_30} hold.
%Thus,
%the conditions of Lemma \ref{lemma:11_30}
In comparison, the multivariate model \eqref{chen:cond1} does not allow for elliptical symmetric distributions in general.

%Note that the general multivariate model in \cite{chen2010tests} is widely adopted by many subsequent works.
%most existing test methods based on sum-of-squares assume the 
%In this sence, Assumption \ref{assumption:wangwang} is fairly weak.

%Assumption \ref{assumption:wangwang} imposes conditions on the tail behavior of the distributions.
Most existing test methods for the hypothesis \eqref{hypothesis} assume that the observations within groups are identically distributed.
In this case, Assumptions \ref{assumption:n} and \ref{assumption:wangwang} are all we need,
and we completely avoid an assumption on the eigenstucture of $\bar \bSigma_k$ like \eqref{eq:condition_chen}.
In the general setting, $\bSigma_{k,i}$ may be different within groups, and we make the following assumption to avoid the case in which there exist certain observations with significantly larger variance than the others.
\begin{assumption}\label{assumption7}
    Suppose that as $n \to \infty$,
    \begin{align*}
        \frac{1}{n_k^2}
        \sum_{i=1}^{n_k}
        \mytr(\bSigma_{k,i}^2)
        =
        o\left\{
        \mytr
        (\bar \bSigma_k^2
        )
    \right\}
    ,
    \quad
    k = 1, 2
        .
    \end{align*}
\end{assumption} 
%Note that most existing methods assume that the observations within groups are identically distributed.
%In this case,
If the covariance matrices within groups are equal, i.e., $\bSigma_{k,i} = \bar \bSigma_k$, $i = 1, \ldots, n_k$, $k =1 ,2$,
then Assumption \ref{assumption7} 
%becomes
%    \begin{align*}
%        \frac{1}{n_k}
%        \mytr
%        (\bar \bSigma_k^2
%        )
%        =
%        o\left\{
%        \mytr
%        (\bar \bSigma_k^2
%        )
%    \right\}
%    ,
%    \quad
%    k = 1, 2
%        ,
%    \end{align*}
%    which 
    holds for arbitrary $\bar \bSigma_k$, $k = 1, 2$, provided $\min(n_1, n_2) \to \infty$ as $n \to \infty$.
    %That is, for identically distributed data, we do not impose any condition on $\bar \bSigma_k$, $k = 1, 2$.
    In this view, Assumption \ref{assumption7} imposes no restriction on $\bar \bSigma_k$, $k = 1, 2$.
    %To the best of our knowledge,
    %the test procedure of \cite{Xue2020Distribution} is the only existing method for hypothesis \eqref{hypothesis} that does not assume the observations within groups are identically distributed.
    %\cite{Xue2020Distribution} also imposed conditions to avoid significantly 
    %Thus, Assumption \ref{assumption7} is fairly weak.

    %First we investivate the distributional behavior of $T_{\mathrm{CQ}} (\BY_1, \BY_2)$.
    %Denote $ \sigma_{T,n}^2 =  \myVar(T_{\text{CQ}}(\BY_1, \BY_2))$.
    Define $\bPsi_n = 
             n_1^{-1} \bar \bSigma_1 + n_2^{-1} \bar \bSigma_2
    $.
    Let $\bxi_p$ be a $p$-dimensional standard normal random vector.
    We have the following theorem.
\begin{theorem}%[Universality of $T_{\mathrm{CQ}}$]
    \label{thm:universality_TCQ}
    Suppose Assumptions \ref{assumption:n}, \ref{assumption:wangwang} and \ref{assumption7} hold, and $\sigma_{T,n}^2 > 0$ for all $n$.
    Then as $n \to \infty$,
    \begin{align*}
        \left\|
        \mathcal L \left\{
            \frac{T_{\mathrm{CQ}} (\BY_1, \BY_2) }{\sigma_{T,n} }
        \right\}
        -
        \mathcal L \left[
        \frac{
         \bxi_p^\myT
         \bPsi_n
         \bxi_p
            - 
            \mytr(\bPsi_n)
        }{
        \{
            2\mytr(
            \bPsi_n^2
        )
    \}^{1/2}
    }
\right]
    \right\|_3
    \to 0.
    \end{align*}
\end{theorem}
Theorem \ref{thm:universality_TCQ} characterizes the general distributional behavior of $T_{\mathrm{CQ}}(\BY_1, \BY_2)$.
It implies that the distributions of 
$T_{\mathrm{CQ}}(\BY_1, \BY_2) / \sigma_{T,n}$
and
$
{
    \{\bxi_p^\myT
         \bPsi_n
         \bxi_p
            - 
            \mytr(\bPsi_n)
        \}
        }/{
        \{
            2\mytr(
            \bPsi_n^2
        )
    \}^{1/2}
    }
$
are equivalent asymptotically.
To gain further insights on the distributional behavior of $T_{\mathrm{CQ}}(\BY_1, \BY_2)$,
we would like to derive the asymptotic distributions of $T_{\mathrm{CQ}}(\BY_1, \BY_2) / \sigma_{T,n}$.
However,
Theorem \ref{thm:universality_TCQ} implies that 
$T_{\mathrm{CQ}}(\BY_1, \BY_2) / \sigma_{T,n}$
may not converge weakly in general.
Nevertheless,
$\mathcal L \{T_{\mathrm{CQ}}(\BY_1, \BY_2) / \sigma_{T,n}\}$
is uniformly tight
and
we can use Theorem \ref{thm:universality_TCQ}
to derive all possible asympotitc distributions of $T_{\mathrm{CQ}}(\BY_1, \BY_2) / \sigma_{T,n}$.

\begin{corollary}\label{corollary:the}
    Suppose the conditions of Theorem \ref{thm:universality_TCQ} hold.
    Then
    $
    \mathcal L \left\{
        {T_{\mathrm{CQ}} (\BY_1, \BY_2) }/{\sigma_{T,n} }
    \right\}
    $
    is uniformly tight and
    all possible asymptotic distributions of $
        {T_{\mathrm{CQ}} (\BY_1, \BY_2) }/{\sigma_{T,n} }
        $
        are give by
    \begin{align}\label{eq:representation}
    \mathcal L 
    \left\{
            (1 - \sum_{i=1}^{\infty} \kappa_{i}^2)^{1/2}
            \xi_0
    +
    2^{-1/2}
    \sum_{i=1}^{\infty} \kappa_i (\xi_i^2-1)
\right\}
    ,
    \end{align}
    where
$\{\xi_i \}_{i=0}^\infty$ is a sequence of independent standard normal random variables,
$\{\kappa_i\}_{i=1}^\infty$ is a sequence of positive numbers such that $\sum_{i=1}^\infty \kappa_i^2 \in [0, 1]$ and 
    $\sum_{i=1}^{\infty} \kappa_i (\xi_i^2-1)$ is the almost sure limit of $\sum_{i=1}^{r} \kappa_i (\xi_i^2-1)$ as $r \to \infty$.
\end{corollary}
\begin{remark}
    From L\'evy's equivalence theorem and three-series theorem (see, e.g., \cite{dudleyProbability}, Theorem 9.7.1 and Theorem 9.7.3), the series $\sum_{i=1}^{\infty} \kappa_i (\xi_i^2-1)$ converges almost surely and weakly.
    Hence the distribution \eqref{eq:representation} is well defined.
    %Hence the asymptotic distributions in Corollary \ref{corollary:the} are well defined.
\end{remark}

Corollary \ref{corollary:the} gives a full characteristic of the possible asymptotic distributions of
$T_{\mathrm{CQ}}(\BY_1, \BY_2) / \sigma_{T,n}$.
In general, these possible asymptotic distributions are weighted sums of an independent normal random variable and centered $\chi^2$ random variables.
From the proof of Corollary \ref{corollary:the}, 
the parameters $\{\kappa_i\}_{i=1}^\infty$ are in fact the limit of the eigenvalues of the matrix $
\bPsi_n / \{\mytr(\bPsi_n^2) \}^{1/2}
$ along a subsequence of $\{n\}$.
If 
$\sum_{i=1}^\infty \kappa_i^2 = 0$, then \eqref{eq:representation} becomes the standard normal distribution,
%which is free of nuisance parameters.
and the test procedure of \cite{Chen2010ATwo} has correct level asymptotically.
On the other hand, if $\sum_{i=1}^\infty \kappa_i^2 = 1$, then \eqref{eq:representation} becomes the distribution of a weighted sum of independent centered $\chi^2$ random variables.
This case was considered in \cite{Zhang2021Two-sample}.
However, these two settings are just special cases among all possible asymptotic distributions where $\sum_{i=1}^\infty \kappa_i^2  \in [0, 1]$.
%In general, 
%the test procedure of \cite{Chen2010ATwo} does not have correct level.
In general, the distribution \eqref{eq:representation} relies on the nuisance paramters $\{\kappa_i\}_{i=1}^\infty$.
To construct a test procedure based on the asymptotic distributions, one needs to estimate these nuisance parameters consistently.
Unfortunately, the estimation of the eigenvalues of high-dimensional covariance matrices may be a highly nontrivial
%, if not impossible,
task; see, e.g., \cite{Kong2017SpectrumEstimation} and the references therein.
Hence in general,
it may not be a good choice to construct test procedures based on the asymptotic distributions.

\section{Test procedure}\label{sec:test_procedure}
%On the other hand, if $\sum_{i = 1}^\infty \kappa_i^2 = 1$, then \eqref{eq:representation} is the distribution of a weighted sum of centered $\chi^2$ random variables.
%This case was considered in \cite{}.
An intuitive idea to control the test level of $T_{\mathrm{CQ}} (\BX_1, \BX_2)$ in the general setting is to use the bootstrap method.
Surprisingly, the empirical bootstrap method and wild bootstrap method may not work for $T_{\mathrm{CQ}}(\BX_1, \BX_2)$.
%hypothesis \eqref{hypothesis}.
This phenomenon will be shown by our numerical experiments.
For now, we give a heuristic argument to understand this phenomenon.
First we consider the empirical bootstrap method.
Suppose the resampled observations $\{X_{k,i}^*\}_{i=1}^{n_k}$ are uniformly sampled from $\{X_{k,i} - \bar X_k\}_{i=1}^{n_k}$ with replacement, $k = 1,2$.
Denote $\BX_k^* = (X_{k,1}^*, \ldots, X_{k,n_k}^*)^\myT$, $k = 1,2$.
The empirical bootstrap method uses the conditional distribution $\mathcal L \{T_{\mathrm{CQ}} (\BX_1^*, \BX_2^*) \mid \BX_1, \BX_2 \}$ to approximate the null distribution of $T_{\mathrm{CQ}}(\BX_1, \BX_2)$.
If this bootstrap method can work, one may expect that the first two moments of
$\mathcal L \{T_{\mathrm{CQ}} (\BX_1^*, \BX_2^*) \mid \BX_1, \BX_2 \}$ 
can approximately match the first two moments of $T_{\mathrm{CQ}}(\BX_1, \BX_2)$ under the null hypothesis.
Under the null hypothesis, we have $ \myE \{T_{\mathrm{CQ}}(\BX_1, \BX_2) \} = 0$.
    Lemma \ref{lemma:712}
implies that 
under Assumptions \ref{assumption:n} and \ref{assumption7} and under the null hypothesis,
%as $n \to \infty$,
\begin{align*}
    \myVar \{T_{\mathrm{CQ}} (\BX_1, \BX_2) \}
=
    \sigma_{T,n}^2
=
\{1+o(1)\} 2   
        \mytr
        (
        \bPsi_n^2
        )
        .
\end{align*}
On the othen hand,
it is straightforward to show that
$\myE \{T_{\mathrm{CQ}} (\BX_1^*, \BX_2^*) \mid \BX_1, \BX_2 \} = 0$.
Also, under Assumptions \ref{assumption:n} and \ref{assumption7},
\begin{align*}
    \myVar \{T_{\mathrm{CQ}} (\BX_1^*, \BX_2^*) \mid \BX_1, \BX_2 \} = 
    \{1+o_p(1)\} 2 \mytr\{ ( n_1^{-1} \BS_1 + n_2^{-1} \BS_2 )^2 \}
.
\end{align*}
Unfortunately,
$
\mytr\{ ( n_1^{-1} \BS_1 + n_2^{-1} \BS_2 )^2 \}
$
is not a ratio-consistent estimator of
$
        \mytr
        (
        \bPsi_n^2
        )
        $ even in the settings where $X_{k,i}$ is normally distributed, the covariance matrices $\bSigma_{k,i}$, $i = 1, \ldots, n_k$, $k = 1, 2$, are all equal and $n_1 = n_2$; see, e.g., \cite{Bai1996Effect} and \cite{Zhou2017a}.
        Consequently, the empirical bootstrap method may not be valid for $T_{\mathrm{CQ}}(\BX_1, \BX_2)$.

We turn to the wild bootstrap method.
Recently, the wild bootstrap method has been widely used for extreme value type statistics in the high-dimensional setting; 
see, e.g., \cite{Chernozhukov2013Gaussian}, \cite{Chernozhukov2017Central},
\cite{Xue2020Distribution} and \cite{Deng2020Beyond}.
For the wild bootstrap method, the resampled observations are defined as $X_{k,i}^* = \varepsilon_{k,i}  (X_{k,i} - \bar X_k) $, $i = 1, \ldots, n_k$, $k = 1, 2$, where $\{\varepsilon_{k,i}\}$ are independent and identically distributed random variables with $\myE (\varepsilon_{k,i}) = 0$ and $\myVar (\varepsilon_{k,i}) = 1$, and are independent of the original data $\BX_1$, $\BX_2$.
We have $\myE \{ T_{\mathrm{CQ}} (\BX_1^*, \BX_2^*) \mid \BX_1, \BX_2 \} = 0$.
With some tedious but straightforward derivations, it can be seen that under Assumptions \ref{assumption:n} and \ref{assumption7},
\begin{align*}
    \myE
    \{
    \myVar(T_{\mathrm{CQ}} (\BX_1^*, \BX_2^*) \mid \BX_1, \BX_2)
\}
    =&
    \{1+o(1)\} 2 \mytr(\bPsi_n^2)
    +
    b
    ,
\end{align*}
where $b$ is the bias term and satisfies
\begin{align*}
    b=&
    \{1+o(1)\}
        \sum_{k=1}^2
        \sum_{i=1}^{n_k}
        \frac{4}{n_k^5}
        \myE \{(Y_{k,i}^\myT Y_{k,i} )^2\}
    -
    \{1+o(1)\}
    \sum_{k=1}^2
    \frac{2}{n_k^4}
        \{\mytr(\bar \bSigma_k)\}^2
        \\
%    \geq&
%        (1+o(1))
%        \sum_{k=1}^2
%        \sum_{i=1}^{n_k}
%        \frac{4}{n_k^5}
% \{\mytr (\bSigma_{k,i})\}^2
%    -
%    (1+o(1))
%    \sum_{k=1}^2
%    \frac{2}{n_k^4}
%        \{\mytr(\bar \bSigma_k)\}^2
%        \\
    \geq&
    \{1+o(1)\}
        \sum_{k=1}^2
        \sum_{i=1}^{n_k}
        \frac{2}{n_k^5}
 \{\mytr (\bSigma_{k,i})\}^2
 .
\end{align*}
%From Jensen's inequality,
%we have
%$
%\myE \{(Y_{k,i}^\myT Y_{k,i} )^2\}
%\geq
 %\{\myE(Y_{k,i}^\myT Y_{k,i} )\}^2
 %=
 %\{\mytr (\bSigma_{k,i})\}^2
%$.
%Hence, the bias term $b$ satisfies
%\begin{align*}
%\end{align*}
The above inequality implies that the bias term $b$ may not be negligible compared with $2\mytr(\bPsi_n^2)$.
For example, if $\bSigma_{k,i} = \BI_p$, $i = 1, \ldots, n_k$, $k = 1, 2$, then we have
$2\mytr(\bPsi_n^2) = 2 (n_1^{-1} + n_2^{-1})^2 p$
and $b \geq \{1+o(1)\} 2 (n_1^{-4} + n_2^{-4}) p^2 $.
In this case, the bias term $b$ is not negligible provided $n^2 = O(p)$.
Thus, the wild bootstrap method may not be valid for $T_{\mathrm{CQ}}(\BX_1, \BX_2)$ either.

%If $\varepsilon_{k,i}$ is standard normal random variable, then
%Perhaps, the harshest condition among the three conditions is \eqref{eq:condition_chen}.
%For example, \eqref{eq:condition_chen} is often violated when the varaibles are affected by some common factors; see, e.g., \cite{Fan2021Robust} and the references therein.

%Perhaps, the harshest condition among the three conditions is \eqref{eq:condition_chen}.
We have seen that the methods based on asymptotic distributions and the bootstrap methods may not work well for the test statistic $T_{\mathrm{CQ}} (\BX_1, \BX_2)$.
These phenomenons imply that it is highly nontrivial to construct a valid test procedure based on $T_{\mathrm{CQ}}(\BX_1, \BX_2)$.
%We shall propose a new test procedure based on the idea of
To construct a valid test procedure,
we resort to the idea of the randomization test, a powerful tool in statistical hypothesis testing.
The randomization test is an old idea and can at least date back to \cite{Fisher1935experiments}, Section 21;
see \cite{Hoeffding1952The},
\cite{Lehmann}, Section 15.2, \cite{Zhu2005NonparametricMonteCarlo}
and
\cite{Hemerik2017Exact}
for general frameworks and extensions of the randomization tests.
The original randomization method considered in
\cite{Fisher1935experiments}, Section 21
can be abstracted into the following general form.
%Now we give the basic motivation for the proposed test procedure.
Suppose $\xi_1, \ldots, \xi_n$ are independent $p$-dimensional random vectors such that $\mathcal L (\xi_i) = \mathcal L (- \xi_i)$.
%where $\mathcal L (\xi_i)$ is the distribution measure of $\xi_i$, a probability measure on $\mathbb R^p$, $i = 1, \ldots, n$.
Let $T(\xi_1, \ldots, \xi_n)$ be any statistics taking values in $\mathbb R$.
Suppose $\epsilon_1, \ldots, \epsilon_n$ are independent and identically distributed Rademacher random variables, i.e., $\Pr(\epsilon_i = 1) = \Pr(\epsilon_i = -1) = 1/2$, and are independent of $\xi_1, \ldots, \xi_n$.
Define the conditional cumulative distribution function $\hat F (\cdot)$ as $ \hat F(x) = \Pr\{T(\epsilon_1 \xi_1, \ldots, \epsilon_n \xi_n) \leq x \mid \xi_1, \ldots, \xi_n \}$.
Then from the theory of randomization test (see, e.g., \cite{Lehmann}, Section 15.2), 
for any $\alpha \in (0, 1)$,
\begin{align*}
    \Pr\left\{
        T (\xi_1, \ldots,  \xi_n)
        > \hat F^{-1} (1-\alpha)
    \right\}
    \leq \alpha,
\end{align*}
where %for any distribution function $F(\cdot):$
for any right continuous cumulative distribution function $F(\cdot)$ on $\mathbb R$, 
$F^{-1}(q) = \min \{{x\in \mathbb R}: F(x) \geq q \}$ for $q \in (0,1)$.
Also, under mild conditions, the difference between the above probability and $\alpha$ is negligible.

To apply Fisher's randomization method to specific problems,
the key is to construct random variables $\xi_i$ such that $\mathcal L (\xi_i) = \mathcal L (- \xi_i)$ under the null hypothesis.
This randomization method can be directly applied to one-sample mean testing problem, as did in \cite{Wang2019A}.
%Recently, \cite{Wang2019A} considered a randomization test for one-sample mean testing problem in the high-dimensional setting.
%which has good theoretical properties.
However, it can not be readily applied to the testing problem \eqref{hypothesis}.
In fact, the mean vector $\mu_k$ of $X_{k,i}$ is unknown under the null hypothesis,
and consequently,
one can not expect that $\mathcal L (X_{k,i}) = \mathcal L (- X_{k, i})$ holds under the null hypothesis.
As a result, $X_{k,i}$ can not serve as $\xi_i$ in Fisher's randomization method.

We observe that the difference $X_{k,i} - X_{k,i+1}$ has zero means and hence is free of the mean vector $\mu_k$.
Also, if
$ \mathcal L (X_{k,i}) = \mathcal L (X_{k,i+1}) $,
then $\mathcal L (X_{k,i} - X_{k,i+1} ) = \mathcal L ( X_{k,i+1} - X_{k,i} ) $.
These facts 
imply that Fisher's randomization method may be applied to
the random vectors
$\tilde X_{k,i} = (X_{k,2i} -  X_{k,2i - 1}) / 2 $, $i = 1, \ldots, m_k$, $k =1, 2$ where $m_k = \lfloor n_k/2 \rfloor$, $k =1 ,2$.
%Hence our idea is to generate a reference distribution for $T_{\mathrm{CQ}}(\BX_1, \BX_2)$ using $\tilde X_{k,i}$, $i = 1, \ldots, m_k$, $k = 1, 2$.
Define
\begin{align*}
    T_{\mathrm{CQ}} (
    \tilde \BX_1, \tilde \BX_2
    )
    =&
    \sum_{k=1}^{2}
\frac{2
\sum_{i=1}^{m_k}
\sum_{j=i+1}^{m_k}
\tilde X_{k,i}^\myT 
\tilde X_{k,j}
}{m_k (m_k - 1)}
-
\frac{2
\sum_{i=1}^{m_1}
\sum_{j=1}^{m_2}
\tilde X_{1,i}^\myT \tilde X_{2,j}
}{m_1 m_2}
,
\end{align*}
where
$\tilde \BX_k = (\tilde X_{k,1}, \ldots, \tilde X_{k, m_k})^\myT$, $k = 1, 2$.
Let $E = (\epsilon_{1, 1}, \ldots, \epsilon_{1, m_1}, \epsilon_{2, 1}, \ldots, \epsilon_{2, m_2})^\myT$
where $\{\epsilon_{k, i}\}$
%Let $\epsilon_{k, i}$, $i = 1, \ldots, m_k$, $k =1, 2$, 
are independent and identically distributed Rademacher random variables which are independent of $\tilde \BX_1$, $\tilde \BX_2$.
%Let $E_k = (\epsilon_{k, 1}, \ldots, \epsilon_{k, m_k})^\myT$, $k =1 , 2$.
Define randomized statistic 
\begin{align*}
    T_{\mathrm{CQ}} (
    E
    ;
    \tilde \BX_1, \tilde \BX_2
    )
    =&
    \sum_{k=1}^{2}
\frac{2
\sum_{i=1}^{m_k}
\sum_{j=i+1}^{m_k}
\epsilon_{k,i}
\epsilon_{k,j}
\tilde X_{k,i}^\myT 
\tilde X_{k,j}
}{m_k (m_k - 1)}
-
\frac{2
\sum_{i=1}^{m_1}
\sum_{j=1}^{m_2}
\epsilon_{1,i}
\epsilon_{2,j}
\tilde X_{1,i}^\myT \tilde X_{2,j}
}{m_1 m_2}
.
\end{align*}
Define the conditional distribution function 
$
    \hat F_{\mathrm{CQ}} (x)
    =
    \Pr
    \{
        T_{\mathrm{CQ}}(
        E;
        \tilde \BX_1, \tilde \BX_2
        )
\leq x
\mid
\tilde \BX_1, \tilde \BX_2
\}
$.
From Fisher's randomization method,
if $\mathcal L (\tilde X_{k,i}) = \mathcal L (-\tilde X_{k,i})$,
then for $\alpha \in (0,1)$,
$\Pr \{T_{\mathrm{CQ}}(\tilde \BX_1, \tilde \BX_2) > \hat F_{\mathrm{CQ}}^{-1}(1-\alpha) \} \leq \alpha$.
It can be seen that
$T_{\mathrm{CQ}}(\tilde \BX_1, \tilde \BX_{2})$
and 
$T_{\mathrm{CQ}}(\BX_1, \BX_{2})$
take a similar form.
Also,
under the null hypothesis,
$\myE \{ T_{\mathrm{CQ}}(\tilde \BX_1, \tilde \BX_{2})\}
    =
    \myE \{ T_{\mathrm{CQ}}(\BX_1, \BX_{2}) \} = 0$
    and $\myVar \{ T_{\mathrm{CQ}} (\tilde \BX_1, \tilde \BX_2) \} = \{ 1+o(1) \}  \myVar \{ T_{\mathrm{CQ}} (\BX_1, \BX_2) \} $.
    Thus, it may be expected that $\mathcal L \{ T_{\mathrm{CQ}}( \BX_1, \BX_{2} ) \} \approx \mathcal L \{ T_{\mathrm{CQ}}(\tilde \BX_1, \tilde \BX_{2}) \} $ under the null hypothesis.
On the other hand, the classical results of \cite{Hoeffding1952The} on randomization tests give the insight that
the randomness of $\hat F_{\mathrm{CQ}}^{-1}(1-\alpha)$ may be negligible for large samples.
From the above insights,
it can be expected that under the null hypothesis,
for $\alpha \in (0,1)$, 
\begin{align*}
\Pr \{
T_{\mathrm{CQ}}(\BX_1, \BX_2) > \hat F_{\mathrm{CQ}}^{-1}(1-\alpha)
\} 
    \approx
\Pr \{T_{\mathrm{CQ}}(\tilde \BX_1, \tilde \BX_2) > \hat F_{\mathrm{CQ}}^{-1}(1-\alpha)\} \approx \alpha
.
\end{align*}
Motivated by the above heuristics, we propose a new test procedure which rejects the null hypothesis if
\begin{align*}
T_{\mathrm{CQ}}(\BX_1, \BX_2) > \hat F_{\mathrm{CQ}}^{-1}(1-\alpha)
.
\end{align*}

{
While the assumption $\mathcal L (\tilde X_{k,i}) = \mathcal L (- \tilde X_{k,i})$ is used in the above heuristic arguments, 
it will not be assumed in our theoretical analysis.
This generality may not be surprising.
Indeed,
for low-dimensional testing problems, it is known that such symmetry conditions can often be relaxed for randomization tests; see, e.g., \cite{Romano1990On}, \cite{Chung2013Exact} and \cite{Canay2017Randomization}.
%For the proposed test procedure,
%our theoretical and numerical results will show that it is also valid without the assumption $\mathcal L (\tilde X_{k,i}) = \mathcal L (- \tilde X_{k,i})$.
}
In the proposed procedure,
the conditional distribution
    $
    \mathcal L \{
T_{\mathrm{CQ}}(
E;
\tilde \BX_1, \tilde \BX_2
)
\mid \tilde \BX_1, \tilde \BX_2 \}
    $
    is used to
    approximate the null distribution of $T_{\mathrm{CQ}} (\BX_1, \BX_2)$.
    We have
    $
    \myE\{
T_{\mathrm{CQ}}(
E;
\tilde \BX_1, \tilde \BX_2
)
\mid \tilde \BX_1, \tilde \BX_2 \}
=0
    $.
From Lemma \ref{lemma:variance},
under Assumptions 
\ref{assumption:n}-\ref{assumption7}, 
%\ref{assumption:n}, \ref{assumption:wangwang} and \ref{assumption7}, 
we have 
\begin{align*}
    \myVar\{
T_{\mathrm{CQ}}(
E;
\tilde \BX_1, \tilde \BX_2
)
\mid \tilde \BX_1, \tilde \BX_2 \}
= \{ 1+o_P(1) \} \sigma_{T,n}^2.
\end{align*}
That is, the first two moments of
$
    \mathcal L \{
T_{\mathrm{CQ}}(
E;
\tilde \BX_1, \tilde \BX_2
)
\mid \tilde \BX_1, \tilde \BX_2 \}
$
can match those of $\mathcal L \{T_{\mathrm{CQ}}(\BX_1, \BX_2)\}$.
As we have seen, this favorable property is not shared by the empirical bootstrap method and wild bootstrap method.

We should emphasis that the proposed test procedure is only inspired by the randomization test, and is not a randomization test in itself.
As a consequence, it can not be expected that the proposed test can have an exact control of the test level.
In fact,
even if the observations are $1$-dimensional and normally distributed,
the exact control of the test level for Behrens-Fisher problem is not trivial; see \cite{Linnik1966Latest} and the references therein.
Nevertheless, we shall show that the proposed test procedure can control the test level asymptotically under Assumptions 
\ref{assumption:n}-\ref{assumption7}.

The conditional distribution
$
    \mathcal L \{
T_{\mathrm{CQ}}(
E;
\tilde \BX_1, \tilde \BX_2
)
\mid \tilde \BX_1, \tilde \BX_2 \}
$
is a discrete distribution uniformly distributed on $2^{m_1 + m_2}$ values.
Hence it is not feasible to compute the exact quantile of $\hat F_{\mathrm{CQ}}(\cdot)$.
In practice, one can use a finite sample from
$
    \mathcal L \{
T_{\mathrm{CQ}}(
E;
\tilde \BX_1, \tilde \BX_2
)
\mid \tilde \BX_1, \tilde \BX_2 \}
$
to approximate the $p$-value of the proposed test procedure; see, e.g., \cite{Lehmann}, Chapter 15.
More specifically, 
given data,
we can independently sample $E^{(i)}$ and compute 
$
T_{\mathrm{CQ}}^{(i)} = 
T_{\mathrm{CQ}}(
E^{(i)};
\tilde \BX_1, \tilde \BX_2
)
$, $i = 1, \ldots, B$, where $B$ is a sufficiently large number.
Then the null hypothesis is rejected if
\begin{align*}
    \frac{1}{B+1} \left[
        1 + \sum_{i=1}^B \mathbf 1_{ \{T_{\mathrm{CQ}}^{(i)} \geq T_{\mathrm{CQ}}(\BX_1, \BX_2)\} }
    \right] \leq \alpha
        .
\end{align*}
In the above procedure, one needs to compute the original statistic
$
T_{\mathrm{CQ}}(
\BX_1, \BX_2
)
$
and
the randomized statistics $T_{\mathrm{CQ}}^{(1)}, \ldots, T_{\mathrm{CQ}}^{(B)}$.
The direct computation of the original statistic
$
T_{\mathrm{CQ}}(
\BX_1, \BX_2
)
$
costs $O(n^2 p)$ time.
During the computation of 
$
T_{\mathrm{CQ}}(
\BX_1, \BX_2
)
$, we can cache the inner products $\tilde X_{k,i}^\myT \tilde X_{k,j}$, $1\leq i < j \leq m_k$, $k = 1,2$, and $\tilde X_{1,i}^\myT \tilde X_{2,j}$, $i = 1, \ldots, m_1$, $j = 1, \ldots, m_2$.
Then the computation of $T_{\mathrm{CQ}}^{(i)}$ only requires $O(n^2)$ time.
In total, the computation of the proposed test procedure can be completed within $O\{n^2 (p + B)\}$ time.

%In this section, we propose a new randomization method to construct a valid test based on $T_{\mathrm{CQ}}$.
%The new test has correct test level asymptotically for any $\bSigma_1$ and $\bSigma_2$.
%Let $\epsilon_{1,1}, \ldots, \epsilon_{1, \lfloor n_1/2 \rfloor} $ and $ \epsilon_{2,1}, \ldots, \epsilon_{2, \lfloor n_2/2 \rfloor }$

%    For a random vector $Z \in \mathbb R^m$, let $\mathcal L(Z)$ denote its distribution, a probability measure on $\mathbb R^m$.

%For a sequence of random variables $\{Z_n\}_{n=1}^{\infty}$ and a random variable $Z$,
%we write $Z_n \rightsquigarrow Z$ if $Z_n$ converges weakly to $Z$.
%\subsection{Distributional behavior of CQ statistics}
Now we rigorously investigate the theoretical properties of the proposed test procedure.
From Theorem \ref{thm:universality_TCQ}, 
the distribution of $T_{\mathrm{CQ}}(\BX_1, \BX_2)$ is asymptotically equivalent to 
the distribution of a quadratic form in normal random variables.
Now we show that the conditional distribution
$
    \mathcal L 
    \{
        T_{\mathrm{CQ}}( E ; \tilde \BX_1, \tilde \BX_2 ) 
        \mid \tilde \BX_1, \tilde \BX_2
    \}
    $
is equivalent to the same quadratic form.
In fact, our result is more general and includes the case that
    %The proposed method is motivated by the randomization test.
    %Hence we use Rademacher random variables $\epsilon_{k,i}$ to generate randomized statistics $T_{\mathrm{CQ}}(E; \tilde \BX_1, \tilde \BX_2)$.
    %Theorem \ref{thm:final_thm} also includes the case that the  randomized statistics are generated by normally distributed variables $\epsilon_{k,i}^*$, which is a popular strategy for wild bootstrap methods.
    %For some other resampling methods, such as the wild bootstrap method, normal random variables are commonly used to generate the randomized statistic.
    the elements of $E$ are generated from the standard normal distribution.

%derive the asymptotic level and power of the proposed test procedure.
%Now we rigorously investigate the properties of the proposed test.
\begin{theorem}\label{thm:final_thm}
    %Suppose Assumptions \ref{assumption:n}, \ref{assumption:wangwang} and \ref{assumption7} hold, and $\sigma_{T,n}^2 > 0$ for all $n$.
Suppose the conditions of Theorem \ref{thm:universality_TCQ} hold.
Let $E^* = (\epsilon^*_{1,1}, \ldots, \epsilon^*_{1, m_1}, \epsilon^*_{2,1}, \ldots, \epsilon^*_{2, m_2} )^\myT$,
where $ \{\epsilon^*_{k,i}\} $ are independent and identically distributed random variables, and $\epsilon^*_{1,1}$ is a standard normal random variable or a Rademacher random variable.
    Then as $n \to \infty$,
\begin{align*}
    \left\|
    \mathcal L \left\{
        \frac{T_{\mathrm{CQ}}( E^* ; \tilde \BX_1, \tilde \BX_2 ) }{
            \sigma_{T,n}
        }
        \mid \tilde \BX_1, \tilde \BX_2
    \right\}
    -
    \mathcal L \left[
    \frac{
        \bxi_p^\myT
        \bPsi_n
    \bxi_p
    -\mytr
        \left(
            \bPsi_n
    \right)
        }{
        \left\{
            2
            \mytr(
            \bPsi_n^2
            )
        \right\}^{1/2}
    }
\right]
    \right\|_3
    \xrightarrow{P} 0 .
\end{align*}
\end{theorem}

From Theorems \ref{thm:universality_TCQ} and \ref{thm:final_thm},
the conditional distribution 
    $\mathcal L \{
        T_{\mathrm{CQ}}( E ; \tilde \BX_1, \tilde \BX_2 ) 
        \mid \tilde \BX_1, \tilde \BX_2
    \}
    $
is asymptotically equivalent to the distribution of $T_{\mathrm{CQ}}(\BX_1, \BX_2)$ under the null hypothesis.
These two theorems allow us to derive the asymptotic level and local power of the proposed test procedure.
Let $G_n(\cdot)$ denote the cumulative distribution function of 
$
    \{
        \bxi_p^\myT
        \bPsi_n
    \bxi_p
    -\mytr
        (
            \bPsi_n
    )
    \}/
        \{
            2
            \mytr(
            \bPsi_n^2
            )
        \}^{1/2}
$.

\begin{corollary}\label{corollary:the2}
    %Suppose Assumptions \ref{assumption:n}, \ref{assumption:wangwang} and \ref{assumption7} hold, and $\sigma_{T,n}^2 > 0$ for all $n$.
Suppose the conditions of Theorem \ref{thm:universality_TCQ} hold,
    $\alpha \in (0, 1)$ is an absolute constant and as $n \to \infty$,
    $
     (\mu_1 - \mu_2)^\myT \bPsi_n (\mu_1 - \mu_2)
     =o\{
    \mytr(\bPsi_n^2)
\}
$.
    Then as $n \to \infty$,
    \begin{align*}
    \Pr \left\{
    T_{\mathrm{CQ}} (\BX_1, \BX_2) > \hat F_{\mathrm{CQ}}^{-1} (1-\alpha)
\right\}
=
1- G_n
\left[
    G_n^{-1} (1-\alpha) - 
    \frac{
    \|\mu_1 - \mu_2\|^2
    }{
        \left\{ 2 \mytr (\bPsi_n^2) \right\}^{1/2}
}
\right]
+o(1)
.
    \end{align*}

\end{corollary}
Corollary \ref{corollary:the2} implies that
under the conditions of Theorem \ref{thm:universality_TCQ},
the proposed test procedure has correct test level asymptotically.
In particular, Corollary \ref{corollary:the2} provides a rigorous theoretical guarantee of the validity of the proposed test procedure for two-sample Behrens-Fisher problem with arbitrary $\bar \bSigma_k$, $i= 1,2$.
Furthermore, the proposed test procedure is still valid when the observations are not identically distributed within groups and the sample sizes are unbalanced.
To the best of our knowledge, the proposed test procedure is the only one that is guaranteed to be valid in such a general setting.
Corollary \ref{corollary:the2} also gives the asymptotic power of the proposed test procedure under the local alternative hypotheses.
If $T_{\mathrm{CQ}}(\BX_1, \BX_2)$ is asymptotically normally distributed,
that is,
$G_n$ converges weakly to the cumulative distribution function of the standard normal distribution,
then the proposed test procedure has the same local asymptotic power as the test procedure of \cite{Chen2010ATwo}.
In general, the proposed test procedure has the same local asymptotic power as the oracle test procedure which rejects the null hypothesis when $T_{\mathrm{CQ}}(\BX_1, \BX_2)$ is greater than the $1-\alpha$ quantile of $\mathcal L \{T_{\mathrm{CQ}} (\BX_1, \BX_2) \}$.
Thus, the proposed test procedure has good power behavior.
%To the best of our knowledge, such a theoretical 

%\section{Experimental results}

{
\section{Simulations}
In this section, we conduct simulations to examine the performance of the proposed test procedure and compare it with $9$ alternative test procedures.
%We consider $5$ competitive test procedures.
The first $4$ competing test procedures are based on $T_{\mathrm{CQ}}(\BX_1, \BX_2)$, 
%but they use different strategies to control the test level.
%The first one is 
including
the original test procedure of \cite{Chen2010ATwo} which is based on the asymptotic normality of $T_{\mathrm{CQ}}(\BX_1, \BX_2)$,
the test procedure based on the empirical bootstrap method,
%The third one is 
the wild bootstrap method described in Section \ref{sec:test_procedure} where $\{\varepsilon_{k,i}\}$ are Rademacher random variables,
and the $\chi^2$-approximation method in \cite{Zhang2020AFurtherStudy}.
%Rademacher distribution is a popular choice for the distribution of $\varepsilon_{k,i}$.
%It can be seen that for Rademacher random variable $\varepsilon_{k,i}$,
%the required conditions $\myE (\varepsilon_{k,i}) = 0$ and $\myVar (\varepsilon_{k,i}) = 1$ are satisfied.
The next $2$ competing test procedures are based on the statistic $\|\bar X_1 - \bar X_2\|^2$, including the $\chi^2$-approximation method in \cite{Zhang2021Two-sample} and 
the half-sampling method in \cite{Lou2020HighDimensional}.
%The fourth one is the half-sampling method proposed in \cite{Lou2020HighDimensional}.
The last $3$ competing test procedures are scalar-invariant tests of \cite{Srivastava2008A}, \cite{SRIVASTAVA2013349} and \cite{Feng2014Two}.
}
%In addition to the above four test procedures, we also include the test procedure of \cite{Zhang2021Two-sample}
%which uses Welch-Satterthwaite $\chi^2$-approximation method to determine the critical value of {\color{red}$T_{\mathrm{BS}}(\BX_1, \BX_2)$}.

In our simulations, the nominal test level is $\alpha = 0.05$.
For the proposed method and competing resampling methods, the resampling number is $B = 1,000$.
The reported empirical sizes and powers are computed based on $10,000$ independent replications.
We consider the following data generation models for $\{Y_{k,i}\}$.
\begin{itemize}
    \item
Model I: $Y_{k,i} \sim \mathcal N ( \mathbf 0_p, \BI_p )$, $i =1 ,\ldots, n_k$, $k = 1, 2$.
\item
    {
Model II:
$Y_{k,i} \sim \mathcal N ( \mathbf 0_p, \bar \bSigma_k )$, $i =1 ,\ldots, n_k$, $k = 1, 2$, where 
%are randomly generated from $\mathcal N ( \mathbf 0_p, \BI_p )$.
$
    \bar \bSigma_k
    =
    \BV_k \bLambda \BV_k^\myT + \BI_p
$ with
\begin{align*}
    \bLambda = 
    \begin{pmatrix}
        1 & 0
        \\
        0 & 1/2
    \end{pmatrix}
    ,
    \quad
    \BV_1
    =
    \begin{pmatrix}
\mathbf 1_{p / 4}
&
\mathbf 1_{p / 4}
\\
\mathbf 1_{p / 4}
&
-
\mathbf 1_{p / 4}
\\
\mathbf 1_{p / 4}
&
\mathbf 1_{p / 4}
\\
\mathbf 1_{p / 4}
&
-\mathbf 1_{p / 4}
    \end{pmatrix}
    ,
    \quad
    \BV_2
    =
    \begin{pmatrix}
\mathbf 1_{p / 4}
&
\mathbf 1_{p / 4}
\\
\mathbf 1_{p / 4}
&
-
\mathbf 1_{p / 4}
\\
-\mathbf 1_{p / 4}
&
-\mathbf 1_{p / 4}
\\
-\mathbf 1_{p / 4}
&
\mathbf 1_{p / 4}
    \end{pmatrix}
    .
\end{align*}

\item
Model III:
%the $j$th element of $Y_{k,i}$ is $(i 1.1^{j})^{1/2} z_{k,i,j}$,
%%$Y_{k,i} = i^{1/2} Z_{k,i}$,
%$j = 1, \ldots, p$,
%$i = 1, \ldots, n_k$, $k = 1, 2$,
%where the random variables $\{z_{k,i,j}\}$ are independent with Student's $t$-distribution with $4$ degrees of freedom.
$Y_{k,i} = \bGamma_{k,i} Z_{k,i}$, $i =1 ,\ldots, n_k$, $k = 1, 2$, where $Z_{k,i} = (z_{k,i,1}, \ldots, z_{k,i,p})^\myT$,
and $\{z_{k,i,j}\}$ are independent standardized $\chi^2$ random variables with $1$ degree of freedom, that is, $z_{k,i,1} \sim \{ \chi^2(1) - 1 \} / \surd {2}$.
For $ i = 1, \ldots, n_k / 2 $, $\bGamma_{k,i} = \{k \mydiag(1, 2, \ldots, p) \}^{1/2}$,
and for $i = n_k / 2 + 1, \ldots, n_k $,
$\bGamma_{k,i} = \{ k \mydiag(p, p-1, \ldots, 1) \}^{1/2} $.
\item
Model IV:
The $j$th element of $Y_{k,i}$ is
$y_{k,i,j} = \sum_{\ell=0}^5 1.01^{j+\ell - 1} z_{k,i,j + \ell}$,
$j=1, \ldots, p$,
$i = 1, \ldots, n_k$, $k = 1, 2$,
where
$\{z_{k,i,j}\}$ are independent standardized $\chi^2$ random variables with $1$ degree of freedom.
}
\end{itemize}

{
For Model I,
observations are simply normal random vectors with identity covariance matrix.
For Model II,
the variables are correlated,
$\bar \bSigma_1 \neq \bar \bSigma_2$ and
the condition \eqref{eq:condition_chen} is not satisfied.
%For Model III,
%$\bSigma_{k,1}, \ldots, \bSigma_{k,n_k}$ are not equal, the condition \eqref{eq:condition_chen} is not satisfied
%and the observations have heavier tail than normal distribution.
For Model III,
$\bSigma_{k,1}, \ldots, \bSigma_{k,n_k}$ are not equal and the observations have skewed distributions.
For Model IV,
the variables are correlated and their variances are different.
%skewed distributions.
}

\begin{table}[t]
    \small
\def~{\hphantom{0}} 
{
    %\scriptsize
    \caption{Empirical sizes (multiplied by $100$) of test procedures
    }
    \label{table1}
    \begin{center}
    \begin{tabular}{*{14}{c}}
    %\toprule
    Model &
$n_1$ & $n_2$ & $p$  & NEW & CQ & EB & WB & ZZ & ZZGZ & LOU & SD & SKK & FZWZ
 \\
    %\midrule
    %\multicolumn{3}{l}{Model I}
%\\
I&
8 & 12 & 300
 & 4.95 & 5.42 & 0.00 & 1.32 & 4.92 & 5.46 & 6.13 & 3.77 & 32.8 & 3.20 \\ 
 &
16 & 24 & 300
 & 5.19 & 5.94 & 0.00 & 3.88 & 5.41 & 5.72 & 5.98 & 4.37 & 14.4 & 5.08 \\ 
 &
32 & 48 & 300
 & 5.00 & 5.65 & 0.00 & 4.79 & 5.21 & 5.33 & 5.42 & 4.21 & 8.36 & 5.11 \\ 
 &
8 & 12 & 600
 & 5.07 & 5.52 & 0.00 & 0.45 & 4.99 & 5.73 & 6.22 & 2.62 & 48.2 & 2.55 \\ 
 &
16 & 24 & 600
 & 4.85 & 5.44 & 0.00 & 2.25 & 5.07 & 5.40 & 5.57 & 3.05 & 17.5 & 4.21 \\ 
 &
32 & 48 & 600
 & 5.23 & 5.50 & 0.00 & 4.22 & 5.21 & 5.35 & 5.40 & 3.94 & 9.48 & 5.14 \\ 
    %\midrule
    %\multicolumn{3}{l}{Model II}
%\\
II &
8 & 12 & 300
 & 8.99 & 10.1 & 5.91 & 8.56 & 8.23 & 9.87 & 9.14 & 3.31 & 5.94 & 7.78 \\ 
 &
16 & 24 & 300
 & 6.65 & 8.15 & 5.54 & 6.70 & 6.41 & 7.39 & 6.70 & 2.12 & 3.52 & 7.37 \\ 
 &
32 & 48 & 300
 & 5.46 & 7.45 & 5.14 & 5.56 & 5.34 & 6.25 & 5.57 & 1.85 & 2.52 & 7.04 \\ 
 &
8 & 12 & 600
 & 9.02 & 10.4 & 6.10 & 8.80 & 8.44 & 9.85 & 9.48 & 2.64 & 4.77 & 8.02 \\ 
 &
16 & 24 & 600
 & 6.10 & 7.83 & 4.79 & 5.95 & 5.60 & 6.84 & 6.06 & 1.64 & 2.60 & 6.72 \\ 
 &
32 & 48 & 600
 & 5.27 & 7.09 & 4.71 & 5.20 & 5.07 & 5.87 & 5.20 & 1.18 & 1.78 & 6.65 \\ 
    %\midrule
    %\multicolumn{3}{l}{Model III}
%\\
III&
8 & 12 & 300
 & 5.23 & 5.50 & 0.00 & 1.30 & 1.63 & 2.18 & 6.14 & 0.17 & 14.1 & 0.00 \\ 
 &
16 & 24 & 300
 & 5.31 & 5.63 & 0.00 & 3.78 & 2.89 & 3.22 & 5.71 & 0.22 & 8.30 & 0.00 \\ 
 &
32 & 48 & 300
 & 5.56 & 6.09 & 0.01 & 5.10 & 3.94 & 4.16 & 5.81 & 0.16 & 6.79 & 0.13 \\ 
 &
8 & 12 & 600
 & 5.58 & 5.77 & 0.00 & 0.58 & 2.04 & 2.46 & 6.38 & 0.00 & 19.5 & 0.00 \\ 
 &
16 & 24 & 600
 & 5.01 & 5.28 & 0.00 & 2.39 & 2.64 & 2.94 & 5.46 & 0.00 & 10.5 & 0.00 \\ 
 &
32 & 48 & 600
 & 5.00 & 5.32 & 0.00 & 4.20 & 3.64 & 3.82 & 5.12 & 0.02 & 7.11 & 0.03 \\ 
    %\midrule
    %\multicolumn{3}{l}{Model IV}
%\\
IV&
8 & 12 & 300
 & 5.90 & 7.27 & 1.33 & 5.92 & 5.22 & 6.15 & 6.64 & 3.05 & 11.2 & 2.05 \\ 
 &
16 & 24 & 300
 & 5.21 & 6.73 & 2.29 & 5.49 & 4.63 & 5.35 & 5.48 & 3.25 & 6.90 & 3.53 \\ 
 &
32 & 48 & 300
 & 5.03 & 6.76 & 3.17 & 5.58 & 4.92 & 5.50 & 5.38 & 4.02 & 5.82 & 5.21 \\ 
 &
8 & 12 & 600
 & 5.97 & 7.49 & 1.29 & 6.03 & 5.20 & 6.17 & 6.75 & 2.24 & 14.4 & 1.48 \\ 
 &
16 & 24 & 600
 & 5.60 & 7.19 & 2.37 & 5.99 & 5.06 & 5.91 & 5.94 & 2.52 & 7.62 & 3.13 \\ 
 &
32 & 48 & 600
 & 5.50 & 7.11 & 3.59 & 5.90 & 5.16 & 5.90 & 5.85 & 3.66 & 5.86 & 4.46 \\ 
%\bottomrule
\end{tabular}
    \end{center}
    NEW, the proposed test procedure;
    CQ, the test of \cite{Chen2010ATwo}; % based on the asymptotic normality of $T_{\mathrm{CQ}}(\BX_1, \BX_2)$;
    EB, the empirical bootstrap method based on $T_{\mathrm{CQ}}(\BX_1, \BX_2)$;
    WB, the wild bootstrap method based on $T_{\mathrm{CQ}}(\BX_1, \BX_2)$;
    ZZ, the test of \cite{Zhang2020AFurtherStudy};
    ZZGZ, the test of \cite{Zhang2021Two-sample}; % based on Welch-Satterthwaite $\chi^2$-approximation.
    %LOU, the half-sampling method proposed in \cite{Lou2020HighDimensional}.
    LOU, the test of \cite{Lou2020HighDimensional};
    SD, the test of \cite{Srivastava2008A}; 
    SKK, the test of \cite{SRIVASTAVA2013349};
    FZWZ, the test of \cite{Feng2014Two}.
}
\end{table}

{
    In Section \ref{sec:anr},
%To verify the correctness of Theorem \ref{thm:universality_TCQ}and Corollary \ref{corollary:the},
 we give quantile-quantile plots to examine the correctness of Theorem \ref{thm:universality_TCQ} and Corollary \ref{corollary:the} of the proposed test statistic.
The results show that the distribution approximation in Theorem \ref{thm:universality_TCQ}
is quite accurate,
and the asymptotic distributions in Corollary \ref{corollary:the} are reasonable even for finite sample size.
}

Now we consider the simulations of empirical sizes.
We take $\mu_1 = \mu_2 = \mathbf 0_p$.
Table \ref{table1} lists the empirical sizes of various test procedures.
%Also, the empirical size of the proposed test procedure tends to the nominal test level as the sample size increases.
{
It can be seen that the test procedure of \cite{Chen2010ATwo}
tends to have inflated empirical sizes, especially for Models II and IV.
The empirical bootstrap method does not work well for Models I, III and IV.
While the wild bootstrap method has a better performance than the empirical bootstrap method, it is overly conservative for Models I and III when $n$ is small.
The performance of bootstrap methods confirm our heuristic arguments in Section \ref{sec:test_procedure}.
It is interesting that the empirical bootstrap method has relatively good performance for Model II.
In fact, in Model II, $\bar \bSigma_k$ has a low-rank structure,
and therefore, the test statistic $T_{\mathrm{CQ}}(\BX_1, \BX_2)$
may have a similar behavior as in the low-dimensional setting.
It is known that the empirical bootstrap method has good performance in the low-dimensional setting.
Hence it is reasonable that the empirical bootstrap method works well for Model II.
The $\chi^2$-approximation methods of \cite{Zhang2021Two-sample} and \cite{Zhang2020AFurtherStudy} have reasonable performance under Model I, II and IV, but are overly conservative for Model III.
The half-sampling method of \cite{Lou2020HighDimensional} has inflated empirical sizes, especially for small $n$.
%On the other hand, the test procedure of \cite{Zhang2021Two-sample} is overly conservative for Model III and Model IV.
Compared with the test procedures based on $\|\bar \BX_1 - \bar \BX_2\|^2$ or $T_{\mathrm{CQ}}(\BX_1, \BX_2)$, the scalar-invariant tests of \cite{Srivastava2008A}, \cite{SRIVASTAVA2013349} and \cite{Feng2014Two} have relatively poor performance, especially when $n$ is small.
For Models I, III and IV, the empirical sizes of the proposed test procedure are quite close to the nominal test level.
For Model II, all test procedures based on $\|\bar \BX_1 - \bar \BX_2\|^2$ or $T_{\mathrm{CQ}}(\BX_1, \BX_2)$ have inflated empirical sizes.
For this model, the proposed test procedure outperforms the test procedure of \cite{Chen2010ATwo}.
%and the proposed test procedure has a relatively good performance.
%For this model, the proposed test procedure has a similar performance as the half-sampling method, and they outperform the test procedure of \cite{Chen2010ATwo} and \cite{Zhang2021Two-sample}.
Also, as $n$ increases, the empirical sizes of the proposed test procedure tends to the nominal test level.
For $n_1 = 32$, $n_2 = 48$, the proposed test procedure has reasonable performance in all settings.
Overall, the proposed test procedure has reasonably good performance in terms of empirical sizes, which confirms our theoretical results.
}

\begin{table}[h]
    \small
    {
    %\scriptsize
    \caption{Empirical powers (multiplied by $100$) of test procedures for $p = 300$
    }
    \label{table2}
    \begin{center}
    \begin{tabular}{*{14}{c}}
    %\toprule
    Model &
$n_1$ & $n_2$ & $\beta$  & NEW & CQ & EB & WB & ZZ & ZZGZ & LOU & SD & SKK & FZWZ
\\
%\midrule
%\multicolumn{3}{l}{Model I}
%\\
I&
8& 12 & 1
 & 24.4 & 26.4 & 0.00 & 10.6 & 23.5 & 26.7 & 28.0 & 19.9 & 65.9 & 16.7 \\ 
 &
16& 24 & 1
 & 25.3 & 27.1 & 0.01 & 20.7 & 26.0 & 26.7 & 27.2 & 21.5 & 43.7 & 23.2 \\ 
 &
32& 48 & 1
 & 25.3 & 27.2 & 0.37 & 24.2 & 25.8 & 26.2 & 26.5 & 22.8 & 33.7 & 25.6 \\ 
 &
8& 12 & 2
 & 55.5 & 59.1 & 0.0 & 35.1 & 53.8 & 59.6 & 61.0 & 48.2 & 88.7 & 43.4 \\ 
 &
16& 24 & 2
 & 57.7 & 60.5 & 0.22 & 51.5 & 58.9 & 59.9 & 60.7 & 52.1 & 75.5 & 54.5 \\ 
 &
32& 48 & 2
 & 58.5 & 60.5 & 3.88 & 57.3 & 58.8 & 59.3 & 59.6 & 55.5 & 66.3 & 58.2 \\ 
 %\midrule
%\multicolumn{3}{l}{Model II}
 %\\
II &
8& 12 & 1
 & 28.1 & 31.7 & 24.0 & 28.3 & 27.8 & 31.0 & 29.5 & 16.4 & 21.8 & 26.4 \\ 
 &
16& 24 & 1
 & 25.5 & 30.0 & 24.2 & 26.6 & 26.0 & 28.0 & 26.6 & 14.9 & 17.9 & 27.8 \\ 
 &
32& 48 & 1
 & 24.5 & 29.3 & 23.7 & 25.0 & 24.5 & 26.7 & 25.3 & 13.7 & 16.1 & 28.4 \\ 
 &
8& 12 & 2
 & 46.7 & 50.9 & 42.4 & 47.5 & 46.8 & 50.3 & 48.7 & 31.8 & 38.7 & 45.0 \\ 
 &
16& 24 & 2
 & 44.3 & 50.6 & 43.1 & 46.2 & 45.4 & 48.5 & 46.5 & 29.6 & 34.6 & 48.1 \\ 
 &
32& 48 & 2
 & 44.0 & 50.0 & 43.1 & 44.6 & 44.2 & 46.6 & 44.9 & 29.4 & 32.9 & 48.8 \\ 
 %\midrule
%\multicolumn{3}{l}{Model III}
%\\
III &
8& 12 & 1
 & 24.8 & 25.9 & 0.00 & 11.1 & 12.7 & 14.7 & 27.3 & 0.01 & 34.0 & 0.00 \\ 
 &
16& 24 & 1
 & 25.5 & 27.2 & 0.02 & 20.8 & 17.5 & 18.7 & 27.1 & 0.00 & 28.2 & 0.01 \\ 
 &
32& 48 & 1
 & 25.1 & 26.8 & 0.49 & 24.2 & 20.7 & 21.3 & 26.0 & 0.03 & 24.4 & 1.22 \\ 
 &
8& 12 & 2
 & 57.9 & 60.7 & 0.00 & 36.9 & 39.4 & 43.9 & 62.0 & 0.00 & 80.1 & 0.00 \\ 
 &
16& 24 & 2
 & 57.9 & 60.2 & 0.29 & 52.5 & 47.8 & 49.6 & 60.0 & 0.10 & 69.3 & 0.87 \\ 
 &
32& 48 & 2
 & 59.7 & 61.7 & 5.37 & 58.6 & 54.0 & 54.8 & 60.6 & 0.36 & 63.2 & 12.4 \\ 
 %\midrule
%\multicolumn{3}{l}{Model IV}
%\\
IV &
8& 12 & 1
 & 22.3 & 26.9 & 7.83 & 22.6 & 20.6 & 23.3 & 25.1 & 100 & 100 & 100 \\ 
 &
16& 24 & 1
 & 21.4 & 25.9 & 12.2 & 22.8 & 20.2 & 22.3 & 22.7 & 100 & 100 & 100 \\ 
 &
32& 48 & 1
 & 21.3 & 26.4 & 15.7 & 22.5 & 20.8 & 22.6 & 22.2 & 100 & 100 & 100 \\ 
 &
8& 12 & 2
 & 49.9 & 57.0 & 24.2 & 51.2 & 47.7 & 51.6 & 54.3 & 100 & 100 & 100 \\ 
 &
16& 24 & 2
 & 49.5 & 56.5 & 33.1 & 51.7 & 47.7 & 51.2 & 51.9 & 100 & 100 & 100 \\ 
 &
32& 48 & 2
 & 49.4 & 56.5 & 39.6 & 51.1 & 48.4 & 51.0 & 50.5 & 100 & 100 & 100 \\ 
%\bottomrule
\end{tabular}
\end{center}
    NEW, the proposed test procedure;
    CQ, the test of \cite{Chen2010ATwo}; % based on the asymptotic normality of $T_{\mathrm{CQ}}(\BX_1, \BX_2)$;
    EB, the empirical bootstrap method based on $T_{\mathrm{CQ}}(\BX_1, \BX_2)$;
    WB, the wild bootstrap method based on $T_{\mathrm{CQ}}(\BX_1, \BX_2)$;
    ZZ, the test of \cite{Zhang2020AFurtherStudy};
    ZZGZ, the test of \cite{Zhang2021Two-sample}; % based on Welch-Satterthwaite $\chi^2$-approximation.
    %LOU, the half-sampling method proposed in \cite{Lou2020HighDimensional}.
    LOU, the test of \cite{Lou2020HighDimensional};
    SD, the test of \cite{Srivastava2008A}; 
    SKK, the test of \cite{SRIVASTAVA2013349};
    FZWZ, the test of \cite{Feng2014Two}.
}
\end{table}

Now we consider the empirical powers of various test procedures.
In view of the expression of the asymptotic power given in Corollary \ref{corollary:the2}, 
we define the signal-to-noise ratio $\beta = \|\mu_1 - \mu_2\|^2 / \{2\mytr(\bPsi_n^2)\}^{1/2}$.
We take $\mu_1 = \mathbf 0_p$ and $\mu_2 = c \mathbf 1_p$ where $c$ is chosen such that $\beta$ reaches given values of signal-to-noise ratio.
Table \ref{table2} lists the empirical powers of various test procedures when $p = 300$.
{
%It can be seen that for Model IV, the scalar-invariant tests of \cite{Srivastava2008A}, \cite{SRIVASTAVA2013349} and \cite{Feng2014Two} are very powerful.
%In fact, for Model IV, the variances of variables are largely different.
%Hence it is reasonable that scalar-invariant tests have better performance than other test procedures in this setting.
%We have seen that some competing test procedures do not have a good control of test level.
%To get rid of the effect of distorted test level,
 %in the Supplementary Material,
%we present the receiver operating characteristic curves of the test procedures.
    It can be seen that for Model IV where the variables have different variance scale,
the scalar-invariant tests have better performance than other tests.
%show great advantage over the test methods based on $\|\bar \BX_1 - \bar \BX_2\|^2$ and $T_{\mathrm{CQ}}(\BX_1, \BX_2)$.
However, for Model III where $\bSigma_{k,1}, \ldots, \bSigma_{k, n_k}$ are not identical and the observations have skewed distributions,
the scalar-invariant tests have relatively low powers.
%somewhat smaller than that of the test methods based on $\|\bar \BX_1 - \bar \BX_2\|^2$ and $T_{\mathrm{CQ}}(\BX_1, \BX_2)$.
The proposed test procedure has a reasonable power behavior
among the test procedures based on $\|\bar \BX_1 - \bar \BX_2\|^2$ or $T_{\mathrm{CQ}}(\BX_1, \BX_2)$.
%although it is sometimes a little less powerful than
%the test procedures of \cite{Chen2010ATwo} and \cite{Lou2020HighDimensional}.
%However, the power gain of these two test procedures may largely due to the inflated test sizes.
We have seen that some competing tests do not have a good control of test level.
%Some competing tests
To get rid of the effect of distorted test level on the power, we present the receiver operating characteristic curves of the test procedures in Section \ref{sec:anr}.
It shows that for a given test level, all test procedures based on $\|\bar \BX_1 - \bar \BX_2\|^2$ or $T_{\mathrm{CQ}}(\BX_1, \BX_2)$ have quite similar power behavior.
%Figure \ref{} illustrates the receiver operating characteristic curve of various test procedures when $n_1=$, $n_2 = $ and $p = 300$.
In summary,
the proposed test procedure has promising performance of empirical sizes, and has no power loss compared with existing tests based on
$\|\bar \BX_1 - \bar \BX_2\|^2$ or $T_{\mathrm{CQ}}(\BX_1, \BX_2)$.
}

\begin{table}[t]
    \small
    {
    %\scriptsize
    \caption{$p$-values of test procedures for the real data
    }
    \label{table_real}
    \begin{center}
    \begin{tabular}{*{14}{c}}
    %\toprule
NEW & CQ & EB & WB & ZZ & ZZGZ & LOU & SD & SKK & FZWZ
\\
%\midrule
9.99e-4 &1.89e-10&9.99e-4&9.99e-4&7.74e-4&7.47e-5 & 0& 0.27 & 0.18& 4.78e-3
\\
%\bottomrule
\end{tabular}
\end{center}
    NEW, the proposed test procedure;
    CQ, the test of \cite{Chen2010ATwo}; % based on the asymptotic normality of $T_{\mathrm{CQ}}(\BX_1, \BX_2)$;
    EB, the empirical bootstrap method based on $T_{\mathrm{CQ}}(\BX_1, \BX_2)$;
    WB, the wild bootstrap method based on $T_{\mathrm{CQ}}(\BX_1, \BX_2)$;
    ZZ, the test of \cite{Zhang2020AFurtherStudy};
    ZZGZ, the test of \cite{Zhang2021Two-sample}; % based on Welch-Satterthwaite $\chi^2$-approximation.
    %LOU, the half-sampling method proposed in \cite{Lou2020HighDimensional}.
    LOU, the test of \cite{Lou2020HighDimensional};
    SD, the test of \cite{Srivastava2008A}; 
    SKK, the test of \cite{SRIVASTAVA2013349};
    FZWZ, the test of \cite{Feng2014Two}.
}
\end{table}

{
\section{Real-data example}
In this section, we apply the proposed test procedure to the gene expression dataset released by \cite{Alon1999BroadPatterns}.
This dataset consists of the gene expression levels of $n_1 = 22$ normal and $n_2 = 40$ tumor colon tissue samples.
It contains the expression of $p = 2,000$ genes with highest minimal intensity across the $n = 62$ tissues.
We would like to test if the normal and tumor colon tissue samples have the same average gene expression levels.
%This is an instance of the hypothesis testing problem considered in our paper.
Table \ref{table_real} lists the $p$-values of various test procedures.
With $\alpha = 0.05$,
all but the test procedures of \cite{Srivastava2008A} and \cite{SRIVASTAVA2013349} reject the null hypothesis,
claiming that the average gene expression levels of normal and tumor colon tissue samples are significantly different.
}

{
We would also like to examine the empirical sizes of various test procedures on the gene expression data.
To mimick the null distribution of the gene expression data, we generate resampled datasets as follows:
the resampled observations $\{X_{1,i}^*\}_{i=1}^{22}$ are uniformly sampled from $\{X_{1,i} - \bar X_1\}_{i=1}^{22}$ with replacement, and $\{X_{2,i}^*\}_{i=1}^{40}$ are uniformly sampled from $\{X_{2,i} - \bar X_2\}_{i=1}^{40}$ with replacement.
We conduct various test procedures with $\alpha = 0.05$ on the resampled observations $\{X_{k,i}^*\}$.
%where $n_1 = 11$, $n_2 = 20$ and $p = 2,000$.
The above procedure is independently replicated for $10,000$ times to compute the empirical sizes.
The results are listed in Table \ref{table_real_sizes}.
It can be seen that the test procedures of \cite{Srivastava2008A} and \cite{SRIVASTAVA2013349} are overly conservative.
Hence the $p$-values of these two test procedures for gene expression data may not be reliable.
The test procedures of \cite{Chen2010ATwo} and \cite{Feng2014Two} are a little inflated.
In comparison, the remaining test procedures, including the proposed test procedure, have a good control of the test level for the resampled gene expression data.
This implies that the $p$-value of the proposed test procedure for the gene expression data is reliable.
}

\begin{table}
    \small
    {
    %\scriptsize
    \caption{Empirical sizes (multiplied by $100$) of test procedures for the resampled real datasets
    }
    \label{table_real_sizes}

    \begin{center}
    \begin{tabular}{*{14}{c}}
    %\toprule
NEW & CQ & EB & WB & ZZ & ZZGZ & LOU & SD & SKK & FZWZ
\\
%\midrule
%9.99e-4 &1.89e-10&9.99e-4&9.99e-4&7.74e-4&7.47e-5 & 0& 0.27 & 0.18& 4.78e-3
5.43 & 7.28 & 5.22 & 5.84 & 5.19 & 5.98 & 5.71 & 0.60 & 0.79 & 7.03
%\\
%\bottomrule
\end{tabular}
    \end{center}
    NEW, the proposed test procedure;
    CQ, the test of \cite{Chen2010ATwo}; % based on the asymptotic normality of $T_{\mathrm{CQ}}(\BX_1, \BX_2)$;
    EB, the empirical bootstrap method based on $T_{\mathrm{CQ}}(\BX_1, \BX_2)$;
    WB, the wild bootstrap method based on $T_{\mathrm{CQ}}(\BX_1, \BX_2)$;
    ZZ, the test of \cite{Zhang2020AFurtherStudy};
    ZZGZ, the test of \cite{Zhang2021Two-sample}; % based on Welch-Satterthwaite $\chi^2$-approximation.
    %LOU, the half-sampling method proposed in \cite{Lou2020HighDimensional}.
    LOU, the test of \cite{Lou2020HighDimensional};
    SD, the test of \cite{Srivastava2008A}; 
    SKK, the test of \cite{SRIVASTAVA2013349};
    FZWZ, the test of \cite{Feng2014Two}.
}
\end{table}

%Affymetrix oligonucleotide arrays
\section*{Acknowledgements}
The authors thank the editor, associate editor and three reviewers for their valuable comments and suggestions.
This work was supported by Beijing Natural Science Foundation (No Z200001), National Natural Science Foundation of China (No 11971478).
Wangli Xu serves as the corresponding author of the present paper.

\bibliographystyle{apalike}
\bibliography{mybibfile}

\appendix

\setcounter{section}{0}
\setcounter{equation}{0}
\setcounter{theorem}{0}
\setcounter{lemma}{0}
\setcounter{assumption}{0}
\renewcommand\thesection{S.\arabic{section}}
\renewcommand\thesubsection{S.\arabic{subsection}}
\renewcommand\theequation{S.\arabic{equation}}
\renewcommand{\thetheorem}{S.\arabic{theorem}}
\renewcommand{\thelemma}{S.\arabic{lemma}}
\renewcommand{\theassumption}{S.\arabic{assumption}}

{\color{myColor}
\section{Additional numerical results}
\label{sec:anr}
In this section, we present additional numerical results.
The experimental setting is as described in the main text.

%First, we plot quantile-quantile plots to examine the correctness of Theorem \ref{thm:universality_TCQ} and Corollary \ref{corollary:the}.

We would like to use quantile-quantile plots to examine the correctness of Theorem \ref{thm:universality_TCQ} and Corollary \ref{corollary:the}.
First we consider the correctness of Theorem \ref{thm:universality_TCQ}.
Theorem \ref{thm:universality_TCQ} implies that the distribution of $T_{\mathrm{CQ}(\BY_1, \BY_2)} / \sigma_{T,n}$ can be approximated by that of $\{\bxi_p^\myT \bPsi_n \bxi_p - \mytr(\bPsi_n)\} / \{2 \mytr(\bPsi_n^2)\}^{1/2}$.
Fig. \ref{fig:QQ1} illustrates the plots of 
the empirical quantiles of $T_{\mathrm{CQ}(\BY_1, \BY_2)} / \sigma_{T,n}$ against that of $ \{ \bxi_p^\myT \bPsi_n \bxi_p - \mytr(\bPsi_n) \} / \{2 \mytr(\bPsi_n^2)\}^{1/2}$
under Models I-IV described in the main text with $n_1 = 16$, $n_2 = 24$, $p = 300$.
The empirical quantiles of $T_{\mathrm{CQ}}(\BY_1, \BY_2) / \sigma_{T,n}$ are obtained by $10,000$ replications.
The results imply that the distribution approximation in Theorem \ref{thm:universality_TCQ} is quite accurate for finite sample size.

\begin{figure}
    \centering
    \subfigure[Model I]{\includegraphics[width = 0.45\textwidth]{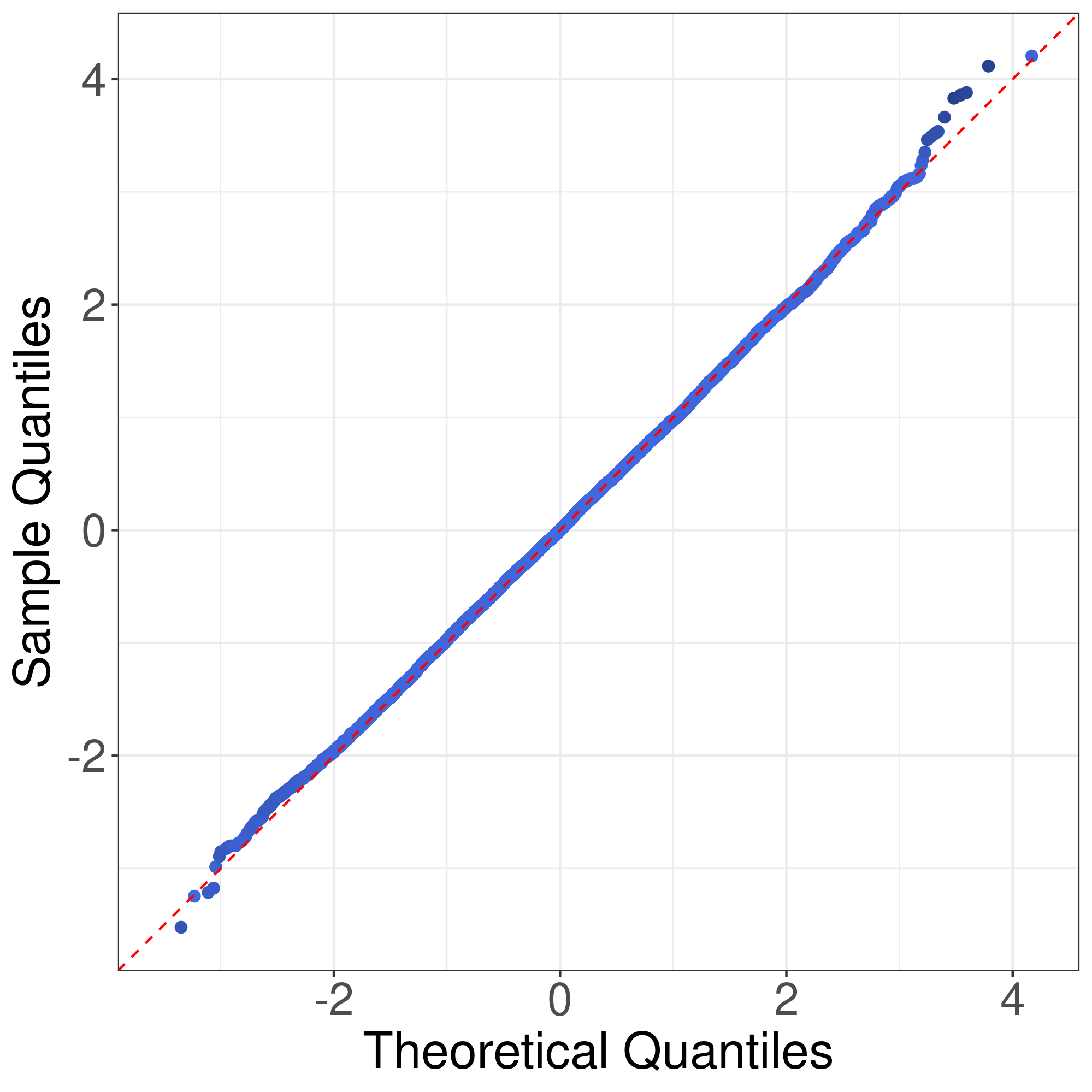}}
    \quad
    \subfigure[Model II]{\includegraphics[width = 0.45\textwidth]{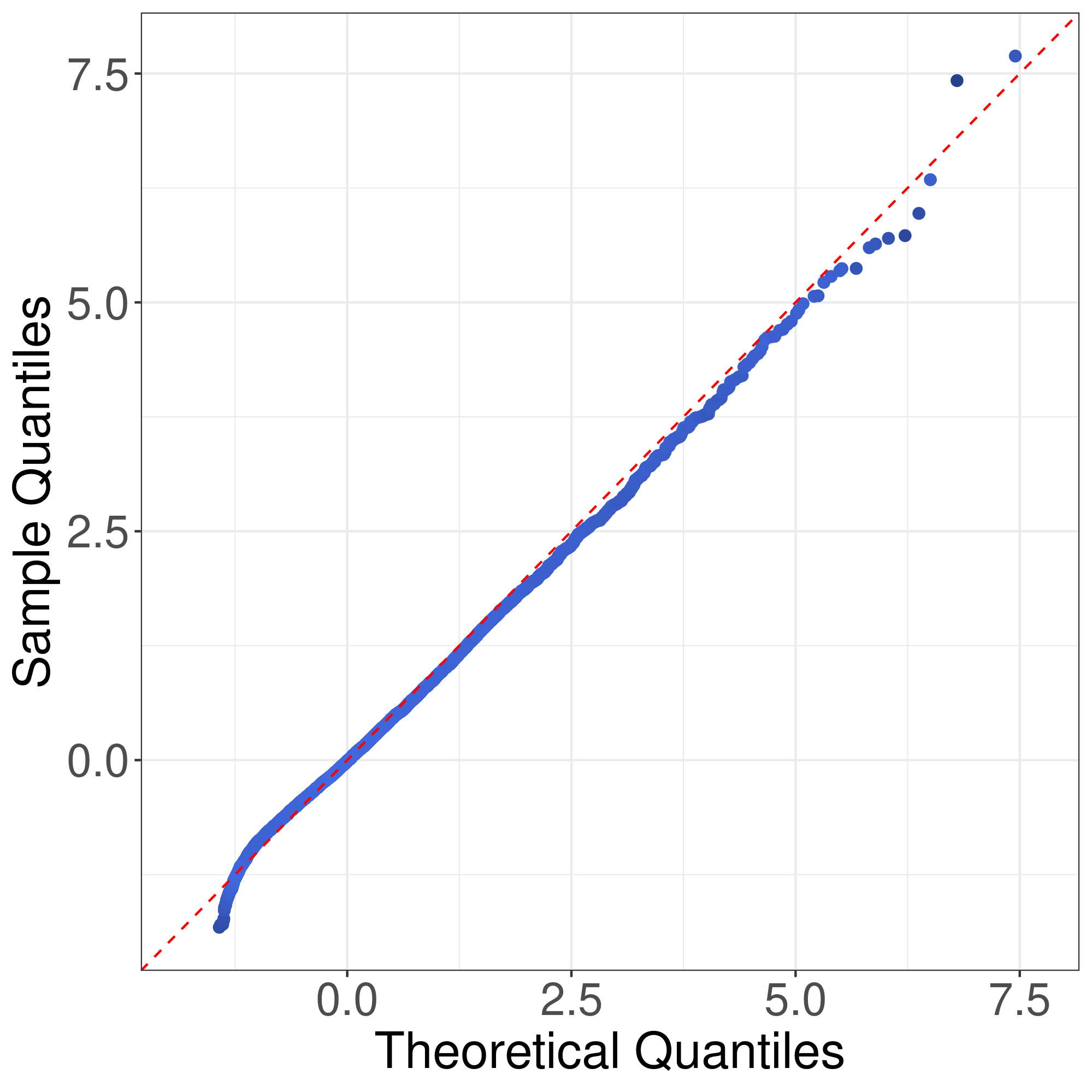}}
    \subfigure[Model III]{\includegraphics[width = 0.45\textwidth]{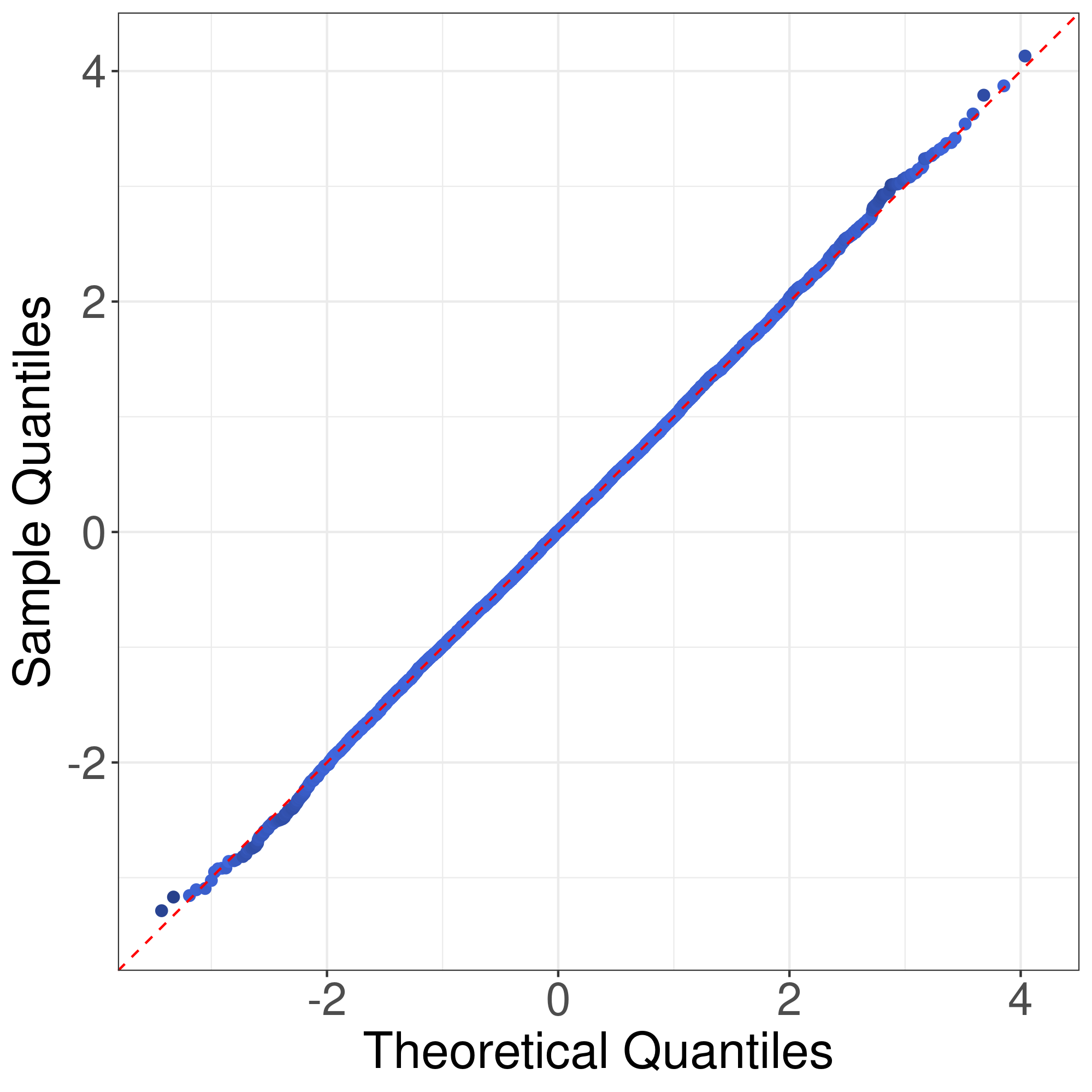}}
    \quad
    \subfigure[Model IV]{\includegraphics[width = 0.45\textwidth]{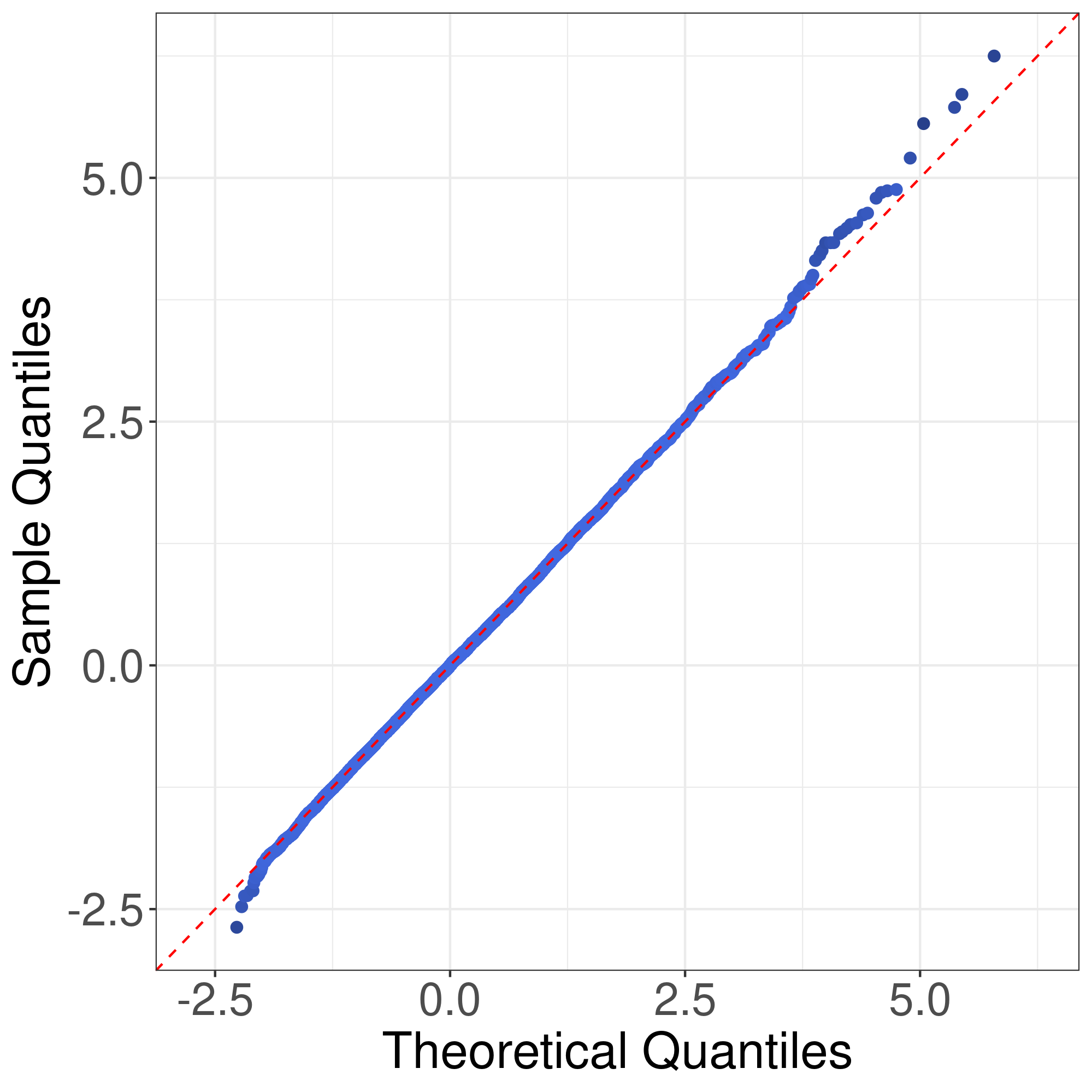}}
    \caption{
        Plots of the empirical quantiles of $T_{\mathrm{CQ}(\BY_1, \BY_2)} / \sigma_{T,n}$ against that of $(\bxi_p^\myT \bPsi_n \bxi_p - \mytr(\bPsi_n)) / \{2 \mytr(\bPsi_n^2)\}^{1/2}$.
        $n_1 = 16$, $n_2 = 24$, $p = 300$.
        % based on $10000$ independently generated $Q$.
    }
    \label{fig:QQ1}
\end{figure}

\begin{figure}
    \centering
    \subfigure[$\gamma = 0$]{\includegraphics[width = 0.45\textwidth]{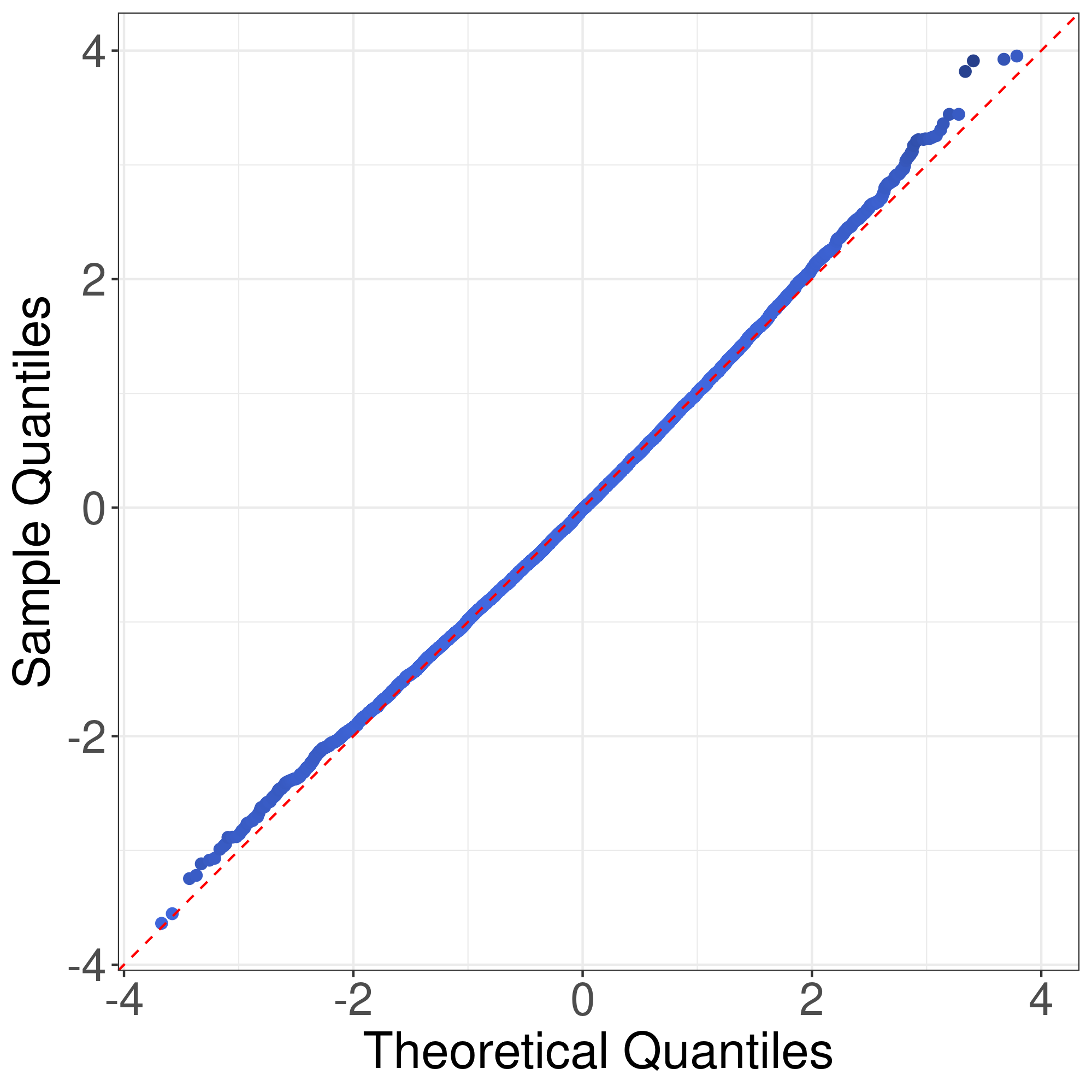}}
    \quad
    \subfigure[$\gamma = 1 / p^{1/2}$]{\includegraphics[width = 0.45\textwidth]{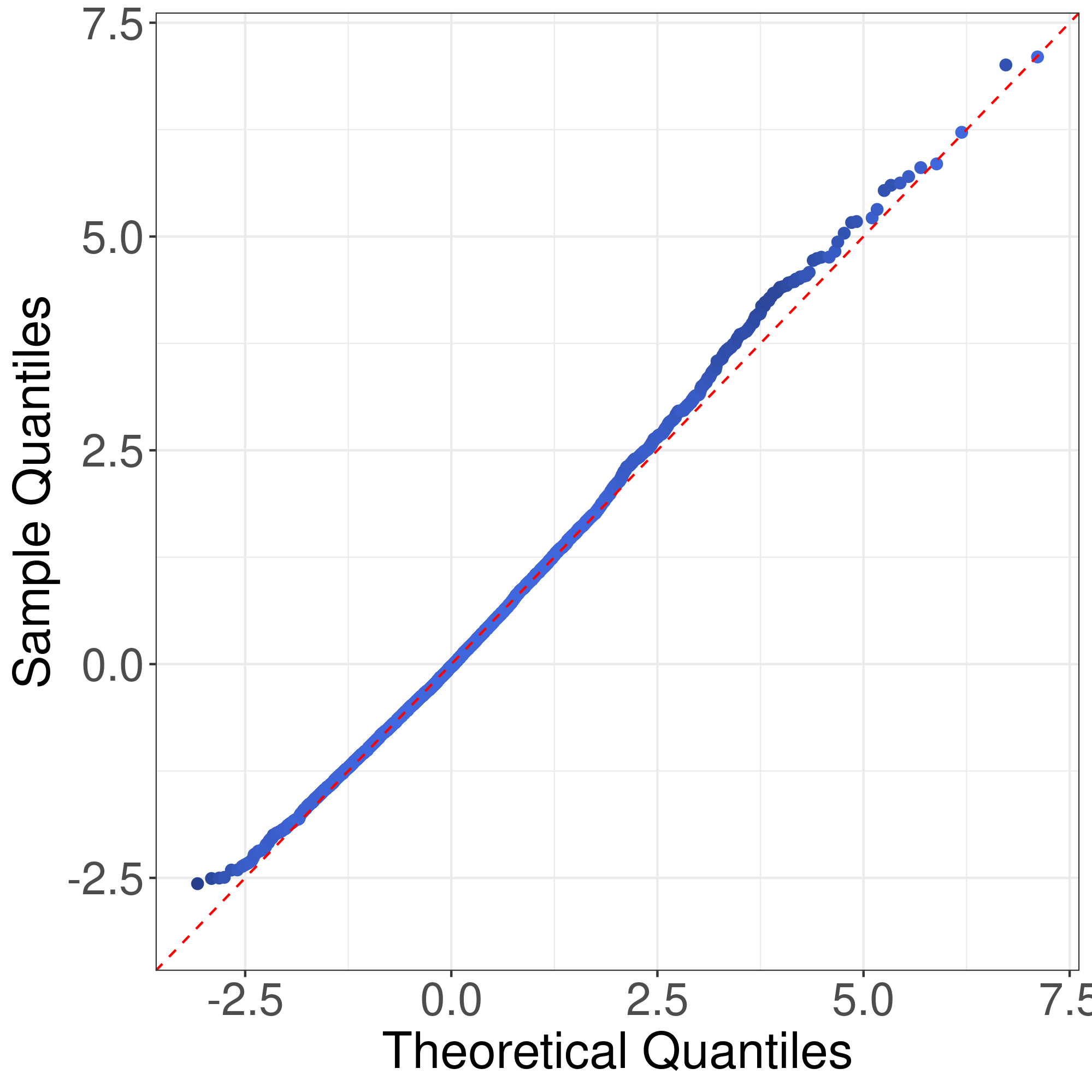}}
    \subfigure[$\gamma = 1/ 2$]{\includegraphics[width = 0.45\textwidth]{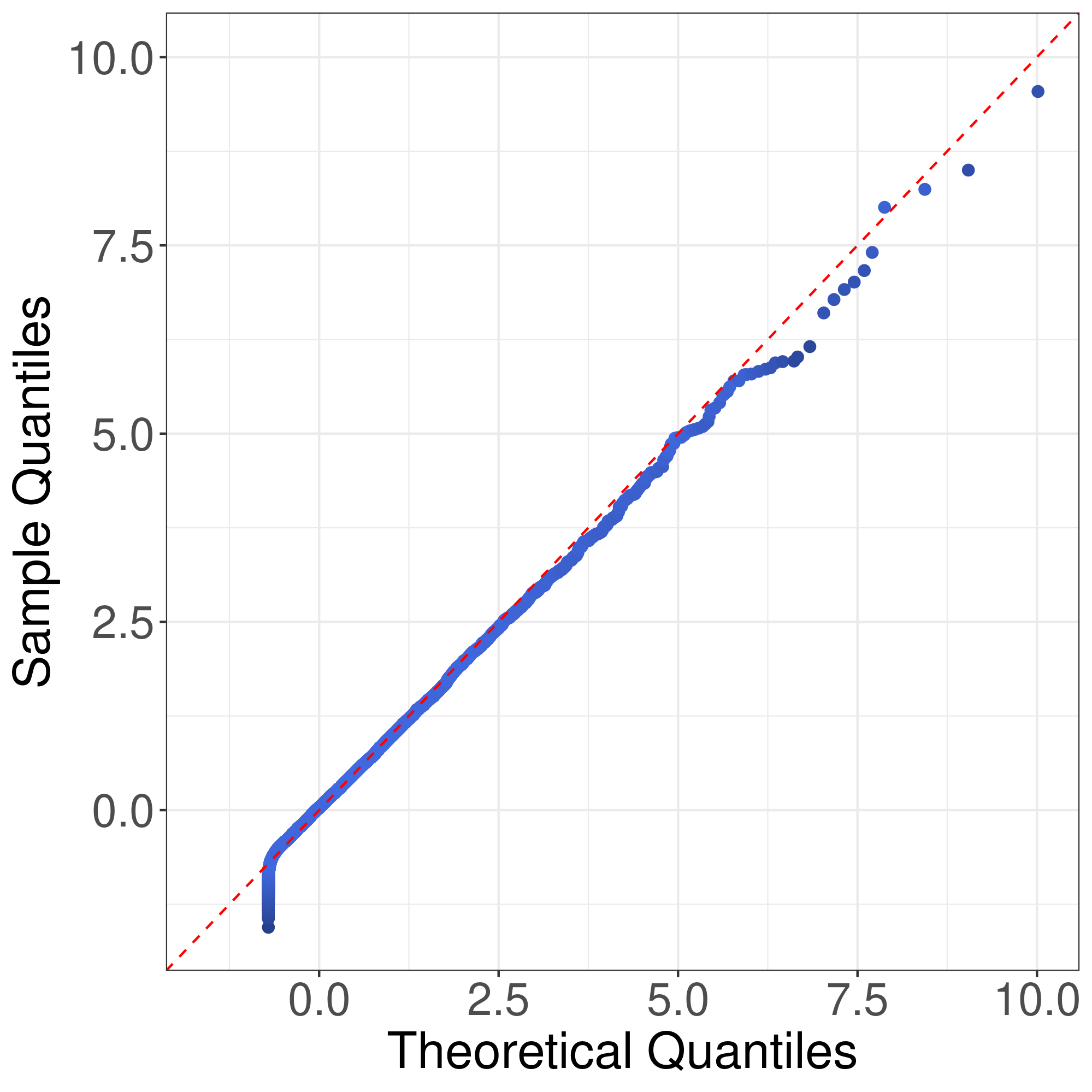}}
    \quad
    \subfigure[$\gamma = 1$]{\includegraphics[width = 0.45\textwidth]{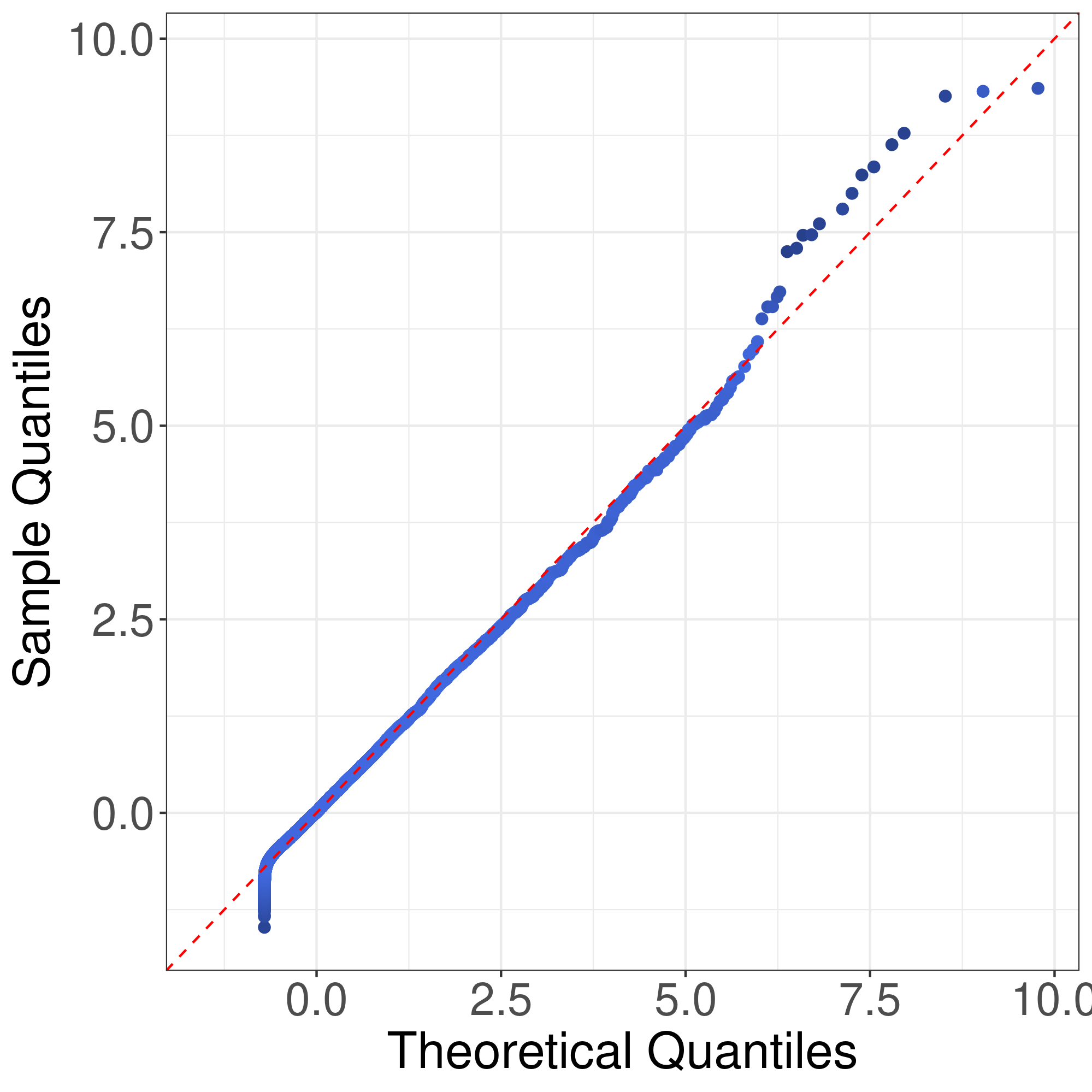}}
    \caption{
        Plots of the empirical quantiles of $T_{\mathrm{CQ}}(\BY_1, \BY_2) / \sigma_{T,n}$ against that of the asymptotic distribution in \eqref{eq:representation}.
        $n_1 = 16$, $n_2 = 24$, $p = 300$.
        % based on $10000$ independently generated $Q$.
    }
    \label{fig:QQ2}
\end{figure}

\begin{figure}
     \centering
     \subfigure[Model I]{\includegraphics[width = 0.48\textwidth]{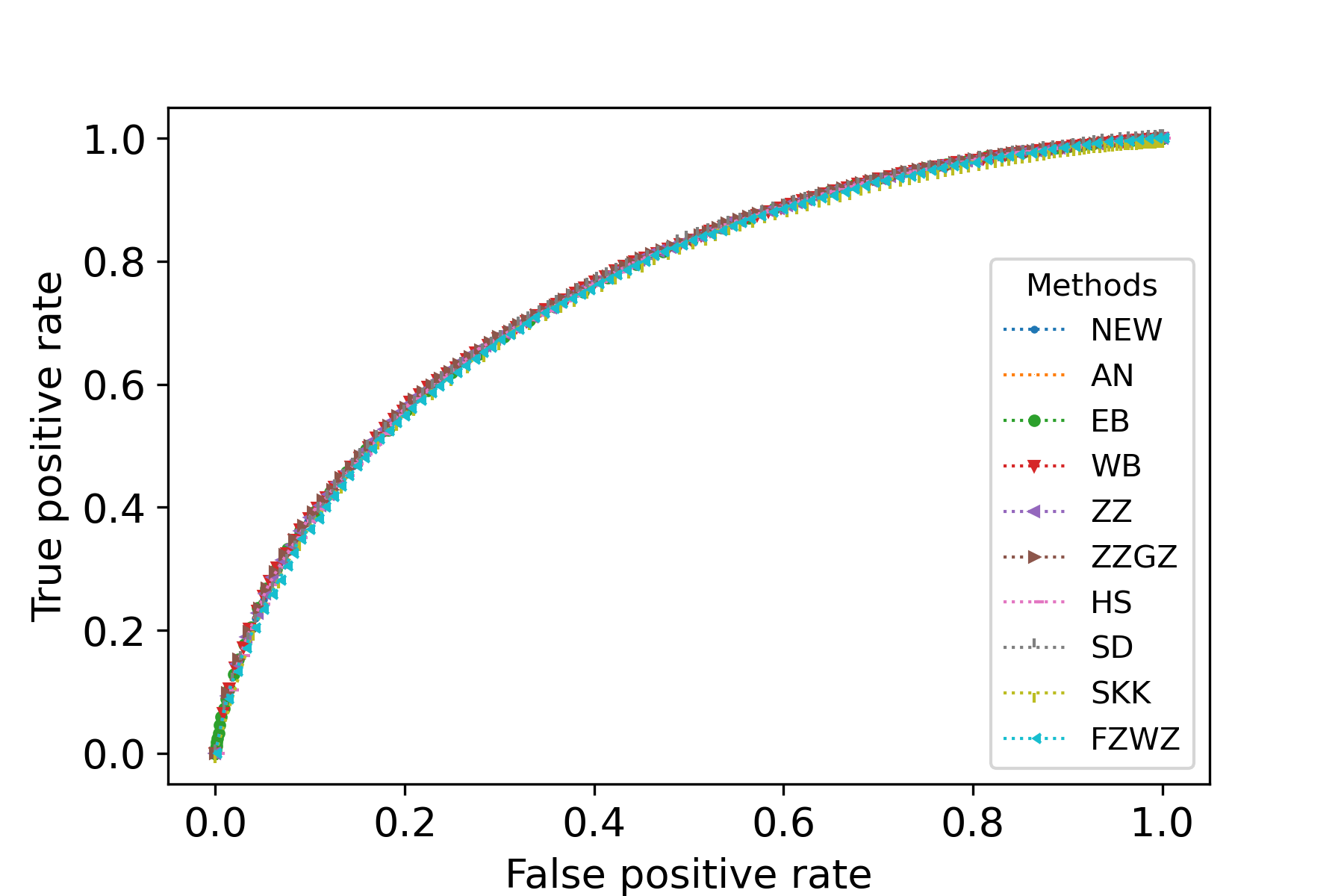}
     }
     \subfigure[Model II]{\includegraphics[width = 0.48\textwidth]{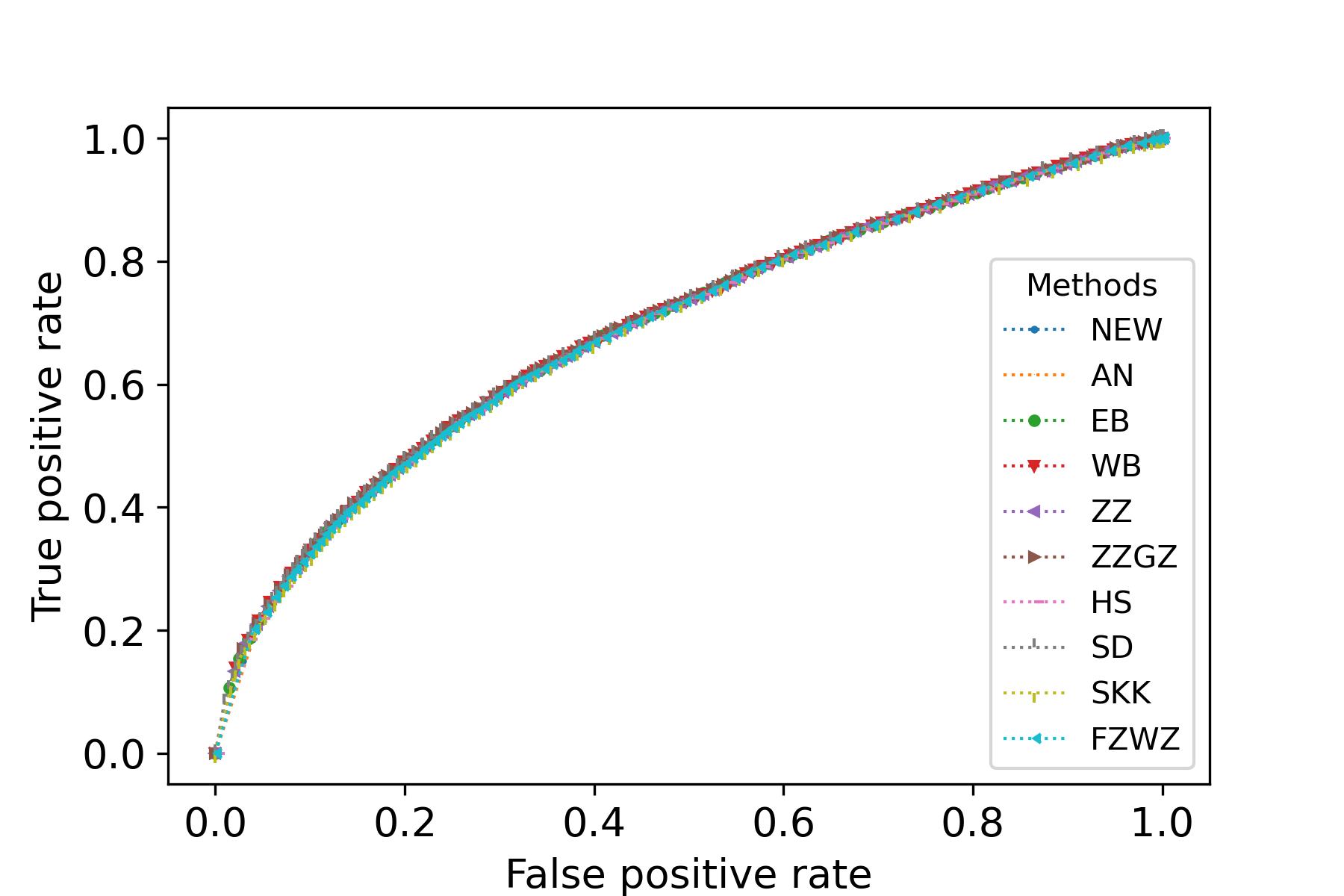}
     }
     \\
     \subfigure[Model III]{\includegraphics[width = 0.48\textwidth]{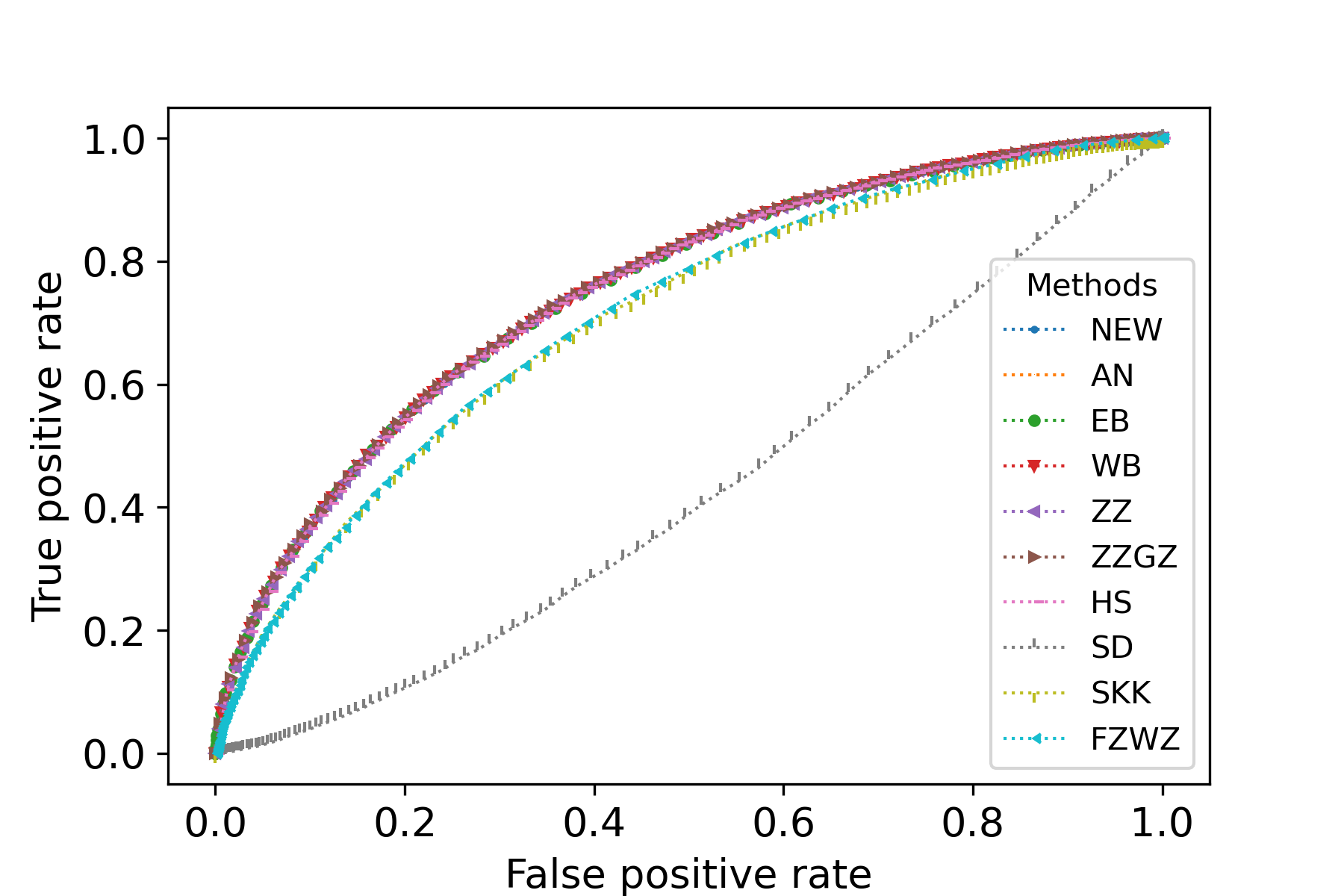}
     }
     \subfigure[Model IV]{\includegraphics[width = 0.48\textwidth]{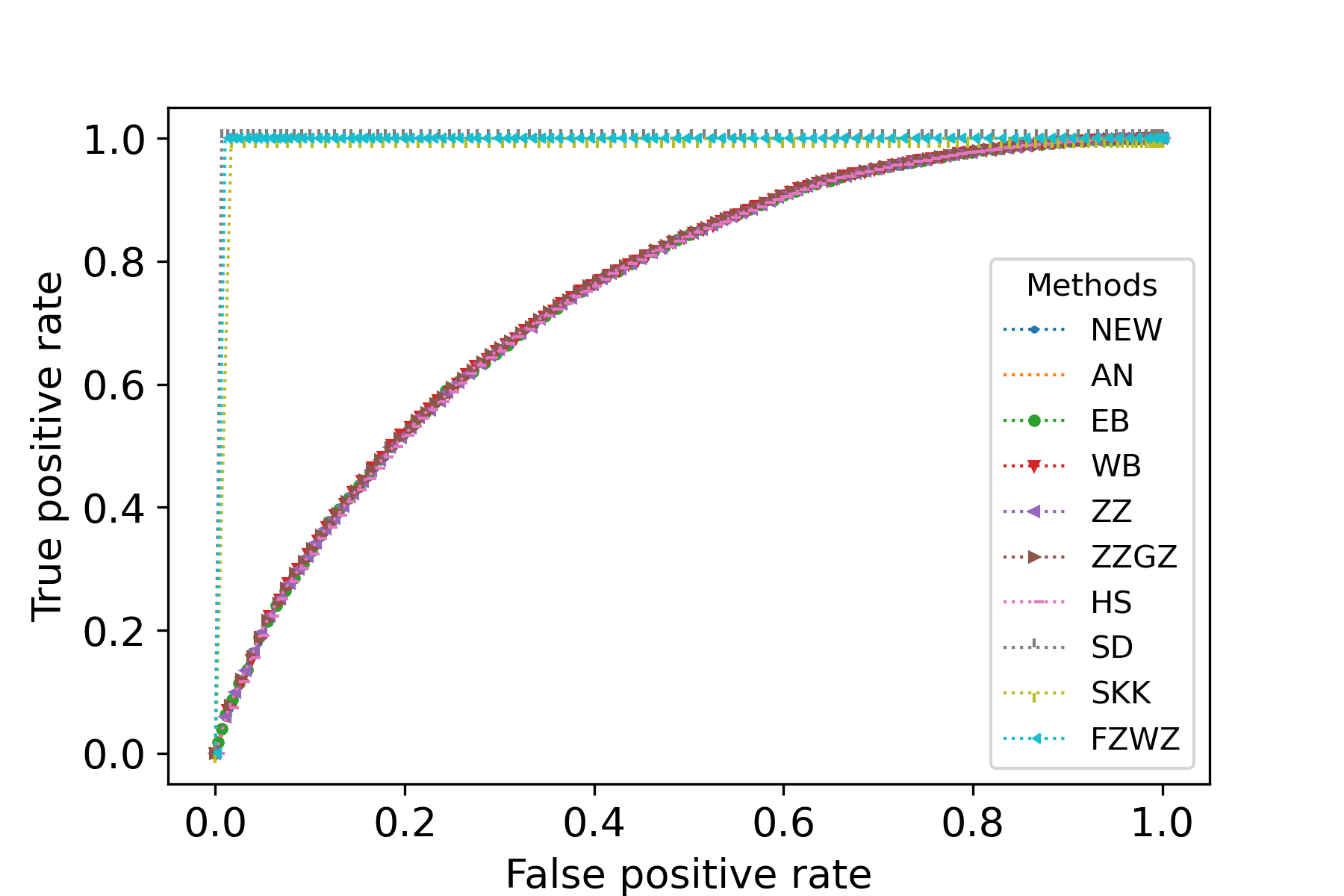}
     }

     \captionsetup{width=\textwidth}
     \caption{Receiver operating characteristic curve of various test procedures.
    NEW, the proposed test procedure;
    CQ, the test procedure of \cite{Chen2010ATwo}; % based on the asymptotic normality of $T_{\mathrm{CQ}}(\BX_1, \BX_2)$;
    EB, the empirical bootstrap method based on $T_{\mathrm{CQ}}(\BX_1, \BX_2)$;
    WB, the wild bootstrap method based on $T_{\mathrm{CQ}}(\BX_1, \BX_2)$;
    ZZ, the test procedure of \cite{Zhang2020AFurtherStudy};
    ZZGZ, the test procedure of \cite{Zhang2021Two-sample}; % based on Welch-Satterthwaite $\chi^2$-approximation.
    %LOU, the half-sampling method proposed in \cite{Lou2020HighDimensional}.
    LOU, the test procedure of \cite{Lou2020HighDimensional};
    SD, the test procedure of \cite{Srivastava2008A}; 
    SKK, the test procedure of \cite{SRIVASTAVA2013349};
    FZWZ, the test procedure of \cite{Feng2014Two}.
     }
     \label{fig:1}
 \end{figure}

Now we consider the correctness of Corollary \ref{corollary:the}.
Corollary \ref{corollary:the} claims that the general asymptotic distributions of $T_{\mathrm{CQ}}(\BY_1, \BY_2) / \sigma_{T, n}$ are weighted sums of independent normal random variable and centered $\chi^2$ random variables.
In Corollay \ref{corollary:the}, the parameters $\{\kappa_i\}_{i=1}^\infty$ relies on the limits of the eigenvalues of $\bPsi_n / \{\mytr(\bPsi_n^2)\}^{1/2}$ along a subsequence of $\{n\}$.
To cover different senarios of asymptotic distributions, we consider the following model.
Suppose $Y_{k,i} \sim \mathcal N  \{ \mathbf 0_p, \gamma \mathbf 1_{p} \mathbf 1_p^\myT + (1 - \gamma) \BI_p \} $, $i = 1, \ldots, n_k$, $k = 1, 2$, where $\gamma \in [0, 1]$.
In this case, the eigenvalues of $\bPsi_n / \{\mytr(\bPsi_n^2)\}^{1/2}$ are 
\begin{align*}
    \frac{
        p \gamma + 1 - \gamma
        }{
        \{
            (
                p \gamma + 1 - \gamma
            )^2
            +(p-1)(1-\gamma)^2
        \}^{1/2}
    }
    \quad 
    \textrm{ and }
    \quad
    \frac{
        1 - \gamma
        }{
        \{
            (
                p \gamma + 1 - \gamma
            )^2
            +(p-1)(1-\gamma)^2
        \}^{1/2}
    }
    ,
\end{align*}
with multiplicities $1$ and $p-1$, respectively.
We assume $p \to \infty$ as $n \to \infty$.
We consider four choices of $\gamma$.
First, we consider $\gamma = 0$.
In this case, $\kappa_i = 0$, $i =1, 2, \ldots$, and the asymptotic distribution of $T_{\mathrm{CQ}}(\BY_1, \BY_2) / \sigma_{T, n}$ is the standard normal distribution.
Second, we consider $\gamma = 1 / p^{1/2}$.
In this case, $\kappa_1 = 1/ \surd 2$, and $\kappa_i = 0$, $i = 2, 3, \ldots$, and the asymptotic distribution of $T_{\mathrm{CQ}}(\BY_1, \BY_2) / \sigma_{T, n}$ is $\mathcal N (0, 1) / \surd 2 + \{ \chi^2(1) - 1 \} / 2$.
In the third and fourth cases,  we consider $\gamma = 1/2$ and $1$, respectively.
In these two cases, $\kappa_1 = 1$, and $\kappa_i = 0$, $i = 2, 3, \ldots$, and the asymptotic distribution of $T_{\mathrm{CQ}}(\BY_1, \BY_2) / \sigma_{T, n}$ is the standardized $\chi^2$ distribution with $1$ degree of freedom. 
%We note that Models I-IV may not be suitable to demonstrate the correctness of Corollary 1.
%In fact, for Models I and III, we have $\kappa_i = 0$, $i = 1, 2, \ldots$  which degenerates to a special case.
%For Model II, the parameters $\{\kappa_i\}_{i=1}^\infty$ relies the limit of $n_1 / n_2$ as $n \to \infty$.
%For Model II, the eigenvalues of $\bPsi_n / \{\mytr(\bPsi_n^2)\}^2$ .
%For Model IV, the eigenvalues of $\bPsi_n / \{\mytr(\bPsi_n^2)\}^2$ do not have a simple form.
%To cover different senarios of limiting eigenvalues, we consider an extension of Model II.
%Suppose $Y_{k,i} \sim \mathcal N (\mathbf 0_p, \bar \bSigma_k)$, $i = 1, \ldots, n_k$, $k = 1, 2$, where $\bar \bSigma_k = \gamma \BV_k \bLambda \BV_k^\myT + \BI_p$.
%Here $\BV_k$ and $\bLambda$ are defined as in Model II.
Fig. \ref{fig:QQ2} illustrates the plots of 
the empirical quantiles of $T_{\mathrm{CQ}}(\BY_1, \BY_2) / \sigma_{T,n}$ against that of the asymptotic distribution in \eqref{eq:representation} for various values of $\gamma$.
The empirical quantiles of $T_{\mathrm{CQ}}(\BY_1, \BY_2) / \sigma_{T,n}$ are obtained by $10,000$ replications.
It can be seen that the distribution of $T_{\mathrm{CQ}}(\BY_1, \BY_2) / \sigma_{T,n}$
can be well approximated by the asymptotic distributions given in Corollary \ref{corollary:the}.
This verifies the conclusion of Corollary \ref{corollary:the}.
We note that the approximation $ \{ \bxi_p^\myT \bPsi_n \bxi_p - \mytr(\bPsi_n) \} / \{2 \mytr(\bPsi_n^2)\}^{1/2}$ in Theorem \ref{thm:universality_TCQ} is slightly better than the distribution approximations in Corollary \ref{corollary:the}.
This phenomenon is reasonable since the distributions in Corollary \ref{corollary:the} are in fact the asymptotic distributions of $ \{ \bxi_p^\myT \bPsi_n \bxi_p - \mytr(\bPsi_n) \} / \{2 \mytr(\bPsi_n^2)\}^{1/2}$ in Theorem \ref{thm:universality_TCQ}, and hence may have larger approximation error than $ \{ \bxi_p^\myT \bPsi_n \bxi_p - \mytr(\bPsi_n) \} / \{2 \mytr(\bPsi_n^2)\}^{1/2}$.

We have seen that many competing test procedures do not have a good control of test level.
To get rid of the effect of distorted test level, we plot the receiver operating characteristic curve of the test procedures.
Fig. \ref{fig:1} illustrates the receiver operating characteristic curve of various test procedures with $n_1=16$, $n_2 = 24$ and $p = 300$.
 It can be seen that for Models I and II, all test procedures have similar power behavior.
 For Model III, the scalar-invariant tests are less powerful than other tests.
 For Model IV, the scalar-invariant tests are more powerful than other tests.
%using the proposed test procedure will not effect the power of
These  results show that 
various test procedures based on $\|\bar \BX_1 - \bar \BX_2\|^2$ or $T_{\mathrm{CQ}}(\BX_1, \BX_2)$
may not have essential difference in power, and their performances are largely driven by the test level.
}

\section{Universality of generalized quadratic forms}% with low influnces}
\label{sec:univer}
In this section, we investigate the universality property of generalized quadratic forms, which is the key tool to study the distributional behavior of the proposed test procedure.
The result in this section is also interesting in its own right.

    Suppose $\xi_1, \ldots, \xi_n$ are independent random elements taking values in a Polish space $\mathcal X$.
    We consider the generalized quadratic form
    {\small
    \begin{align*}
        W(\xi_1, \ldots, \xi_n) = \sum_{1\leq i < j \leq n} w_{i,j}(\xi_i, \xi_j),
    \end{align*}
}%
    where $w_{i,j}(\cdot, \cdot): \mathcal X \times \mathcal X \to \mathbb R$ is measurable with respect to the product $\sigma$-algebra on $\mathcal X \times \mathcal X$,
    %(which, according to \cite{Cohn2013Measure}, Proposition 8.1.7, is also the Borel $\sigma$-algebra of $\mathcal X \times \mathcal X$),
    $1\leq i < j \leq n$.
    %and is symmetric in the sense that for any $\Ba, \Bb \in \mathbb R^p$, $w_{i,j}(\Ba, \Bb) = w_{i,j}(\Bb, \Ba)$,
    The generalized quadratic form includes the statistic $T_{\mathrm{CQ}}(\BY_1, \BY_2)$ as a special case.
    To see this, consider
    $\xi_i = Y_{1,i}$, $i = 1, \ldots, n_1$ 
    and 
    $\xi_j = Y_{2, j - n_1}$, $j = n_1 + 1, \ldots, n$.
    Let
    {\small
\begin{align*}
w_{i,j}(\xi_i, \xi_j) = 
\left\{
    \begin{array}{ll}
        \frac{2\xi_i^\myT \xi_j}{n_1 (n_1 - 1) }
&
\text{for } 1\leq i< j \leq n_1
,
\\
\frac{-2\xi_i^\myT \xi_j}{n_1 n_2 } 
&
\text{for } 1\leq i \leq n_1  \text{ and }  n_1+1 \leq j \leq n
,
\\
\frac{2\xi_i^\myT \xi_j}{n_2 (n_2 - 1) } 
&
\text{for } n_1 + 1\leq i< j \leq n
.
    \end{array}
\right.
\end{align*}
}%
In this case, the generalized quadratic form $W(\xi_1, \ldots, \xi_n) $ becomes the statistic $T_{\mathrm{CQ}} (\BY_1, \BY_2)$.
Similarly, conditioning on $\tilde \BX_1$ and $\tilde \BX_2$, the randomized statistic $T_{\mathrm{CQ}}(E ; \tilde \BX_1, \tilde \BX_2)$ is a special case of the generalized quadratic form with $\xi_i = \epsilon_{1,i}$, $i = 1, \ldots, m_1$ 
and 
$\xi_j = \epsilon_{2, j - m_1}$, $j = m_1 + 1, \ldots, m_1 + m_2$.
Hence it is meaningful to investigate the general behavior of the generalized quadratic form.
    %The distributional properties of the generalized quadratic forms have been extensively researched.

    The asymptotic normality of the generalized quadratic forms was studied by \cite{Jong1987A} via martingale central limit theorem and by \cite{Doebler2017Quantitative} via Stein's method.
    However, we are interested in the general setting in which
    $W(\xi_1, \ldots, \xi_n)$ may not be asymptotically normally distributed.
    Therefore,
    compared with the asymptotic normality,
    we are more interested in the \emph{universality} property of $W(\xi_1, \ldots, \xi_n)$; i.e., the distributional behavior of $W(\xi_1, \ldots, \xi_n)$ does not rely on the particular distribution of $\xi_1, \ldots, \xi_n$ asymptotically.
In this regard, many achievements have been made for the universality of $W(\xi_1, \ldots, \xi_n)$ for special form of $w_{i,j}(\xi_i, \xi_j)$; see, e.g., \cite{Mossel2010NoiseStability}, \cite{Nourdin2010InvariancePrinciples}, \cite{Xu2019Pearson} and the references therein.
    However, these results can not be used to deal with $T_{\mathrm{CQ}}(\BY_1, \BY_2)$ in our setting.
    In fact, the results in \cite{Mossel2010NoiseStability} and \cite{Nourdin2010InvariancePrinciples}
    can not be readily applied to $T_{\mathrm{CQ}}(\BY_1, \BY_2)$ while
the result in \cite{Xu2019Pearson} can only be applied to identically distributed observations.
    %In this section, we derive the universality of the general statistic $W(\xi_1, \ldots, \xi_n)$. 
    To the best of our knowledge, the universality of the generalized quadratic forms was never considered in the literature.
    We shall derive a universality property of the generalized quadratic forms using Lindeberg principle, an old and powerful technique; see, e.g., \cite{Chatterjee2006AGeneralization}, \cite{Mossel2010NoiseStability} for more about Lindeberg principle.

    %we study the \emph{universality phenomenon} of $W(\xi_1, \ldots, \xi_n)$.
    %We impose the following conditions on $\xi_1, \ldots, \xi_n$.
    \begin{assumption}\label{assumption:oh1}
    Suppose $\xi_1, \ldots, \xi_n$ are independent random elements taking values in a Polish space $\mathcal X$.
    Assume the following conditions hold for all $1\leq i < j \leq n $:
    \begin{enumerate}[(a)]
        \item
    %For all $\Ba \in \mathcal X$,
%$\myE \{w_{i,j}(\xi_i, \Ba)^4\} < \infty$,
%$\myE \{w_{i,j}(\Ba, \xi_j)^4\} < \infty
    %$,
    %and
    $
    \myE \{w_{i,j}(\xi_i, \xi_j)^4\}
    %=
    %\myE \{g_{i,j,4}(\xi_j)\}
    %=
    %\myE \{h_{i,j,4}(\xi_i)\}
    < \infty
    $.
    %Denote
    %        $g_{i,j,\ell}(\Ba):=\myE \{w_{i,j}(\xi_i, \Ba)\}^\ell$
    %        and
    %        $
%h_{i,j,\ell}(\Ba):=
%\myE \{w_{i,j}(\Ba, \xi_j )\}^\ell$,
 %$\ell =1 ,\ldots, 4$, $\Ba \in \mathcal X$.
\item
For all $\Ba \in \mathcal X$,
$
            %w_{i,j}(\Ba, \mathbf 0_p)
            %=
            %w_{i,j}(\mathbf 0_p, \Ba)
            %=
    %g_{i,j, 1}(\Ba)
\myE  \left \{w_{i,j}(\xi_i, \Ba)\right\}
            =
            \myE \left\{w_{i,j}(\Ba, \xi_j) \right\}
            = 0
            $.
    \end{enumerate}
    \end{assumption}
    Define
    $
    \sigma_{i,j}^2 = \myE \{w_{i,j}(\xi_i, \xi_j)^2\} 
    %= \myE \{g_{i,j,2}(\xi_j)\} = \myE \{h_{i,j,2}(\xi_i)\}
    $.
    Under Assumption \ref{assumption:oh1},
    we have
    $\myE \{W(\xi_1, \ldots, \xi_n) \} = 0 $ and $ \myVar \{W(\xi_1, \ldots, \xi_n)\} = \sum_{1 = 1}^n \sum_{j=i+1}^{n} \sigma_{i,j}^2$.
    %It turns out that the asymptotic distribution of $W(\xi_1, \ldots, \xi_n)$ does not rely on the particular distribution of $\xi_1, \ldots, \xi_n$ provided that they satisfy certain moment constraits and have low influences.
    %The universality phenomenon refers to does not heavily depend on the particular distribution of $\xi_1, \ldots, \xi_n$.
    %Let $\eta_1, \ldots, \eta_n$ be independent random elements in $\mathcal X$.
    We would like to give explicit bound for the difference between the distributions of $W(\xi_1, \ldots, \xi_n)$ and $W(\eta_1, \ldots, \eta_n)$ for a general class of random vectors $\eta_1, \ldots, \eta_n$.
    We impose the following conditions on $\eta_1, \ldots, \eta_n$.
    \begin{assumption}\label{assumption:oh2}
    Suppose $\eta_1, \ldots, \eta_n$ are independent random elements taking values in $\mathcal X$
    and are independent of $\xi_1, \ldots, \xi_n$.
    Assume the following conditions hold for all $1\leq i < j \leq n $:
    \begin{enumerate}[(a)]
        \item
    %For all $\Ba \in \mathcal X$,
%$\myE \{w_{i,j}(\eta_i, \Ba)^4\} < \infty$,
%$\myE \{w_{i,j}(\Ba, \eta_j)^4\} < \infty$,
%and
$
    \myE 
    \{
    w_{i,j}(\xi_i, \eta_j)^4
\} 
< \infty
    $,
$
    \myE 
    \{
    w_{i,j}(\eta_i, \xi_j)^4
\}
< \infty
    $ and
    $
    \myE \{w_{i,j}(\eta_i, \eta_j)^4\}
< \infty
    $.
\item
For all $\Ba \in \mathcal X$,
$
\myE \{ w_{i,j}(\eta_i, \Ba) \}
            =
            \myE \{ w_{i,j}(\Ba, \eta_j) \}
            =
            0
    $.
\item
    For any $\Ba, \Bb \in \mathcal X$,
    {\small
    \begin{align*}
    &
    \myE \{
        w_{i,k} (\Ba, \xi_k)
        w_{j,k} (\Bb, \xi_k)
    \}
    =
    \myE \{
        w_{i,k} (\Ba, \eta_k)
        w_{j,k} (\Bb, \eta_k)
    \},
    \quad
    \mathrm{for }
    \quad
    1\leq i \leq j < k \leq n
    ,
    \\
     &
    \myE \{
        w_{i,j} (\Ba, \xi_j)
        w_{j,k} (\xi_j, \Bb)
    \}
    =
    \myE \{
        w_{i,j} (\Ba, \eta_j)
        w_{j,k} (\eta_j, \Bb)
    \},
    \quad
    \mathrm{for }
    \quad
    1\leq i < j < k \leq n
    ,
    \\
     &
    \myE \{
        w_{i,j} (\xi_i, \Ba)
        w_{i,k} (\xi_i, \Bb)
    \}
    =
    \myE \{
        w_{i,j} (\eta_i, \Ba)
        w_{i,k} (\eta_i, \Bb)
    \}
    ,
    \quad
    \mathrm{for }
    \quad
    1\leq i < j \leq k \leq n
    .
\end{align*}
}%
%As a consequence,
%for any $1\leq i < j \leq n$
%    \begin{align*}
%    \myE 
%    \{
%    w_{i,j}(\xi_i, \eta_j)
%\}^2
%    =
%    \myE 
%    \{
%    w_{i,j}(\eta_i, \xi_j)
%\}^2
%     =
%    \sigma_{i,j}^2 
%    .
%    \end{align*}
    \end{enumerate}
    \end{assumption}
    We claim that under Assumptions \ref{assumption:oh1} and \ref{assumption:oh2}, there exists a nonnegative $C$ (which possibly depends on $n$) such that for all $1\leq i< j \leq n$,
    {\small
\begin{align}\label{eq:zuoteng}
    \max\left[
     \myE \{w_{i,j}(\xi_i, \xi_j)^4\} 
     ,
     \myE \{w_{i,j}(\xi_i, \eta_j)^4\} 
,
     \myE \{w_{i,j}(\eta_i, \xi_j)^4\} 
,
     \myE \{w_{i,j}(\eta_i, \eta_j)^4\} 
 \right]
     \leq C \sigma_{i,j}^4.
\end{align}
}%
In fact, by Assumption \ref{assumption:oh1}, (a)  and Assumption \ref{assumption:oh2}, (a), the left hand side of \eqref{eq:zuoteng} is finite.
Also, if $\sigma_{i,j}^4 = 0$, that is, $\myE\{w_{i,j}(\xi_i, \xi_j)^2\} = 0$, then from (c) of Assumption \ref{assumption:oh2},
{\small
\begin{align*}
    0
    =
    \myE\{w_{i,j}(\xi_i, \xi_j)^2\}
    =
    \myE\{w_{i,j}(\xi_i, \eta_j)^2\}
    =
    \myE\{w_{i,j}(\eta_i, \xi_j)^2\}
    =
    \myE\{w_{i,j}(\eta_i, \eta_j)^2\}
    .
\end{align*}
}%
It follows that
$
w_{i,j}(\xi_i, \xi_j)
=
w_{i,j}(\xi_i, \eta_j)
=
w_{i,j}(\eta_i, \xi_j)
=
w_{i,j}(\eta_i, \eta_j)
=0
$
almost surely.
In this case, the left hand side of \eqref{eq:zuoteng} is also $0$.
Hence our claim is valid.
%Hence there exists a $C \geq 0$ such that \eqref{eq:zuoteng}.
Let $\rho_n$ denote the minimum nonnegative $C$ such that \eqref{eq:zuoteng} holds for all $1\leq i< j \leq n$,
It will turn out that the difference between the distributions of $W(\xi_1, \ldots, \xi_n)$ and $W(\eta_1, \ldots, \eta_n)$ relies on $\rho_n$.

%Intuitively, $\rho_n$ characterizes how heavy the tails of $w_{i,j}(\xi_i, \xi_j)$ are.
    %As in \cite{Jong1987A}, we consider the case that $w_{i,j}(\xi_i, \xi_j)$ is degenerate.
    As in \cite{Mossel2010NoiseStability},
    %for $(\xi_1, \ldots, \xi_n) \in \mathscr U (C, \sigma_{1,2}^2, \ldots, \sigma_{n-1, n}^2)$,
    we define the 
    \emph{influence}
    of $\xi_i$ on $W(\xi_1, \ldots, \xi_n)$ as
    {\small
    \begin{align*}
        \mathrm{Inf}_{i}
        =
        \myE\left\{
            \myVar( W(\xi_1, \ldots, \xi_n) \mid \xi_1, \ldots, \xi_{i-1}, \xi_{i+1}, \ldots, \xi_n )
        \right\}.
    \end{align*}
}%
    It can be seen that
        $\mathrm{Inf}_i
        =
    \sum_{j=1}^{i-1} 
    %\myE \{w_{i,k} (\xi_i, \xi_k)\}^2
    \sigma_{j,i}^2
        +
    \sum_{j=i+1}^{n}
    %\myE \{w_{k,j} (\xi_k, \xi_j)\}^2
    \sigma_{i,j}^2
    $.
    %As noted in \cite{Mossel2010NoiseStability}, the universality of $W(\xi_1, \ldots, \xi_n)$ holds if $\max_{k\in\{1, \ldots, n\}}\mathrm{Inf}_k$ is small.
    %Hence $\mathrm{Inf}_{k}$ does not depend on the particular choice of $(\xi_1, \ldots, \xi_n)$ in $\mathscr U (C, \sigma_{1,2}^2, \ldots, \sigma_{n-1, n}^2)$.

    The following theorem provides a universality property of $W(\xi_1, \ldots, \xi_n)$.
% which is proved by Lindeberg principle

    \begin{theorem}\label{thm:universality_GQF}
        Under Assumptions \ref{assumption:oh1} and \ref{assumption:oh2},
        %and
         %$\sum_{i = 1}^n \sum_{j=i+1}^n \sigma_{i,j}^2 = 1$.
        we have
    %for any
    %$f \in \mathscr C_b^3(\mathbb R)$,
        {\small
\begin{align*}
\left\|
\mathcal L \{ W(\xi_1, \ldots, \xi_n) \}
-
\mathcal L \{ W(\eta_1, \ldots, \eta_n) \}
\right\|_3
    \leq
\frac{
    \rho_n^{3/4}
}
{3^{1/4}}
        \sum_{i=1}^n \mathrm{Inf}_i^{3/2}
        .
\end{align*}
}%
    \end{theorem}

    From Theorem \ref{thm:universality_GQF}, the distance between $W(\xi_1, \ldots, \xi_n)$ and $W(\eta_1, \ldots, \eta_n)$ is bounded by a function of $\rho_n$ and the influences.
    Suppose as $n \to \infty$, $\rho_n$ is bounded and $\sum_{i=1}^n \mathrm{Inf}_i^{3/2}$ tends to $0$.
    Then Theorem \ref{thm:universality_GQF} implies that
    %if $\rho_n$ is bounded and $\max_{k \in \{1, \ldots, n\}} \mathrm{Inf}_k \to 0$, 
    $W(\xi_1, \ldots, \xi_n)$
    and
    $W(\eta_1, \ldots, \eta_n)$
    share the same possible asymptotic dsitribution.
    That is, the distribution of $W(\xi_1, \ldots, \xi_n)$ enjoys a universality property.
In the proof of Theorem \ref{thm:universality_TCQ}
and Theorem \ref{thm:final_thm},
    we apply Theorem \ref{thm:universality_GQF} to $T_{\mathrm{CQ}}(\BY_1, \BY_2)$ and $T_{\mathrm{CQ}}(E; \tilde \BX_1, \tilde \BX_2)$, respectively, and consider normally distributed $\eta_i$, $i = 1, \ldots, n$.
    With this technique, the distributional behaviors of
$T_{\mathrm{CQ}}(\BY_1, \BY_2)$ and $T_{\mathrm{CQ}}(E ; \tilde \BX_1, \tilde \BX_2)$
are reduced to the circumstances where the observations are normally distributed.

\begin{proof}[{of Theorem \ref{thm:universality_GQF}}]

    For $k = 1, \ldots, n + 1$,
    define 
    {\small
    \begin{align*}
        W_k =& W(\eta_1, \ldots, \eta_{k-1}, \xi_{k}, \dots, \xi_n)
        .
    \end{align*}
}%
    Then $W_1 = W(\xi_1, \ldots, \xi_n)$ and $W_{n+1} = W(\eta_1, \ldots, \eta_n)$.
    Fix an $f \in \mathscr C_b^3 (\mathbb R)$.
    We have
    {\small
    \begin{align}\label{eq:723_1}
        \left|
        \myE f(W(\xi_1, \ldots, \xi_n))
        -
        \myE f(W(\eta_1, \ldots, \eta_n))
        \right|
        \leq&
        \sum_{k=1}^n
        \left|
        \myE
        \left\{
        f(W_k)
        -
        f(W_{k+1})
    \right\}
        \right|
        .
    \end{align}
}%

    Define
    {\small
    \begin{align*}
        W_{k,0}
        =
        \sum_{1\leq i < j \leq k-1} w_{i,j}(\eta_i, \eta_j)
        +
        \sum_{k+1\leq i < j \leq n} w_{i,j}(\xi_i, \xi_j)
        +
        \sum_{1\leq i \leq k-1}
        \sum_{k+1\leq j \leq n}
        w_{i,j}(\eta_i, \xi_j)
        .
    \end{align*}
}%
    Note that $W_{k,0}$ only relies on $\eta_1, \ldots, \eta_{k-1}, \xi_{k+1}, \ldots, \xi_{n}$.
    It can be seen that 
    {\small
\begin{align*}
    W_k 
    =&
    W_{k,0}
    +\sum_{i=1}^{k-1} w_{i,k}(\eta_i, \xi_k)
    + \sum_{j=k+1}^n w_{k,j}(\xi_k, \xi_j)
    ,
    \\
    W_{k+1} =&
    W_{k,0} 
    +\sum_{i=1}^{k-1} w_{i,k}(\eta_i, \eta_k)
    + \sum_{j=k+1}^n w_{k,j}(\eta_k, \xi_j)
    .
\end{align*}
}%
From Taylor's theorem, 
{\small
\begin{align}
    \label{eq:lindeberg_taylor_1}
    &
    \left|
    f(W_k) 
    -
    f(W_{k,0}) 
    -
    \sum_{i=1}^2 
    \frac{1}{i!}
    (W_{k} - W_{k,0})^i f^{(i)} (W_{k,0})
    \right|
    \leq
    \frac{\sup_{x\in \mathbb R} |f^{(3)}(x)| }{6} 
    \left|W_k - W_{k,0}\right|^3
    ,
    \\
    \label{eq:lindeberg_taylor_2}
    &
    \left|
    f(W_{k+1}) 
    -
    f(W_{k,0}) 
    -
    \sum_{i=1}^2 
    \frac{1}{i!}
    (W_{k+1} - W_{k,0})^i f^{(i)} (W_{k,0})
    \right|
    \leq
    \frac{\sup_{x\in \mathbb R} |f^{(3)}(x)| }{6} 
    \left|W_{k+1} - W_{k,0}\right|^3
    .
\end{align}
}%
Now we show that
{\small
\begin{align}\label{eq:moment_matching}
    \myE
    \left\{
    \sum_{i=1}^2 
    \frac{1}{i!}
    (W_{k} - W_{k,0})^i f^{(i)} (W_{k,0})
\right\}
    = 
    \myE
    \left\{
    \sum_{i=1}^2 
    \frac{1}{i!}
    (W_{k+1} - W_{k,0})^i f^{(i)} (W_{k,0})
\right\}
.
\end{align}
}%
By conditioning on $\eta_1, \ldots, \eta_{k-1}, \xi_{k+1}, \ldots, \xi_n$, it can be seen that \eqref{eq:moment_matching} holds provided that for $k = 1, \ldots, n$ and $\ell = 1, 2$,
{\small
\begin{align*}
    \myE \{(W_k - W_{k,0})^\ell \mid \eta_1, \ldots, \eta_{k-1}, \xi_{k+1}, \ldots, \xi_n \}
    =
    \myE \{(W_{k+1} - W_{k,0})^\ell \mid \eta_1, \ldots, \eta_{k-1}, \xi_{k+1}, \ldots, \xi_n \}
    .
\end{align*}
}%
For the case of $\ell = 1$, we have
{\small
\begin{align*}
    &
    \myE \{W_k - W_{k,0} \mid \eta_1, \ldots, \eta_{k-1}, \xi_{k+1}, \ldots, \xi_n \}
    =
    \sum_{i=1}^{k-1} \myE \{ w_{i,k}(\eta_i, \xi_k) \mid \eta_i \}
    + \sum_{j=k+1}^n \myE \{ w_{k,j}(\xi_k, \xi_j) \mid \xi_j \}
    ,
\end{align*}
}%
which equals $0$ by (b) of Assumption \ref{assumption:oh1}.
Similarly,
we have
{\small
\begin{align*}
\myE \{W_{k+1} - W_{k,0} \mid \eta_1, \ldots, \eta_{k-1}, \xi_{k+1}, \ldots, \xi_n \}
    =0
.
\end{align*}
}%
Now we deal with the case of $\ell = 2$.
From (c) of Assumption \ref{assumption:oh2}, we have
{\small
\begin{align*}
    &
    \myE \{(W_{k} - W_{k,0})^2 \mid \eta_1, \ldots, \eta_{k-1}, \xi_{k+1}, \ldots, \xi_n \}
    \\
    =&
    \sum_{i_1=1}^{k-1} 
    \sum_{i_2=1}^{k-1} 
    \myE \{
    w_{i_1,k}(\eta_{i_1}, \xi_k)
    w_{i_2,k}(\eta_{i_2}, \xi_k)
    \mid \eta_{i_1}, \eta_{i_2} \}
    +
    \sum_{j_1=k+1}^n 
    \sum_{j_2=k+1}^n 
    \myE \{
        w_{k,j_1}(\xi_k, \xi_{j_1})
        w_{k,j_2}(\xi_k, \xi_{j_2})
    \mid \xi_{j_1}, \xi_{j_2} \}
    \\
     &
    +
    2
    \sum_{i=1}^{k-1}
    \sum_{j=k+1}^n
    \myE \{
        w_{i,k}(\eta_{i}, \xi_k)
        w_{k,j}(\xi_k, \xi_{j})
        \mid
        \eta_i, \xi_j
    \}
    \\
    =&
    \sum_{i_1=1}^{k-1} 
    \sum_{i_2=1}^{k-1} 
    \myE \{
    w_{i_1,k}(\eta_{i_1}, \eta_k)
    w_{i_2,k}(\eta_{i_2}, \eta_k)
    \mid \eta_{i_1}, \eta_{i_2} \}
    +
    \sum_{j_1=k+1}^n 
    \sum_{j_2=k+1}^n 
    \myE \{
        w_{k,j_1}(\eta_k, \xi_{j_1})
        w_{k,j_2}(\eta_k, \xi_{j_2})
    \mid \xi_{j_1}, \xi_{j_2} \}
    \\
     &
    +
    2
    \sum_{i=1}^{k-1}
    \sum_{j=k+1}^n
    \myE \{
        w_{i,k}(\eta_{i}, \eta_k)
        w_{k,j}(\eta_k, \xi_{j})
        \mid
        \eta_i, \xi_j
    \}
    \\
    =
    &
    \myE \{(W_{k+1} - W_{k,0})^2 \mid \eta_1, \ldots, \eta_{k-1}, \xi_{k+1}, \ldots, \xi_n \}
    .
\end{align*}
}%
Therefore, \eqref{eq:moment_matching} holds.
It follows from %\eqref{eq:723_1},
\eqref{eq:lindeberg_taylor_1}, \eqref{eq:lindeberg_taylor_2} and \eqref{eq:moment_matching} that
{\small
\begin{align*}
        \left|
        \myE
        f(W_{k+1})
        -
        \myE
        f(W_k)
        \right|
        \leq&
        \frac{\sup_{x\in \mathbb R} |f^{(3) }(x)|}{6}
        \left(
        \myE
        \left|W_k - W_{k,0}\right|^3
        +
        \myE
        \left|W_{k+1} - W_{k,0}\right|^3
        \right)
        \\
        \leq&
        \frac{\sup_{x\in \mathbb R} |f^{(3) }(x)|}{6}
        \left[
            \left[
        \myE
        \left\{
            \left(W_k - W_{k,0}\right)^4
        \right\}
    \right]^{3/4}
        +
        \left[
        \myE
        \left\{
        \left(W_{k+1} - W_{k,0}\right)^4
    \right\}
\right]^{3/4}
        \right]
        .
\end{align*}
}%
Combining
\eqref{eq:723_1}
and the above inequality leads to
{\small
\begin{align*}
    &
        \left|
        \myE
        f(W(\eta_1, \ldots, \eta_n))
        -
        \myE
        f(W(\xi_1, \ldots, \xi_n))
        \right|
        \\
        \leq&
        \frac{
\sup_{x\in \mathbb R} |f^{(3)} (x)|
        }{6}
        \sum_{k=1}^n
        \left[
            \left[
        \myE
        \left\{
        \left(W_k - W_{k,0}\right)^4
    \right\}
\right]^{3/4}
        +
        \left[
            \myE\left\{\left(W_{k+1} - W_{k,0}\right)^4 \right\}
        \right]^{3/4}
    \right]
        .
\end{align*}
}%
Now we derive upper bounds for 
$
        \myE
        \{
        \left(W_{k} - W_{k,0}\right)^4
    \}
$
and
$
\myE
\{
\left(W_{k+1} - W_{k,0}\right)^4
\}
$.
We have
{\small
\begin{align*}
        \myE
        \left\{
        \left(W_{k} - W_{k,0}\right)^4
    \right\}
%        \\
%        =&
%        \myE
%\left\{
%    \sum_{i=1}^{k-1} w_{i,k}(\eta_i, \xi_k)
%    + \sum_{j=k+1}^N w_{k,j}(\xi_k , \xi_j)
%\right\}^4
        =&
        \myE
        \Big[
\Big\{
    \sum_{i=1}^{k-1} w_{i,k}(\eta_i, \xi_k)
\Big\}^4
\Big]
    + 
        \myE
        \Big[
\Big\{
    \sum_{j=k+1}^n w_{k,j}(\xi_k, \xi_j)
\Big\}^4
\Big]
\\
         &
+
6
\myE
\Big[
\Big\{
    \sum_{i=1}^{k-1} w_{i,k}(\eta_i, \xi_k)
\Big\}^2
\Big\{
    \sum_{j=k+1}^n w_{k,j} (\xi_k, \xi_j)
\Big\}^2
\Big]
\\
        =&
    \sum_{i=1}^{k-1}
        \myE
\left\{
    w_{i,k} (\eta_i, \xi_k)^4
\right\}
+
6
\sum_{i_1=1}^{k-1}
\sum_{i_2=i_1 + 1}^{k-1}
\myE
\left\{
    w_{i_1,k}(\eta_{i_1}, \xi_k)^2
    w_{i_2,k}(\eta_{i_2}, \xi_k)^2
\right\}
\\
         &
    +\sum_{j=k+1}^{n}
        \myE
\left\{
    w_{k,j}(\xi_k, \xi_j)^4
\right\}
+
6
\sum_{j_1=k + 1}^{n}
\sum_{j_2=j_1 + 1}^{n}
\myE
\left\{
    w_{k, j_1}(\xi_k, \xi_{j_1})^2
    w_{k, j_2}(\xi_k, \xi_{j_2})^2
\right\}
         \\
         &
+
    6\sum_{i=1}^{k-1} 
    \sum_{j=k+1}^n 
\myE
\left\{
    w_{i,k}(\eta_i, \xi_k)^2
    w_{k,j} (\xi_k, \xi_j)^2
\right\}
.
\end{align*}
}%
From the above equality and the definition of $\rho_n$, 
we have
{\small
\begin{align*}
        \myE
        \left\{
            \left(W_k - W_{k,0} \right)^4
    \right\}
        \leq&
        \rho_n
        \bigg(
    \sum_{i=1}^{k-1}
    \sigma_{i,k}^4
+
6
\sum_{i_1=1}^{k-1}
\sum_{i_2=i_1 + 1}^{k-1}
\sigma_{i_1, k}^2
\sigma_{i_2, k}^2
    \\
            &
    +\sum_{j=k+1}^{n}
    \sigma_{k,j}^4
+
6
\sum_{j_1=k + 1}^{n}
\sum_{j_2=j_1 + 1}^{n}
\sigma_{k, j_1}^2
\sigma_{k, j_2}^2
+
    6\sum_{i=1}^{k-1} 
    \sum_{j=k+1}^n 
\sigma_{i, k}^2
\sigma_{k, j}^2
\bigg)
\\
        \leq&
        3\rho_n
            \left(
            \sum_{i=1}^{k-1}\sigma_{i,k}^2 
            +
            \sum_{j=k+1}^n \sigma_{k,j}^2
            \right)^2
            \\
        =&
        3\rho_n \mathrm{Inf}_k^2
        .
\end{align*}
}%
Similarly,
        $\myE
        \left\{
            \left(W_{k+1} - W_{k,0} \right)^4
    \right\}
        \leq
        3\rho_n
        \mathrm{Inf}_k^2
        $.
Thus,
{\small
\begin{align*}
        \left|
        \myE f(W(\eta_1, \ldots, \eta_n))
        -
        \myE f(W(\xi_1, \ldots, \xi_n))
        \right|
        \leq&
\sup_{x\in \mathbb R} |f^{(3)} (x)|
\frac{
    \rho_n^{3/4}
}
{3^{1/4}}
        \sum_{i=1}^n \mathrm{Inf}_i^{3/2}
        .
\end{align*}
}%
This completes the proof.
\end{proof}

    \section{Technical results}

We begin with some notations that will be used throughout our proofs of main results.
For a matrix $\BA$,
let $\|\BA\|_F$ denote the Frobenious norm of $\BA$.
For a matrix $\BA \in \mathbb R^{p\times p}$, let $\lambda_i(\BA)$ denote the $i$th largest eigenvalue of $\BA$.
We adopt the convention that $\lambda_i(\BA) = 0$ for $i  > p$.

%    \begin{lemma}\label{lemma:subsequence_trick}
%A sequence of $m$-dimensional random vectors $\{Z_n\}_{n=1}^\infty$ converges weakly to a $m$-dimensional random variable $Z$ if and only if
%for any subsequence of $\{n\}$, there is a further subsequence along which $Z_n$ converges weakly ot $Z$.
%    \end{lemma}

The following lemma collects some trace inequalities that will be used in our proofs of main results.
To make our proofs of main results concise, the application of these trace inequalities are often tacit.
%the use of these inequalities are often silent.
\begin{lemma}\label{lemma:tao_matrix}
    The following trace inequalities hold:
    \begin{enumerate}[(i)]
        \item
            For any matrices $\BA \in \mathbb R^{p \times q}$, $\BB \in \mathbb R^{q \times r}$, 
        $\mytr( \BA \BB ) \leq \{\mytr(\BA \BA^\myT) \mytr(\BB \BB^\myT)\}^{1/2}$.
        \item
            For any positive semi-definite matrices $\BA, \BB  \in \mathbb R^{p \times p}$, 
                $0 \leq \mytr (\BA \BB) \leq \lambda_1(\BA) \mytr(\BB)$.
        \item
            Suppose $f$ is an increasing function from $\mathbb R$ to $\mathbb R$,
            and $\BA, \BB \in \mathbb R^{p \times p}$ are symmetric.
            Then
            $\mytr(f(\BA)) \leq \mytr (f(\BB))$.
        \item
            For any symmetric matrices $\BA, \BB \in \mathbb R^{p \times p}$, $\mytr\{(\BA \BB )^2\} \leq \mytr(\BA^2 \BB^2 ) $.
    \end{enumerate}
\end{lemma}
\begin{remark}
    Lemma \ref{lemma:tao_matrix}, (i) and (ii) are standard.
We refer to \cite{Tropp2015AnIntroduction}, Proposition 8.3.3 for a proof of Lemma \ref{lemma:tao_matrix}, (iii).
%Lemma \ref{lemma:tao_matrix} is a widely used trace inequality.
We refer to \cite{Tao2012TopicsInRMT}, inequality (3.18) for a proof of Lemma \ref{lemma:tao_matrix}, (iv).
\end{remark}

\begin{lemma}\label{lemma:final_matrix_1}
    Suppose $\BA_1, \ldots, \BA_n$ are $p \times p$ positive semi-definite matrices.
    Let $\bar \BA = n^{-1} \sum_{i=1}^n \BA_i$.
    Then
    {\small
        \begin{align*}
            &
            \left|
            \sum_{i=1}^n 
            \sum_{j=1}^n 
            \sum_{k=1}^n 
            \sum_{\ell=1}^n 
            \mytr(\BA_i \BA_j \BA_k \BA_{\ell})
            \mathbf 1_{\{i,j,k,\ell \textrm{ are distinct}\}}
            \right|
            \\
            \leq
            &
            n^4
            \mytr\left\{
                (\bar \BA)^4
            \right\}
            +
            10
            n^2
            \mytr\left\{
                (\bar \BA)^2
        \right\}
            \sum_{i=1}^n
            \mytr\left(
                \BA_i^2
            \right)
            +13
            \left\{
            \sum_{i=1}^n 
            \mytr(\BA_i^2 )
            \right\}^2
            .
        \end{align*}
    }%
\end{lemma}
\begin{proof}
    For distinct $i,j,k$, we have $\mathbf 1_{\{\ell \notin \{i,j,k\} \}} = 1 - \mathbf 1_{\{\ell = i \}} - \mathbf 1_{\{\ell = j \}} - \mathbf 1_{\{\ell = k \}}$.
    Hence
    {\small
        \begin{align*}
            &
            \sum_{i=1}^n 
            \sum_{j=1}^n 
            \sum_{k=1}^n 
            \sum_{\ell=1}^n 
            \mytr(\BA_i \BA_j \BA_k \BA_{\ell})
            \mathbf 1_{\{i,j,k,\ell \textrm{ are distinct}\}}
            \\
            =&
            n
            \sum_{i=1}^n 
            \sum_{j=1}^n 
            \sum_{k=1}^n 
            \mytr\left\{\BA_i \BA_j \BA_k 
                \bar \BA
            \right\}
            \mathbf 1_{\{i,j,k \textrm{ are distinct}\}}
            -
            2
            \sum_{i=1}^n 
            \sum_{j=1}^n 
            \sum_{k=1}^n 
            \mytr(\BA_i^2 \BA_j \BA_k)
            \mathbf 1_{\{i,j,k \textrm{ are distinct}\}}
            \\
             &
            -
            \sum_{i=1}^n 
            \sum_{j=1}^n 
            \sum_{k=1}^n 
            \mytr(\BA_i \BA_j \BA_k \BA_j)
            \mathbf 1_{\{i,j,k \textrm{ are distinct}\}}
            .
        \end{align*}
    }%
    Similarly, we have
    {\small
        \begin{align*}
            &
            \sum_{i=1}^n 
            \sum_{j=1}^n 
            \sum_{k=1}^n 
            \mytr\left\{\BA_i \BA_j \BA_k 
                \bar \BA
            \right\}
            \mathbf 1_{\{i,j,k \textrm{ are distinct}\}}
            \\
            =&
            n
            \sum_{i=1}^n 
            \sum_{j=1}^n 
            \mytr\left\{\BA_i \BA_j
                (\bar \BA)^2
            \right\}
            \mathbf 1_{\{i\neq j\}}
            -
            \sum_{i=1}^n 
            \sum_{j=1}^n 
            \mytr\left\{\BA_i \BA_j \BA_i 
                \bar \BA
            \right\}
            \mathbf 1_{\{i\neq j\}}
            -
            \sum_{i=1}^n 
            \sum_{j=1}^n 
            \mytr\left\{\BA_i \BA_j^2
                \bar \BA
            \right\}
            \mathbf 1_{\{i\neq j\}}
            \\
            %=&
            %n^3
            %\mytr\left\{
            %    (\bar \BA)^4
            %\right\}
            %-
            %n
            %\sum_{i=1}^n
            %\mytr\left\{
            %    \BA_i^2
            %    (\bar \BA)^2
            %\right\}
            %-
            %n
            %\sum_{i=1}^n 
            %\mytr\left(\BA_i \bar \BA \BA_i 
            %    \bar \BA
            %\right)
            %+
            %\sum_{i=1}^n 
            %\mytr\left(\BA_i^3 
            %    \bar \BA
            %\right)
            %-
            %n
            %\sum_{i=1}^n 
            %\mytr\left\{ \BA_i^2
            %    (\bar \BA)^2
            %\right\}
            %+
            %\sum_{i=1}^n 
            %\mytr\left(\BA_i^3
            %    \bar \BA
            %\right)
            %\\
            =&
            n^3
            \mytr\left\{
                (\bar \BA)^4
            \right\}
            -
            2
            n
            \sum_{i=1}^n
            \mytr\left\{
                \BA_i^2
                (\bar \BA)^2
            \right\}
            -
            n
            \sum_{i=1}^n 
            \mytr\left\{ (\BA_i \bar \BA)^2 \right\}
            +
            2
            \sum_{i=1}^n 
            \mytr\left(\BA_i^3 
                \bar \BA
            \right)
            .
        \end{align*}
    }%
    And
    {\small
        \begin{align*}
            &
            \sum_{i=1}^n 
            \sum_{j=1}^n 
            \sum_{k=1}^n 
            \mytr(\BA_i^2 \BA_j \BA_k)
            \mathbf 1_{\{i,j,k \textrm{ are distinct}\}}
            \\
            =&
            n
            \sum_{i=1}^n 
            \sum_{j=1}^n 
            \mytr(\BA_i^2 \BA_j \bar \BA)
            \mathbf 1_{\{i \neq j\}}
            -
            \sum_{i=1}^n 
            \sum_{j=1}^n 
            \mytr(\BA_i^3 \BA_j)
            \mathbf 1_{\{i \neq j\}}
            -
            \sum_{i=1}^n 
            \sum_{j=1}^n 
            \mytr(\BA_i^2 \BA_j^2)
            \mathbf 1_{\{i \neq j\}}
            \\
            %=&
            %n^2
            %\sum_{i=1}^n 
            %\mytr\{\BA_i^2 (\bar \BA)^2\}
            %-
            %n
            %\sum_{i=1}^n 
            %\mytr(\BA_i^3 \bar \BA)
            %-
            %n
            %\sum_{i=1}^n 
            %\mytr(\BA_i^3 \bar \BA)
            %+
            %\sum_{i=1}^n 
            %\mytr(\BA_i^4)
            %-
            %\sum_{i=1}^n 
            %\sum_{j=1}^n
            %\mytr(\BA_i^2 \BA_{j}^2)
            %+
            %\sum_{i=1}^n 
            %\mytr(\BA_i^4)
            %\\
            =&
            n^2
            \sum_{i=1}^n 
            \mytr\{\BA_i^2 (\bar \BA)^2\}
            -
            2 n
            \sum_{i=1}^n 
            \mytr(\BA_i^3 \bar \BA)
            +
            2
            \sum_{i=1}^n 
            \mytr(\BA_i^4)
            -
            \sum_{i=1}^n 
            \sum_{j=1}^n
            \mytr(\BA_i^2 \BA_{j}^2)
            .
        \end{align*}
    }%
    And
    {\small
        \begin{align*}
            &
            \sum_{i=1}^n 
            \sum_{j=1}^n 
            \sum_{k=1}^n 
            \mytr(\BA_i \BA_j \BA_k \BA_j)
            \mathbf 1_{\{i,j,k \textrm{ are distinct}\}}
            \\
            =&
            n
            \sum_{i=1}^n 
            \sum_{j=1}^n 
            \mytr(\BA_i \BA_j \bar \BA \BA_j)
            \mathbf 1_{\{i\neq j \}}
            -
            \sum_{i=1}^n 
            \sum_{j=1}^n 
            \mytr\{(\BA_i \BA_j)^2 \}
            \mathbf 1_{\{i\neq j \}}
            -
            \sum_{i=1}^n 
            \sum_{j=1}^n 
            \mytr(\BA_i \BA_j^3)
            \mathbf 1_{\{i\neq j \}}
            \\
            %=&
            %n^2
            %\sum_{i=1}^n \mytr\{(\BA_i \bar \BA)^2\}
            %-
            %n
            %\sum_{i=1}^n \mytr( \BA_i^3 \bar \BA)
            %-
            %\sum_{i=1}^n 
            %\sum_{j=1}^n 
            %\mytr\{(\BA_i \BA_j)^2 \}
            %+
            %\sum_{i=1}^n 
            %\mytr (\BA_i^4)
            %-
            %n
            %\sum_{i=1}^n 
            %\mytr(\BA_i^3 \bar \BA )
            %+
            %\sum_{i=1}^n 
            %\mytr(\BA_i^4)
            %\\
            =&
            n^2
            \sum_{i=1}^n \mytr\{(\BA_i \bar \BA)^2\}
            -
            2
            n
            \sum_{i=1}^n \mytr( \BA_i^3 \bar \BA)
            -
            \sum_{i=1}^n 
            \sum_{j=1}^n 
            \mytr\{(\BA_i \BA_j)^2 \}
            +
            2
            \sum_{i=1}^n 
            \mytr (\BA_i^4)
            .
        \end{align*}
    }%
It follows that
    {\small
        \begin{align*}
            &
            \sum_{i=1}^n 
            \sum_{j=1}^n 
            \sum_{k=1}^n 
            \sum_{\ell=1}^n 
            \mytr(\BA_i \BA_j \BA_k \BA_{\ell})
            \mathbf 1_{\{i,j,k,\ell \textrm{ are distinct}\}}
            \\
            %=
            %&
            %n^4
            %\mytr\left\{
            %    (\bar \BA)^4
            %\right\}
            %-
            %2
            %n^2
            %\sum_{i=1}^n
            %\mytr\left\{
            %    \BA_i^2
            %    (\bar \BA)^2
            %\right\}
            %-
            %n^2
            %\sum_{i=1}^n 
            %\mytr\left\{ (\BA_i \bar \BA)^2 \right\}
            %+
            %2 n
            %\sum_{i=1}^n 
            %\mytr\left(\BA_i^3 
            %    \bar \BA
            %\right)
            %\\
            %&
            %-2 n^2
            %\sum_{i=1}^n 
            %\mytr\{\BA_i^2 (\bar \BA)^2\}
            %+4 n
            %\sum_{i=1}^n 
            %\mytr(\BA_i^3 \bar \BA)
            %-4
            %\sum_{i=1}^n 
            %\mytr(\BA_i^4)
            %+2
            %\sum_{i=1}^n 
            %\sum_{j=1}^n
            %\mytr(\BA_i^2 \BA_{j}^2)
            %\\
            %&
            %-n^2
            %\sum_{i=1}^n \mytr\{(\BA_i \bar \BA)^2\}
            %+
            %2
            %n
            %\sum_{i=1}^n \mytr( \BA_i^3 \bar \BA)
            %+
            %\sum_{i=1}^n 
            %\sum_{j=1}^n 
            %\mytr\{(\BA_i \BA_j)^2 \}
            %-
            %2
            %\sum_{i=1}^n 
            %\mytr (\BA_i^4)
            %\\
            =
            &
            n^4
            \mytr\left\{
                (\bar \BA)^4
            \right\}
            -
            4
            n^2
            \sum_{i=1}^n
            \mytr\left\{
                \BA_i^2
                (\bar \BA)^2
            \right\}
            -
            2 n^2
            \sum_{i=1}^n 
            \mytr\left\{ (\BA_i \bar \BA)^2 \right\}
            +
            8 n
            \sum_{i=1}^n 
            \mytr\left(\BA_i^3 
                \bar \BA
            \right)
            \\
            &
            -6
            \sum_{i=1}^n 
            \mytr(\BA_i^4)
            +2
            \sum_{i=1}^n 
            \sum_{j=1}^n
            \mytr(\BA_i^2 \BA_{j}^2)
            +
            \sum_{i=1}^n 
            \sum_{j=1}^n 
            \mytr\{(\BA_i \BA_j)^2 \}
            .
        \end{align*}
    }%
    %Note that 
    %$
    %\mytr(\BA_i^3 \bar \BA) \leq (\mytr(\BA_i^4) + \mytr\{(\BA_i \bar \BA)^2\}) / 2
    %$.
    Thus, we have
    {\small
        \begin{align*}
            &
            \left|
            \sum_{i=1}^n 
            \sum_{j=1}^n 
            \sum_{k=1}^n 
            \sum_{\ell=1}^n 
            \mytr(\BA_i \BA_j \BA_k \BA_{\ell})
            \mathbf 1_{\{i,j,k,\ell \textrm{ are distinct}\}}
            \right|
            \\
            \leq
            &
            n^4
            \mytr\left\{
                (\bar \BA)^4
            \right\}
            +
            6
            n^2
            \sum_{i=1}^n
            \mytr\left\{
                \BA_i^2
                (\bar \BA)^2
        \right\}
            +
            8 n
            \left\{
            \sum_{i=1}^n 
            \mytr(\BA_i^4)
        \right\}^{1/2}
            \left\{
            \sum_{i=1}^n 
            \mytr\{\BA_i^2 (\bar \BA)^2\}
        \right\}^{1/2}
            +9
            \left\{
            \sum_{i=1}^n 
            \mytr(\BA_i^2 )
            \right\}^2
            \\
            \leq
            &
            n^4
            \mytr\left\{
                (\bar \BA)^4
            \right\}
            +
            10
            n^2
            \sum_{i=1}^n
            \mytr\left\{
                \BA_i^2
                (\bar \BA)^2
        \right\}
            +13
            \left\{
            \sum_{i=1}^n 
            \mytr(\BA_i^2 )
            \right\}^2
            \\
            \leq
            &
            n^4
            \mytr\left\{
                (\bar \BA)^4
            \right\}
            +
            10
            n^2
            \mytr\left\{
                (\bar \BA)^2
        \right\}
            \sum_{i=1}^n
            \mytr\left(
                \BA_i^2
            \right)
            +13
            \left\{
            \sum_{i=1}^n 
            \mytr(\BA_i^2 )
            \right\}^2
            .
        \end{align*}
    }%
    This completes the proof.
    
\end{proof}

\begin{lemma}\label{lemma:final_matrix_2}
    Suppose $\BA_1, \ldots, \BA_{n_1}$ and $\BB_1, \ldots, \BB_{n_2}$ are $p \times p$ positive semi-definite matrices.
    Let $\bar \BA = n_1^{-1} \sum_{i=1}^{n_1} \BA_i$ and $\bar \BB = n_2^{-1} \sum_{i=1}^{n_2} \BB_i$.
    Then
    {\small
        \begin{align*}
            &
            \left|
            \sum_{i=1}^{n_1}
            \sum_{j=1}^{n_1}
            \sum_{k=1}^{n_2}
            \sum_{\ell=1}^{n_2}
            \mytr(\BA_i \BB_{k} \BA_j \BB_{\ell})
            \mathbf{1}_{\{i \neq j\}}
            \mathbf{1}_{\{k \neq \ell\}}
            \right|
            \\
            \leq&
            \frac{1}{2}
            n_2^4
            \mytr\{( \bar \BA)^4\}
            +
            \frac{1}{2}
            n_1^4 \mytr\{ (\bar \BB )^4\}
            +
            n_2^2
            \mytr \{( \bar \BB )^2 \}
            \sum_{i=1}^{n_1}
            \mytr \{ \BA_i^2 \}
            +
            n_1^2
            \mytr \{(\bar \BA)^2\}
            \sum_{i=1}^{n_2}
        \mytr \{ ( \BB_{i})^2 \}
        +
        \left\{
            \sum_{i=1}^{n_1}
            \mytr ( \BA_i^2 )
        \right\}
        \left\{
            \sum_{i=1}^{n_2}
            \mytr ( \BB_{i}^2 )
        \right\}
            .
        \end{align*}
    }%
\end{lemma}
\begin{proof}
    Note that
    {\small
        \begin{align*}
            \mathbf{1}_{\{i \neq j\}}
            \mathbf{1}_{\{k \neq \ell\}}
            =
            (1 - \mathbf{1}_{\{i = j\}})
            (1 - \mathbf{1}_{\{k =\ell\}})
            =
            1
            -
            \mathbf{1}_{\{i = j\}} 
            - 
            \mathbf{1}_{\{k = \ell\}}
            + 
            \mathbf{1}_{\{i = j, k = \ell\}}
            .
        \end{align*}
    }%
    Hence we have
    {\small
        \begin{align*}
            &
            \sum_{i=1}^{n_1}
            \sum_{j=1}^{n_1}
            \sum_{k=1}^{n_2}
            \sum_{\ell=1}^{n_2}
            \mytr(\BA_i \BB_{k} \BA_j \BB_{\ell})
            \mathbf{1}_{\{i \neq j\}}
            \mathbf{1}_{\{k \neq \ell\}}
            \\
            =&
            n_1^2 n_2^2 \mytr\{( \bar \BA \bar \BB)^2 \}
            -
            n_2^2
            \sum_{i=1}^{n_1}
        \mytr \{(\BA_i \bar \BB)^2 \}
            -
            n_1^2
            \sum_{i=1}^{n_2}
        \mytr \{ (\bar \BA \BB_{i})^2 \}
            +
            \sum_{i=1}^{n_1}
            \sum_{j=1}^{n_2}
        \mytr \{ (\BA_i \BB_{j})^2 \}
            .
        \end{align*}
    }%
    It follows that
    {\small
        \begin{align*}
            &
            \left|
            \sum_{i=1}^{n_1}
            \sum_{j=1}^{n_1}
            \sum_{k=1}^{n_2}
            \sum_{\ell=1}^{n_2}
            \mytr(\BA_i \BB_{k} \BA_j \BB_{\ell})
            \mathbf{1}_{\{i \neq j\}}
            \mathbf{1}_{\{k \neq \ell\}}
            \right|
            \\
            \leq&
            \frac{1}{2}
            n_2^4
            \mytr\{( \bar \BA)^4\}
            +
            \frac{1}{2}
            n_1^4 \mytr\{ (\bar \BB )^4\}
            +
            n_2^2
            \mytr \{( \bar \BB )^2 \}
            \sum_{i=1}^{n_1}
            \mytr \{ \BA_i^2 \}
            +
            n_1^2
            \mytr \{(\bar \BA)^2\}
            \sum_{i=1}^{n_2}
        \mytr \{ ( \BB_{i})^2 \}
        +
        \left\{
            \sum_{i=1}^{n_1}
            \mytr ( \BA_i^2 )
        \right\}
        \left\{
            \sum_{i=1}^{n_2}
            \mytr ( \BB_{i}^2 )
        \right\}
            .
        \end{align*}
    }%
    This completes the proof.
\end{proof}

\begin{lemma}\label{lemma:add_error}
    Suppose $\xi$ and $\eta$ are two random variables.
    Then
    {\small
    \begin{align*}
        \|
        \mathcal L (\xi + \eta) - \mathcal L (\xi)
        \|_3
        \leq
        \left\{\myE (\eta^2) \right\}^{1/2}
        .
    \end{align*}
}%
\end{lemma}
\begin{proof}
    For $f \in \mathscr C_b^3 (\mathbb R)$,
    Taylor's theorem implies that
    {\small
    \begin{align*}
        |
        \myE f (\xi + \eta) - 
        \myE f (\xi)
        |
        \leq
        \sup_{x\in \mathbb R} |f^{(1)}(x)|
        \myE |\eta|
        \leq
        \sup_{x\in \mathbb R} |f^{(1)}(x)|
        \left\{\myE (\eta^2) \right\}^{1/2}
        .
    \end{align*}
}%
    Then the conclusion follows from the definition of the norm $\| \cdot \|_3$.
\end{proof}

    \begin{lemma}\label{lemma:712}
Under Assumptions \ref{assumption:n} and \ref{assumption7}, as $n \to \infty$,
{\small
\begin{align*}
    \sigma_{T,n}^2
    =
    &
    (1+o(1)) 2   
    \mytr (\bPsi_n^2)
    .
\end{align*}
}%
    \end{lemma}
    \begin{proof}
We have
{\small
\begin{align*}
    \sigma_{T,n}^2
    =
    &
    \sum_{k=1}^2
    \frac{2}{(n_k - 1)^2}
    \left\{
        \mytr
        (
            \bar \bSigma_{k}^2
            )
        -
    \frac{1}{n_k^2}
        \sum_{i=1}^{n_k}
        \mytr(\bSigma_{k,i}^2)
    \right\}
    +
    \frac{4}{n_1 n_2} 
    \mytr
    (
        \bar \bSigma_1
        \bar \bSigma_2
        )
        \\
    =&
    (1+o(1))
    \left\{
    \sum_{k=1}^2
    \frac{2}{(n_k-1)^2}
        \mytr
        (
            \bar \bSigma_{k}^2
            )
    +
    \frac{4}{n_1 n_2} 
    \mytr
    (
        \bar \bSigma_1
        \bar \bSigma_2
        )
    \right\}
        \\
    =&
    (1+o(1)) 2   
        \mytr
        \left\{
        (
        n_1^{-1}
            \bar \bSigma_{1}
            +
        n_2^{-1}
            \bar \bSigma_{2}
            )^2
        \right\}
    ,
\end{align*}
}%
where the second equality follows from Assumption \ref{assumption7} and the third equality follows from Assumption \ref{assumption:n}.
The conclusion follows.
    \end{proof}

\begin{lemma}\label{lemma:630}
    Suppose $\bZeta_n $ is a $d_n$-dimensional standard normal random vector where $\{d_n\}$ is an arbitrary sequence of positive integers.
    Suppose $\BA_n$ is a $d_n \times d_n$ symmetric matrix and $\BB_n$ is an $r \times d_n$ matrix where $r$ is a fixed positive integer.
    Furthermore, suppose
    {\small
    \begin{align*}
        \limsup_{n \to \infty} 
        %\{2\mytr(\BA_n^2) \}^{1/2}
        \mytr(\BA_n^2) 
        \in [0, +\infty)
        ,
        \quad
        \lim_{n\to \infty} \mytr(\BA_n^4)  = 0
        ,
        \quad
        \limsup_{n \to \infty} \|\BB_n \BB_n^\myT\|_F  \in [0, +\infty)
        .
    \end{align*}
}%
    Let $\{c_n\}$ be a sequence of positive numbers such that $| c_n - \{2 \mytr(\BA_n^2) \}^{1/2} | \to 0$.
    Let $\{\BD_n\}$ be a sequence of $r \times r$ matrices such that
    $\|\BD_n - \BB_n \BB_n^\myT\|_F \to 0$.
    Then
    {\small
    \begin{align*}
        \left\|
        \mathcal L
        \left(
            \bZeta_n^\myT \BA_n \bZeta_n - \mytr(\BA_n)
            + \bZeta_n^\myT \BB_n^\myT \BB_n \bZeta_n
            -\mytr( \BB_n^\myT \BB_n )
        \right)
        -
        \mathcal L
        \left(
            c_n
            \xi_0
            +
            \bxi_r^\myT
            \BD_n
            \bxi_r
            -\mytr(\BD_n)
        \right) 
        \right\|_3
        \to 0,
    \end{align*}
}%
where $\bxi_r$ is an $r$-dimensional standard normal random vector and $\xi_0$ is a standard normal random variable which is independent of $\bxi_r$.

\end{lemma}
%\begin{remark}
%Lemma \ref{lemma:630} does not follow from the asymptotic normality of quadratic forms.
%In fact, we allow $\lim_{n \to \infty}\mytr(\BA^2) = 0$.
%In this case,
    %$(Z^\myT \BA Z - \mytr(\BA)) / (2 \mytr(\BA^2))^{1/2}$ may not converges weakly to a normal distribution.
%Nevertheless, Lemma \ref{lemma:630} implies that
%$(Z^\myT \BA Z - \mytr(\BA))$ converges weakly to the unit mass on $0$ which is a special case of normal distribution.
%\end{remark}
\begin{proof}
    %Using the subsequence trick (see Lemma \ref{lemma:subsequence_trick}), we only need to prove the conclusion holds for a subsequence.
    %Note that 
    The conclusion holds if and only if for any subsequence of $\{n\}$, there is a further subsequence along which the conclusion holds.
    Using this subsequence trick, we only need to prove that the conclusion holds for a subsequence of $\{n\}$.
    %From the assumptions,
    By taking a subsequence, we can assume without loss of generality that
    {\small
    \begin{align*}
        \lim_{n\to \infty} \{2\mytr(\BA_n^2) \}^{1/2} = \gamma ,
        \quad
        \lim_{n\to \infty} \mytr(\BA_n^4) =0,
        \quad
        \lim_{n\to \infty} \BB_n\BB_n^\myT = \bOmega,
    \end{align*}
}%
where $\gamma \geq 0$ and $\bOmega$ is an $r \times r$ positive semi-definite matrix.
    Then we have $\lim_{n\to \infty} c_n = \gamma$ and $\lim_{n \to \infty} \BD_n = \bOmega $.
    Consequently,
    $
        \mathcal L
        \left(
            c_n \xi_0
            +
            \bxi_r^\myT 
            \BD_n
            \bxi_r
            -\mytr(\BD_n)
        \right) 
    $
    converges weakly to
    $
        \mathcal L
        \left(
            \gamma \xi_0
            +
            \bxi_r^\myT \bOmega \bxi_r
            -\mytr(\bOmega)
        \right) 
    $.
    Hence we only need to prove that
    {\small
    \begin{align*}
        \left\|
        \mathcal L
        \left(
            \bZeta_n^\myT \BA_n \bZeta_n - \mytr(\BA_n)
            + \bZeta_n^\myT \BB_n^\myT \BB_n \bZeta_n
            - \mytr(\BB_n^\myT \BB_n)
        \right)
        -
        \mathcal L
        \left(
            \gamma \xi_0
            +
            \bxi_r^\myT \bOmega \bxi_r
            -\mytr(\bOmega)
        \right) 
        \right\|_3
        \to 0.
    \end{align*}
}%
    To prove this, it suffices to prove that
    {\small
    \begin{align*}
        \begin{pmatrix}
         \bZeta_n^\myT \BA_n \bZeta_n - \mytr(\BA_n)
            \\
         \BB_n \bZeta_n
        \end{pmatrix}
        \rightsquigarrow
        \mathcal N 
        \left(
            \mathbf 0_{r+1}
            ,
            \begin{pmatrix}
                \gamma^2 &0\\
                0 & \bOmega
            \end{pmatrix}
        \right)
        ,
    \end{align*}
}%
    where ``$\rightsquigarrow$'' means weak convergence.

Denote by $\BA_n = \BP_n \bLambda_n \BP_n^\myT$ the spectral decomposition of $\BA_n$ where $\bLambda_n = \mydiag\{\lambda_{1} (\BA_n), \ldots, \lambda_{d_n}(\BA_{n})\}$
and $\BP_n \in \mathbb R$ is an ${d_n \times d_n}$ orthogonal matrix.
    Define $\bZeta_n^* = \BP_n^\myT \bZeta_n$.
    Then $\bZeta_n^*$ is also a $d_n$-dimensional standard normal random vector and
    $
    \bZeta_n^\myT \BA_n \bZeta_n - \mytr(\BA_n) = \bZeta_n^{*\myT} \bLambda_n \bZeta_n^* - \mytr (\bLambda_n)
    $,
    $\BB_n \bZeta_n = \BB_n^* \bZeta_n^* $ where $\BB_n^* =  \BB_n \BP_n $.
    %Without loss of generality, we assume $\BA$ is a diagonal matrix, otherwise one can multiply $Z$ by an orthogonal matirx $\BP$ such that $\BP^\myT \BA \BP$.

    %By Cram\'er-Wold device, we only need to show that
    %for any fixed vector $\Bb \in \mathbb R^r$,
    Fix $a \in \mathbb R$ and $\Bb \in \mathbb R^r$.
    {\small
    %\begin{align*}
    %    Z^{*\myT} \bLambda Z^* - \mytr(\bLambda) - \Bb^\myT \BB^* \BZ^*
    %    \rightsquigarrow
    %    \mathcal N(0, xxx ).
    %\end{align*}
    }%
    Let $\tilde b_{n, i}$ be the $i$th element of $\BB_n^{*\myT} \Bb$, $i=1, \ldots , d_n$.
    The characteristic function of the random vector $(\bZeta_n^{\myT} \BA_n \bZeta_n - \mytr(\BA_n) , (\BB_n \bZeta_n)^\myT)^\myT$ at $(a, \Bb^\myT)^\myT$ is 
    {\small
    \begin{align*}
        &
        \myE
        \exp
        \left\{
            i 
            \left(
            a (\bZeta_n^{\myT} \BA_n \bZeta_n - \mytr(\BA_n)) - \Bb^\myT \BB_n \bZeta_n
        \right)
    \right\}
        \\
=
        &
        \myE
        \exp
        \left\{
            i 
            \left(
            a (\bZeta_n^{*\myT} \bLambda_n \bZeta_n^* - \mytr(\bLambda_n)) - \Bb^\myT \BB_n^* \bZeta_n^*
        \right)
    \right\}
    \\
    =&
        \exp\left\{ 
            -\frac{1}{2}
            \sum_{j=1}^{d_n} \log (1- 2ia\lambda_{j}(\BA_n) )
            -\frac{1}{2} 
            \sum_{j=1}^{d_n} \frac{ \tilde b_{n,j}^2 }{1- 2ia\lambda_{j}(\BA_n)}
            -
            i
            a 
            \mytr(\BA_n)
    \right\},
    \end{align*}
}%
where the last equality can be obtained from the characteristic function of noncentral $\chi^2$ random variable, 
and for $\log$ functions, we put the cut along $(-\infty, 0]$, that is, $\log(z) = \log (|z|) + i \arg(z)$ with $-\pi < \arg(z) < \pi$.
%and for $\log$ and square root functions, we put the cut along $(-\infty, 0]$, that is, $\log(z) = \log (|z|) + i \arg(z)$ and $z^{1/2} = |z|^{1/2}\exp( i \arg(z)/2)$ with $-\pi < \arg(z) < \pi$.

    The condition $\lim_{n\to \infty} \mytr(\BA_n^4) = 0$ implies that $\lambda_{1}(\BA_n) \to 0$.
    Consequently, Taylor's theorem implies that
    {\small
    \begin{align*}
        \sum_{j=1}^{d_n} \log (1- 2ia\lambda_{j}(\BA_n))
        =
        \sum_{j=1}^{d_n}
        \left\{
            - 2ia\lambda_{j}(\BA_n) - \frac{1}{2} \left(2ia\lambda_{j}(\BA_n)\right)^2
            (1+e_{n,j})
    \right\},
    \end{align*}
}%
    where $\lim_{n\to \infty}\max_{j\in\{1, \ldots, d_n\}} |e_{n,j}| \to 0$.
    It follows that
    {\small
    \begin{align*}
        \sum_{j=1}^{d_n} \log (1- 2ia\lambda_j(\BA_n))
        =
        - 2ia
        \mytr(\BA_n)
        +
        (1+o(1)) 2 a^2 
        %\sum_{j=1}^{d_n} \lambda_j(\BA_n)^2
        \mytr(\BA_n^2)
        .
    \end{align*}
}%
On the other hand,
{\small
\begin{align*}
\sum_{j=1}^{d_n} { \tilde b_{n,j}^2 }/{(1- 2ia\lambda_j(\BA_n))}
=
(1+o(1))\sum_{j=1}^n \tilde b_{n,j}^2 
=
(1+o(1))
\Bb^\myT \BB_n \BB_n^\myT \Bb
.
\end{align*}
}%
Thus,
{\small
    \begin{align*}
        \myE
        \left[
        \exp
        \left\{
            i 
            \left(
            a (\bZeta_n^{\myT} \BA_n \bZeta_n - \mytr(\BA_n)) - \Bb^\myT \BB_n \bZeta_n
        \right)
    \right\}
\right]
    =&
        \exp\left\{ 
            -
            (1+o(1))
            \frac{1}{2} 
            \left(
                2 a^2 \mytr(\BA_n^2)
            +
                \Bb^\myT \BB_n \BB_n^\myT \Bb
        \right)
    \right\}
    \\
    \to&
        \exp\left\{ 
            -
            \frac{1}{2} 
            \left(
                \gamma^2
                a^2
            +
                \Bb^\myT \bOmega \Bb
        \right)
    \right\}
    ,
    \end{align*}
}%
where
$
        \exp\left\{ 
            -
            \frac{1}{2} 
            \left(
                \gamma^2
                a^2
            +
                \Bb^\myT \bOmega \Bb
        \right)
    \right\}
$
is the characteristic function of the distribution
{\small
    \begin{align*}
        \mathcal N 
        \left(
            \mathbf 0_{r+1}
            ,
            \begin{pmatrix}
                \gamma^2 &0\\
                0 & \bOmega
            \end{pmatrix}
        \right)
        .
    \end{align*}
}%
    This completes the proof.
\end{proof}

\begin{lemma}\label{lemma:713aa}
    Suppose $\bPsi_n$ is a $p \times p$ symmetric matrix where $p$ is a function of $n$.
    {\small
    \begin{align*}
    \lim_{n \to \infty} \frac{\lambda_i(\bPsi_n)}{\{\mytr(\bPsi_n^2) \}^{1/2}} = \kappa_i
    ,
    \quad
    i = 1, 2, \ldots.
    \end{align*}
}%
Let $\bxi_p$ be a $p$-dimensional standard normal random vector and $\{\xi_i\}_{i=0}^\infty$ a sequence of independent standard normal random variables.
    Then as $n \to \infty$,
    {\small
\begin{align*}
\left\|
\mathcal L \left(
    \frac{
        \bxi_p^\myT
        \bPsi_n
    \bxi_p
    -\mytr
        \left(
            \bPsi_n
    \right)
        }{
        \left\{
            2
            \mytr(
            \bPsi_n^2
            )
        \right\}^{1/2}
    }
\right) 
-
\mathcal L 
\left(
        (1 - \sum_{i=1}^{\infty} \kappa_{i}^2)^{1/2}
        \xi_0
+
2^{-1/2}
\sum_{i=1}^{\infty} \kappa_i (\xi_i^2-1)
\right)
\right\|_3
\to 0.
\end{align*}
}%
\end{lemma}
\begin{proof}
    Let $\kappa_{i,n} = \lambda_i (\bPsi_n) / \{\mytr(\bPsi_n^2) \}^{1/2}$.
    Then we have
    {\small
    \begin{align*}
    \mathcal L \left(
        \frac{
         \bxi_p^\myT
         \bPsi_n
         \bxi_p
            - 
            \mytr(\bPsi_n)
        }{
        \left\{
            2\mytr\left(
            \bPsi_n^2
        \right)
    \right\}^{1/2}
    }
\right)
        =
    \mathcal L 
    \left(
        2^{-1/2}\sum_{i=1}^p \kappa_{i,n} (\zeta_i^2 - 1)
    \right),
    \end{align*}
}%
    where $\{\zeta_i\}_{i=1}^\infty$ is a sequence of independent standard normal random variables which are independent of $\{\xi_i\}_{i=0}^\infty$.
%To prove \eqref{eq:toprove13},
From Fatou's lemma, we have $\sum_{i=1}^\infty \kappa_{i}^2 
\leq
\lim_{n\to \infty} \sum_{i=1}^\infty \kappa_{i,n}^2
= 1$.
    It follows from L\'evy's equivalence theorem and three-series theorem (see, e.g., \cite{dudleyProbability}, Theorem 9.7.1 and Theorem 9.7.3) that 
$
            \sum_{i=1}^r \kappa_{i} (\zeta_i^2 - 1)
$
converges weakly to 
$
            \sum_{i=1}^\infty \kappa_{i} (\zeta_i^2 - 1)
$
as $r\to \infty$.
Hence for any $\delta > 0$, there is a positive integer $r$ such that
{\small
\begin{align}
    &
    \Bigg\|
            \mathcal L
            \left(
                (1- \sum_{i=1}^\infty \kappa_i^2)^{1/2} \xi_0
                +
    2^{-1/2} 
            \sum_{i=1}^\infty \kappa_{i} (\zeta_i^2 - 1)
            \right)
            -
            \mathcal L
            \left(
                (1- \sum_{i=1}^r \kappa_i^2)^{1/2} 
                \xi_0
                +
    2^{-1/2} 
            \sum_{i=1}^r \kappa_{i} (\zeta_i^2 - 1)
            \right)
            \Bigg\|_3
            \leq \delta 
            .
            \label{eq:miaomiaoda3}
\end{align}
}%
By taking a possibly larger $r$, we can also assume that $\kappa_{r+1} \leq \delta $.
Now we fix such an $r$ and apply Lindeberg principle.
For $j =r+1, \ldots, p$,
define
{\small
\begin{align*}
    V_{j,0} = 
    2^{-1/2} 
    \sum_{i=1}^{j-1} \kappa_{i,n} (\zeta_i^2 - 1)
            +
            \sum_{i=j+1}^p \kappa_{i,n} \xi_i
            %,
            %\quad
    %V_{j} = 
    %2^{-1/2} 
    %\sum_{i=1}^{j-1} \kappa_{i,n} (\zeta_i^2 - 1)
            %+
            %\sum_{i=j}^p \kappa_{i,n} \xi_i
.
\end{align*}
}%
Then we have
%$V_j = V_{j, 0} + \kappa_{j, n} \xi_j$ and 
$ 
V_{j + 1, 0} + \kappa_{j + 1, n} \xi_{j+1} 
=
V_{j, 0} + 2^{-1/2} \kappa_{j, n}(\zeta_j^2 - 1)
$,
$j = r + 1, \ldots, p - 1$,
and
{\small
\begin{align*}
    &
    V_{r+1, 0} + \kappa_{r+1, n} \xi_{r+1} = 2^{-1/2} \sum_{i=1}^{r} \kappa_{i,n} (\zeta_i^2 - 1)+\sum_{i=r+1}^p \kappa_{i,n} \xi_i
,
\\
                                            &
    V_{p, 0} + 2^{-1/2} \kappa_{p, n} (\zeta_p^2 - 1) = 2^{-1/2} \sum_{i=1}^{p} \kappa_{i,n} (\zeta_i^2 - 1).
\end{align*}
}%
%$V_{r+1} = 2^{-1/2} \sum_{i=1}^{r} \kappa_{i,n} (\zeta_i^2 - 1)+\sum_{i=r+1}^p \kappa_{i,n} \xi_i$
%and
%$V_{p+1} = 2^{-1/2} \sum_{i=1}^{p} \kappa_{i,n} (\zeta_i^2 - 1)$.
%Also,
%$V_j = V_{j, 0} + \kappa_{j, n} \xi_j$
%and $V_{j+1} =  V_{j, 0} + 2^{-1/2} \kappa_{j, n}(\zeta_j^2 - 1)$.
For $f \in \mathscr C_b^3 (\mathbb R)$,
we have
{\small
\begin{align}
    &
    \left\|
    f\left(
    2^{-1/2} 
    \sum_{i=1}^{p} \kappa_{i,n} (\zeta_i^2 - 1)
    \right)
    -
    f\left(
    2^{-1/2} 
    \sum_{i=1}^{r} \kappa_{i,n} (\zeta_i^2 - 1)
            +
            \sum_{i=r+1}^p \kappa_{i,n} \xi_i
        \right)
            \right\|_3
            \notag
            \\
=
    &
\left|
\myE
f(V_{p, 0} + 2^{-1/2} \kappa_{p, n} (\zeta_p^2 - 1) )
-
\myE
f(V_{r+1,0} + \kappa_{r+1, n} \xi_{r+1})
\right|
            \notag
\\
%=
    %&
%\left|
%\myE
%f(V_{p, 0} + 2^{-1/2} \kappa_{p, n} (\zeta_p^2 - 1) )
%-
%\myE
%f(V_{p,0} + \kappa_{p, n} \xi_{p})
%+
%\sum_{j=r+1}^{p - 1}
%(
%\myE
%f(V_{j+1,0} + \kappa_{j+1, n} \xi_{j+1})
%-
%\myE
%f(V_{j,0} + \kappa_{j, n} \xi_{j})
%)
%\right|
%\\
=
    &
\left|
\sum_{j=r+1}^{p}
(
\myE
f(
V_{j, 0} + 2^{-1/2} \kappa_{j, n}(\zeta_j^2 - 1)
)
-
\myE
f(V_{j,0} + \kappa_{j, n} \xi_{j})
)
\right|
            \notag
\\
%\leq&
%\frac{1}{6} \myE \left\{|\xi_1|^3 + |2^{-1/2}(\zeta_1^2 - 1)|^3 \right\} 
%    \sup_{x\in \mathbb R} |f^{(3)}(x)|
%    \sum_{j=r+1}^p\kappa_{j,n}^3
\leq
    &
\sum_{j=r+1}^{p}
\left|
\myE
f(
V_{j, 0} + 2^{-1/2} \kappa_{j, n}(\zeta_j^2 - 1)
)
-
\myE
f(V_{j,0} + \kappa_{j, n} \xi_{j})
\right|
.
\label{eq:zuihouyitian1}
\end{align}
}%
For $j = r+1, \ldots, p$, we have
{\small
\begin{align*}
    &
    \Bigg|
f\left(
    V_{j, 0}
    +
    2^{-1/2} 
    \kappa_{j,n} (\zeta_j^2 - 1)
    \right)
    -
f\left(
    V_{j,0}
    \right)
    - \sum_{i=1}^2 \frac{1}{i!}
    \kappa_{j,n}^i 
    \left\{ 2^{-1/2}(\zeta_j^2 - 1) \right\}^i
    f^{(i)} \left(
        V_{j,0}
        \right)
        \Bigg|
    \leq
     \frac{1}{6}
     \kappa_{j,n}^3
    \left|2^{-1/2}(\zeta_j^2 - 1)\right|^3
     \sup_{x\in \mathbb R} |f^{(3)} (x)|
     ,
     \\
    &
    \Bigg|
f(V_{j,0} + \kappa_{j, n} \xi_{j})
    -
f\left(
    V_{j,0}
    \right)
    - \sum_{i=1}^2 \frac{1}{i!}
    \kappa_{j,n}^i \xi_j^i
    f^{(i)} \left(
        V_{j,0}
        \right)
        \Bigg|
    \leq
     \frac{1}{6}
     \kappa_{j,n}^3
     \left|\xi_j\right|^3
     \sup_{x\in \mathbb R} |f^{(3)} (x)|.
\end{align*}
}%
But
{\small
\begin{align*}
    \myE
    \left(\xi_j \right)
    =
    \myE
    \left\{ 2^{-1/2}(\zeta_j^2 - 1) \right\}
    =
    0
    ,\quad
    \myE
    \left(\xi_j^2 \right)
    =
    \myE
    \left[
    \left\{ 2^{-1/2}(\zeta_j^2 - 1) \right\}^2
\right]
    =
    1
.
\end{align*}
}%
Thus, for $j = r+1, \ldots, p$, we have
{\small
    \begin{align}\label{eq:zuihouyitian2}
\left|
\myE
f(
V_{j, 0} + 2^{-1/2} \kappa_{j, n}(\zeta_j^2 - 1)
)
-
\myE
f(V_{j,0} + \kappa_{j, n} \xi_{j})
\right|
\leq
    \frac{1}{6}
    \kappa_{j, n}^3
    \myE
    \left\{
        |2^{-1/2} (\zeta_1^2 - 1 )|^3
        +
        |\xi_1|^3
    \right\}
    \sup_{x\in \mathbb R} |f^{(3)} (x)|
    .
\end{align} 
}%
From \eqref{eq:zuihouyitian1} and \eqref{eq:zuihouyitian2}, we have
{\small
\begin{align*}
    &
    \left\|
    f\left(
    2^{-1/2} 
    \sum_{i=1}^{p} \kappa_{i,n} (\zeta_i^2 - 1)
    \right)
    -
    f\left(
    2^{-1/2} 
    \sum_{i=1}^{r} \kappa_{i,n} (\zeta_i^2 - 1)
            +
            \sum_{i=r+1}^p \kappa_{i,n} \xi_i
        \right)
            \right\|_3
            \\
\leq
    &
    \frac{1}{6}
    \sum_{j = r+1}^p
    \kappa_{j, n}^3
    \myE
    \left\{
        |2^{-1/2} (\zeta_1^2 - 1 )|^3
        +
        |\xi_1|^3
    \right\}
    \sup_{x\in \mathbb R} |f^{(3)} (x)|
\\
\leq&
\frac{\kappa_{r+1, n}}{6} \myE \left\{
|2^{-1/2}(\zeta_1^2 - 1)|^3
+ 
|\xi_1|^3
\right\} 
    \sup_{x\in \mathbb R} |f^{(3)}(x)|
    ,
\end{align*}
}%
where the last inequality holds since
$
    \sum_{j=r+1}^p\kappa_{j,n}^3
    \leq
    \kappa_{r+1, n}
    \sum_{j=r+1}^p\kappa_{j,n}^2
$
and
$\sum_{j=1}^p \kappa_{j,n}^2 = 1$.
Note that $\lim_{n \to \infty} \kappa_{r+1, n} = \kappa_{r+1} \leq \delta$.
Hence we have
{\small
\begin{align}
    &
    \limsup_{n\to\infty}
    \left\|
    \mathcal L \left(
    2^{-1/2} 
    \sum_{i=1}^{p} \kappa_{i,n} (\zeta_i^2 - 1)
    \right)
    -
    \mathcal L \left(
    2^{-1/2} 
    \sum_{i=1}^{r} \kappa_{i,n} (\zeta_i^2 - 1)
            +
            \sum_{i=r+1}^p \kappa_{i,n} \xi_i
        \right)
            \right\|_3
            \notag
            \\
    \leq&
\frac{
\delta
    }{6} \myE \left\{
|2^{-1/2}(\zeta_1^2 - 1)|^3 
+ 
|\xi_1|^3 
\right\} 
    .
    \label{eq:miaomiaoda}
\end{align}
}%
Note that $\sum_{i=r+1}^p \kappa_{i,n} \xi_i$ has the same distribution as $( 1- \sum_{i=1}^r \kappa_{i,n}^2 )^{1/2} \xi_0$.
Then from Lemma \ref{lemma:add_error},
{\small
\begin{align}
    &
    \left\|
    \mathcal L \left(
    2^{-1/2} 
    \sum_{i=1}^{r} \kappa_{i,n} (\zeta_i^2 - 1)
            +
            \sum_{i=r+1}^p \kappa_{i,n} \xi_i
        \right)
        -
    \mathcal L \left(
( 1- \sum_{i=1}^r \kappa_{i}^2 )^{1/2} \xi_0
+
    2^{-1/2} 
    \sum_{i=1}^{r} \kappa_{i} (\zeta_i^2 - 1)
        \right)
        \right\|_3
        \notag
        \\
        \leq&
        \left|
( 1- \sum_{i=1}^r \kappa_{i,n}^2 )^{1/2} 
-
( 1- \sum_{i=1}^r \kappa_{i}^2 )^{1/2} 
        \right|
        \left\{
            \myE (\xi_0^2)
        \right\}^{1/2}
        +
        \left(
        \sum_{i=1}^r
        |\kappa_{i,n} - \kappa_i|
    \right)
        2^{-1/2}
        \left[
            \myE \left\{(\zeta_1^2 -1)^2 \right\}
        \right]^{1/2},
    \label{eq:miaomiaoda2}
\end{align}
}%
where the right hand side tends to $0$ as $n \to \infty$.
From \eqref{eq:miaomiaoda3}, \eqref{eq:miaomiaoda} and \eqref{eq:miaomiaoda2},
we have
{\small
\begin{align*}
    &
    \limsup_{n \to \infty}
    \Bigg\|
    \mathcal L \left(
    2^{-1/2} 
    \sum_{i=1}^{p} \kappa_{i,n} (\zeta_i^2 - 1)
    \right)
    -
            \mathcal L
            \left(
                (1- \sum_{i=1}^\infty \kappa_i^2)^{1/2} \xi_0
                +
    2^{-1/2} 
            \sum_{i=1}^\infty \kappa_{i} (\zeta_i^2 - 1)
            \right)
            \Bigg\|_3
            \\
    \leq&
    \left[
\frac{1}{6} \myE \left\{|\xi_1|^3 + |2^{-1/2}(\zeta_1^2 - 1)|^3 \right\} 
+1
\right]
    \delta 
            .
\end{align*}
}%
But $\delta$ is an arbitrary positive real number.
Hence the above limit must be $0$.
This completes the proof.
\end{proof}

%\begin{lemma}\label{lemma:trace_ABAB}
%    Suppose $\BA, \BB \in \mathbb R^{p\times p}$ are symmetric matrices.
%    Then $\mytr(\BA \BB)$
%\end{lemma}

%\section{Proofs of main results} \label{sec:proofs_of_main_results}

\section{Proof of Lemma \ref{lemma:11_30}}

In this section, we provide the proof of Lemma \ref{lemma:11_30}.

%In this section, we provide proofs of our theoretical results in the main text.
%First we prove Theorem \ref{thm:universality_GQF} since this result serves as a basis for the proofs of other main results.

%\begin{proof}[of Lemma \ref{lemma:11_30}]
Fix an arbitrary $p \times p$ positive semi-definite matrix $\BB$.
    Let $\BG = \bGamma_{k,i}^\myT \BB \bGamma_{k,i}$.
    We only need to show that
    {\small
    \begin{align*}
        \myE \{(Z_{k,i}^\myT \BG Z_{k,i})^2\}
        \leq 3C \{\mytr(\BG)\}^2.
    \end{align*}
}%
    Let $g_{i,j}$ denote the $(i,j)$th element of $\BG$.
    Then
    {\small
    \begin{align*}
        (Z_{k,i}^\myT \BG Z_{k,i})^2
        =&
        \Big(
        2
        \sum_{j_1=1}^{m_{k,i}}
        \sum_{j_2=j_1+1}^{m_{k,i}}
        g_{j_1,j_2} z_{k,i,j_1} z_{k,i,j_2}
        +
        \sum_{j_3 = 1}^{m_{k,i}}
        g_{j_3, j_3} z_{k,i,j_3}^2
        \Big)^2
        \\
        =&
        4
        \Big(
            \sum_{j_1=1}^{m_{k,i}}
            \sum_{j_2=j_1+1}^{m_{k,i}}
        g_{j_1,j_2} z_{k,i,j_1} z_{k,i,j_2}
        \Big)^2
        +
        \Big(
            \sum_{j_3 = 1}^{m_{k,i}}
        g_{j_3, j_3} z_{k,i,j_3}^2
        \Big)^2
        \\
         &
        +
        4
        \Big(
            \sum_{j_1=1}^{m_{k,i}}
            \sum_{j_2=j_1+1}^{m_{k,i}}
        g_{j_1,j_2} z_{k,i,j_1} z_{k,i,j_2}
        \Big)
        \Big(
            \sum_{j_3 = 1}^{m_{k,i}}
        g_{j_3, j_3} z_{k,i,j_3}^2
        \Big)
        .
    \end{align*}
}%
    From the condition \eqref{eq:11_30a}, the expectation of the cross term is zero, and
    {\small
    \begin{align*}
        \myE
        \Big\{
        \Big(
            \sum_{j_1=1}^{m_{k,i}}
            \sum_{j_2=j_1+1}^{m_{k,i}}
        g_{j_1,j_2} z_{k,i,j_1} z_{k,i,j_2}
        \Big)^2
        \Big\}
        =&
        \sum_{j_1=1}^{m_{k,i}}
        \sum_{j_2=j_1+1}^{m_{k,i}}
        g_{j_1,j_2}^2 \myE(z_{k,i,j_1}^2 z_{k,i,j_2}^2)
        \\
        \leq&
        (C/2)
        \mytr(\BG^2)
        \\
        \leq&
        (C/2)
        \{\mytr(\BG)\}^2
        ,
    \end{align*}
}%
where the last inequality holds since $\BG$ is positive semi-definite.
    On the other hand
    {\small
    \begin{align*}
        \myE
        \Big\{
        \Big(
            \sum_{j_3 = 1}^{m_{k,i}}
        g_{j_3, j_3} z_{k,i,j_3}^2
        \Big)^2
        \Big\}
        \leq
        C
        \Big(
            \sum_{j_3 = 1}^{m_{k,i}} g_{j_3, j_3}
        \Big)^2
        = C \{\mytr(\BG)\}^2
        .
    \end{align*}
}
    Hence the conclusion holds.
%\end{proof}

    \section{Proof of Theorem \ref{thm:universality_TCQ}}
    In this section, we provide the proof of Theorem \ref{thm:universality_TCQ}.

Let $Y_{k, i}^*$, $i = 1, \ldots, n_k$, $k = 1, 2$, be independent $p$-dimensional random vectors with $Y_{k, i}^* \sim \mathcal N (\mathbf 0_p, \bSigma_{k,i})$.
Let $\BY_k^* = (Y_{k,1}, \ldots, Y_{k,n_k})^\myT$, $k = 1, 2$.
First we apply Theorem \ref{thm:universality_GQF} to prove that
{\small
\begin{align}\label{eq:first_to_prove}
        \left\|
        \mathcal L \left(
            \frac{T_{\mathrm{CQ}} (\BY_1, \BY_2) }{\sigma_{T,n} }
        \right)
        -
        \mathcal L \left(
            \frac{ T_{\mathrm{CQ}} (\BY_1^*, \BY_2^*) }{ \sigma_{T,n} }
        \right)
        \right\|_3
        \to 0
        %\max(n_1^{-1/2}, n_2^{-1/2})
        .
    \end{align}
}%

%In Theorem \ref{thm:universality_GQF},
Let 
{\small
\begin{align*}
\xi_{i} = 
\left\{
\begin{array}{ll}
    Y_{1,i}& \text{for } i = 1, \ldots, n_1,
    \\
    Y_{2, i-n_1}& \text{for }i = n_1 + 1, \ldots, n,
\end{array}
\right.
\quad
\text{and}
\quad
\eta_{i} = 
\left\{
\begin{array}{ll}
    Y_{1,i}^*& \text{for } i = 1, \ldots, n_1,
    \\
    Y_{2, i-n_1}^*& \text{for }i = n_1 + 1, \ldots, n.
\end{array}
\right.
\end{align*}
}%
Define
{\small
\begin{align*}
w_{i,j}(\Ba, \Bb) = 
\left\{
    \begin{array}{ll}
        \frac{2\Ba^\myT \Bb}{n_1 (n_1 - 1)\sigma_{T, n} }
&
\text{for } 1\leq i< j \leq n_1
,
\\
\frac{-2\Ba^\myT \Bb}{n_1 n_2 \sigma_{T, n} } 
&
\text{for } 1\leq i \leq n_1  \text{ and }  n_1+1 \leq j \leq n
,
\\
\frac{2\Ba^\myT \Bb}{n_2 (n_2 - 1)\sigma_{T, n} } 
&
\text{for } n_1 + 1\leq i< j \leq n
.
    \end{array}
\right.
\end{align*}
}%
%Then in Theorem \ref{thm:universality_GQF},
With the above definitions,
we have
$W(\xi_1, \ldots, \xi_n) = T_{\mathrm{CQ}} (\BY_1, \BY_2) / \sigma_{T, n} $
and
$W(\eta_1, \ldots, \eta_n) = 
T_{\mathrm{CQ}} ( \BY_1^*, \BY_2^* ) / \sigma_{T, n} $,
where the function $W(\cdot, \ldots, \cdot)$ is defined in Section \ref{sec:univer}.

%Now we verify the conditions of Theorem \ref{thm:universality_GQF}.
It follows from Assumption \ref{assumption:wangwang} and the fact that $Y_{k,i}^*$ has the same first two moments as $Y_{k,i}$ that
Assumptions \ref{assumption:oh1} and \ref{assumption:oh2} hold.
By direct calculation, we have
{\small
\begin{align*}
    \sigma_{i,j}^2=
    \left\{
    \begin{array}{ll}
        \frac{4 \mytr( \bSigma_{1,i}\bSigma_{1,j} )}{n_1^2 (n_1 - 1)^2 \sigma_{T,n}^2 } 
&
\text{for } 1\leq i< j \leq n_1
,
\\
\frac{4\mytr(\bSigma_{1,i} \bSigma_{2, j - n_1})}{n_1^2 n_2^2 \sigma_{T,n}^2 }
&
\text{for } 1\leq i \leq n_1 \text{ and }  n_1+1 \leq j \leq n
,
\\
\frac{4 \mytr(\bSigma_{2,i-n_1}\bSigma_{2,j-n_1}) }{n_2^2 (n_2 - 1)^2 \sigma_{T,n}^2 }
&
\text{for } n_1 + 1\leq i< j \leq n
.
    \end{array}
        \right.
\end{align*}
}%
Consequently,
{\small
\begin{align*}
    \mathrm{Inf}_i
    %=
    %\sum_{i=1}^{i-1}
    %\sigma_{i,i}^2
    %+
    %\sum_{j=i+1}^n
    %\sigma_{i,j}^2
    =
    \left\{
    \begin{array}{ll}
        \frac{4 \mytr(\bar \bSigma_1 \bSigma_{1,i})}{n_1 (n_1 - 1)^2 \sigma_{T,n}^2 } 
        -
        \frac{4 \mytr(\bSigma_{1,i}^2 )}{n_1^2 (n_1 - 1)^2 \sigma_{T,n}^2 } 
        +
        \frac{4\mytr(\bar \bSigma_2 \bSigma_{1, i} )}{n_1^2 n_2 \sigma_{T,n}^2 }
&
\text{for } 1\leq i \leq n_1
,
\\
\frac{4 \mytr(\bar \bSigma_2 \bSigma_{2, i - n_1})}{n_2 (n_2 - 1)^2\sigma_{T,n}^2 } 
-
\frac{4 \mytr(\bSigma_{2, i - n_1}^2 )}{n_2^2 (n_2 - 1)^2\sigma_{T,n}^2 } 
        +
        \frac{4\mytr(\bar \bSigma_1 \bSigma_{2, i - n_1})}{n_1 n_2^2 \sigma_{T,n}^2 }
&
\text{for } n_1+1 \leq i \leq n
.
    \end{array}
        \right.
\end{align*}
}%
We have
{\small
\begin{align*}
\sum_{i=1}^n \mathrm{Inf}_i^{3/2} 
\leq
\left( \max_{i\in\{1, \ldots, n\}} \mathrm{Inf}_i \right)^{1/2}
\left(\sum_{i=1}^n \mathrm{Inf}_i\right)
\leq
\left( 
\sum_{i=1}^n \mathrm{Inf}_i^2 \right)^{1 / 4}
\left(\sum_{i=1}^n \mathrm{Inf}_i\right)
.
\end{align*}
}%
It can be seen that
$
\sum_{i=1}^n \mathrm{Inf}_i
= 2
$.
On the other hand, from Assumption \ref{assumption7},
%for $i \in \{1, \ldots, n_1\}$,
{\small
\begin{align*}
    \sum_{i=1}^{n} \mathrm{Inf}_i^2
    \leq
    &
    \sum_{i=1}^{n_1}
    \left\{
    \frac{
        32 \mytr (\bar \bSigma_1^2) \mytr (\bSigma_{1,i}^2)
    }{ 
        n_1^2 (n_1 - 1)^4 \sigma_{T,n}^4
    }
    +
    \frac{32 \mytr (\bar \bSigma_2^2)\mytr ( \bSigma_{1,i}^2)}{n_1^4 n_2^2 \sigma_{T,n}^4}
\right\}
\\
&
+
    \sum_{i=1}^{n_2}
    \left\{
    \frac{
        32 \mytr (\bar \bSigma_2^2) \mytr (\bSigma_{2,i}^2)
    }{ 
        n_2^2 (n_2 - 1)^4 \sigma_{T,n}^4
    }
    +
    \frac{32 \mytr (\bar \bSigma_1^2)\mytr ( \bSigma_{2,i}^2)}{n_1^2 n_2^4 \sigma_{T,n}^4}
\right\}
\\
=&
o\left(
    \frac{
        \left\{\mytr (\bar \bSigma_1^2)
        \right\}^2
    }{ 
        (n_1 - 1)^4 \sigma_{T,n}^4
    }
+
    \frac{
        \left\{ \mytr (\bar \bSigma_2^2) \right\}^2
    }{ 
        (n_2 - 1)^4 \sigma_{T,n}^4
    }
    +
    \frac{\mytr (\bar \bSigma_1^2)\mytr (\bar \bSigma_2^2)}{n_1^2 n_2^2 \sigma_{T,n}^4}
    \right)
\\
=&
o\left(
    \frac{
        \left\{\mytr (\bar \bSigma_1^2)
        \right\}^2
    }{ 
        (n_1 - 1)^4 \sigma_{T,n}^4
    }
+
    \frac{
        \left\{ \mytr (\bar \bSigma_2^2) \right\}^2
    }{ 
        (n_2 - 1)^4 \sigma_{T,n}^4
    }
    \right)
    \\
    =&o(1).
\end{align*}
}%
It follows that $\sum_{i=1}^n \mathrm{Inf}_i^{3/2} \to 0$.

On the other hand,
from Assumption \ref{assumption:wangwang}, for all $1\leq i < j \leq n$,
{\small
\begin{align*}
    \max\left(
     \myE \{w_{i,j}(\xi_i, \xi_j)^4\} 
     ,
     \myE \{w_{i,j}(\xi_i, \eta_j)^4\} 
,
     \myE \{w_{i,j}(\eta_i, \xi_j)^4\} 
,
     \myE \{w_{i,j}(\eta_i, \eta_j)^4\} 
     \right)
     \leq \tau^2 \sigma_{i,j}^4.
\end{align*}
}%
Hence $\rho_n$ is upper bounded by the absolute constant $\tau^2$.
Thus, %from Thoerem \ref{thm:universality_GQF}, 
\eqref{eq:first_to_prove} holds.

Now we deal with the distribution of $T_{\mathrm{CQ}}(\BY_1^*, \BY_2^*) / \sigma_{T,n}$.
We have
{\small
\begin{align*}
    T_{\mathrm{CQ}} (\BY_1^*, \BY_2^*)
    =&
    \|\bar Y_1^* - \bar Y_2^* \|^2 -  
    \mytr(\bPsi_n)
     + 
     \sum_{k=1}^2 \frac{1}{n_k - 1} \left\{ \|\bar Y_k^*\|^2 - \frac{1}{n_k} \mytr (\bar \bSigma_k) \right\}
     \\
     &
     - \sum_{k=1}^2 \frac{1}{n_k (n_k - 1)}
     \sum_{i=1}^{n_k} \left\{\ \|Y_{k,i}^*\|^2 - \mytr(\bSigma_{k,i})\right\}
    ,
\end{align*}
}%
where $\bar Y_k^* = n_k^{-1} \sum_{i=1}^{n_k} Y_{k,i}$, $k = 1, 2$.
From Lemma \ref{lemma:add_error},
{\small
\begin{align}
    &
    \left\|
    \mathcal L \left(
        \frac{
        T_{\mathrm{CQ}} (\BY_1^*,, \BY_2^*)
    }{\sigma_{T, n}}
    \right)
    -
    \mathcal L \left(
        \frac{
            \|\bar Y_1^* - \bar Y_2^* \|^2 
            -\mytr(\bPsi_n)
        }{
        \sigma_{T,n}
    }
    \right)
    \right\|_3
    \notag
    \\
    \leq&
    \frac{1}{\sigma_{T, n}}
     \sum_{k=1}^2 
\frac{1}{n_k - 1} 
    \left[
    \myE
    \left\{
        \left( \|\bar Y_k^*\|^2 - \frac{1}{n_k} \mytr (\bar \bSigma_k) \right)^2
    \right\}
    \right]^{1/2}
    \notag
    \\
        &
    +
    \frac{1}{\sigma_{T, n}}
    \sum_{k=1}^2 \frac{1}{n_k (n_k - 1)}
    \left[
        \myE \left\{ \left(
                \sum_{i=1}^{n_k} \left(\ \|Y_{k,i}^*\|^2 - \mytr(\bSigma_{k,i})\right)
        \right)^2\right\}
    \right]^{1/2}
    \notag
    \\
    =&
    \frac{1}{\sigma_{T, n}}
    \sum_{k=1}^2 
    \frac{1}{n_k (n_k - 1)}
    \left\{ 2 \mytr(\bar \bSigma_k^2) \right\}^{1/2}
    +
    \frac{1}{\sigma_{T, n}}
    \sum_{k=1}^2
    \frac{1}{n_k (n_k - 1)} \left\{ 2\sum_{i=1}^{n_k} \mytr(\bSigma_{k,i}^2) \right\}^{1/2}
    \notag
    \\
    =& o(1)
    ,
    \label{eq:second_to_prove}
\end{align}
}%
where the last equality follows from Assumption \ref{assumption7}.

Note that $\bar Y_1^* - \bar Y_2^* \sim \mathcal N (\mathbf 0_p, \bPsi_n )$.
Then from Lemma \ref{lemma:add_error}, 
{\small
\begin{align}
    &
    \left\|
    \mathcal L \left(
        \frac{
            \|\bar Y_1^* - \bar Y_2^* \|^2 -
            \mytr(\bPsi_n)
        }{
        \sigma_{T,n}
    }
    \right)
    -
    \mathcal L \left(
        \frac{
            \|\bar Y_1^* -  \bar Y_2^* \|^2 -
            \mytr(\bPsi_n)
        }{
        \left\{
            2\mytr(
                \bPsi_n^2
                )
            \right\}^{1/2}
    }
\right)
    \right\|_3
    \notag
    \\
    \leq&
    \left|
    \frac{
        \left\{
            2\mytr\left(
                \bPsi_n^2
            \right)
        \right\}^{1/2}
    }
    {
    \sigma_{T,n}
}
    -1
    \right|
    \left[
    \myE
    \left[
        \frac{
        \|\bar Y_1^* - \bar Y_2^* \|^2 -  
        \mytr(\bPsi_n)
        }{
        \left\{
            2\mytr\left(
                \bPsi_n^2
            \right)
        \right\}^{1/2}
    }
\right]^2
\right]^{1/2}
\notag
    \\
    =&
    \left|
    \frac{
        \left\{
            2\mytr\left(
                \bPsi_n^2
            \right)
        \right\}^{1/2}
    }
    {
    \sigma_{T,n}
}
    -1
    \right|
    \notag
    \\
    \to& 0
    ,
    \label{eq:third_to_prove}
\end{align}
}%
where the last equality follows from Lemma \ref{lemma:712}.
Then the conclusion follows from \eqref{eq:first_to_prove}, \eqref{eq:second_to_prove} and \eqref{eq:third_to_prove}.

\section{Proof of Corollary \ref{corollary:the}}
In this section, we provide the proof of Corollary \ref{corollary:the}.

    Since ${T_{\mathrm{CQ}} (\BY_1, \BY_2) }/{\sigma_{T,n} }$ has zero mean and unit variance, 
    the distribution of
$
        {T_{\mathrm{CQ}} (\BY_1, \BY_2) }/{\sigma_{T,n} }
    $ is uniformly tight.
    From Theorem \ref{thm:universality_TCQ}, 
    $
        {T_{\mathrm{CQ}} (\BY_1, \BY_2) }/{\sigma_{T,n} }
    $
    and
    $
    (
         \bxi_p^\myT
         \bPsi_n
         \bxi_p
            - 
            \mytr(\bPsi_n)
            )
            /
        \{
            2\mytr(
            \bPsi_n^2
        )
    \}^{1/2}
    $
    share the same possible asymptotic distributions.
    Hence we only need to find all possible asymptotic distributions of
$
    (
         \bxi_p^\myT
         \bPsi_n
         \bxi_p
            - 
            \mytr(\bPsi_n)
            )
            /
        \{
            2\mytr(
            \bPsi_n^2
        )
    \}^{1/2}
    $.
    %Hence without loss of generality, we can assume that
    %$
    %        \mytr(
    %        \bPsi_n^2
    %    ) = 1
    %$.
    %Since the distribution of $\bxi_p$ is invariant under the linear transformation by orthogonal matrices, we can without loss of generality and assume that $\bPsi_n = \mydiag(\kappa_{1,n}, \ldots, \kappa_{p, n})$ is a diagonal matrix where $\kappa_{1, n} \geq \cdots \geq \kappa_{p,n}$ are the eigenvalues of $\bPsi_n$ with $\sum_{i=1}^p \kappa_{i,n} = 1$.
%    Define
%    $\kappa_{i,n} = \lambda_i (\bPsi_n) / \{\mytr(\bPsi_n^2) \}^{1/2}$ for $i \in \{1, \ldots, p\}$
%    and $\kappa_{i,n} = 0$ for $i > p$.

    Let $\nu$ be a possible asymptotic distribution of
$
    (
         \bxi_p^\myT
         \bPsi_n
         \bxi_p
            - 
            \mytr(\bPsi_n)
            )
            /
        \{
            2\mytr(
            \bPsi_n^2
        )
    \}^{1/2}
    $.
    We show that $\nu$ can be represented in the form of \eqref{eq:representation}.
    Note that there is a subsequence of $\{n\}$ along which
$
    (
         \bxi_p^\myT
         \bPsi_n
         \bxi_p
            - 
            \mytr(\bPsi_n)
            )
            /
        \{
            2\mytr(
            \bPsi_n^2
        )
    \}^{1/2}
    $
    converges weakly to $\nu$.
Denote $\kappa_{i,n} = \lambda_i (\bPsi_n) / \{\mytr(\bPsi_n^2) \}^{1/2}$, $i = 1, 2, \ldots$.
By Cantor's diagonalization trick (see, e.g., \cite{Simon2015RealAnalysis}, Section 1.5),
there exists a further subsequence along which
$
    \lim_{n\to\infty}
            \kappa_{i,n}
            = \kappa_i
            $, $i = 1, 2, \ldots$,
where $\kappa_1, \kappa_2, \ldots$ are real numbers in $[0,1]$.
Then Lemma \ref{lemma:713aa} implies that
along this further subsequence,
{\small
\begin{align*}%\label{eq:toprove13}
    \left\|
    \mathcal L
    \left(
        \frac{
         \bxi_p^\myT
         \bPsi_n
         \bxi_p
            - 
            \mytr(\bPsi_n)
        }
        {
        \{
            2\mytr(
            \bPsi_n^2
        )
    \}^{1/2}
}
            \right)
            -\mathcal L
            \left(
                (1- \sum_{i=1}^\infty \kappa_i^2)^{1/2} \xi_0
                +
    2^{-1/2} 
            \sum_{i=1}^\infty \kappa_{i} (\zeta_i^2 - 1)
            \right)
            \right\|_3
            \to 0.
\end{align*}
}%
Thus, $\nu =
            \mathcal L
            \left(
                (1- \sum_{i=1}^\infty \kappa_i^2)^{1/2} \xi_0
                +
    2^{-1/2} 
            \sum_{i=1}^\infty \kappa_{i} (\zeta_i^2 - 1)
            \right)
$.

Now we prove that for any sequence of positve numbers $\{\kappa_i\}_{i=1}^\infty$ such that $\sum_{i=1}^\infty \kappa_i^2 \in [0,1]$,
there exits a sequence $\{\bPsi_n\}_{n=1}^\infty$ such that
    $
            \mathcal L
            \left(
                (1- \sum_{i=1}^\infty \kappa_i^2)^{1/2} \xi_0
                +
    2^{-1/2} 
            \sum_{i=1}^\infty \kappa_{i} (\zeta_i^2 - 1)
            \right)
    $ 
    is the asymptotic distribution of
        $
        \left(
             \bxi_p^\myT
         \bPsi_n
         \bxi_p
            - 
            \mytr(\bPsi_n)
        \right)
        /{
        \{
            2\mytr(
            \bPsi_n^2
        )
    \}^{1/2}
    }
    $.
    To construct the sequence $\{\bPsi_n\}_{n=1}^\infty$, we take $p = n^2$ and
    let $\bPsi_n = \mydiag(\kappa_{1,n}, \ldots, \kappa_{n^2,n})$, where $\kappa_{i,n} = \kappa_i$ for $i \in \{1, \ldots,  n\}$ and
    $\kappa_{i,n} = \{(1- \sum_{i=1}^n \kappa_i^2)/(n^2 - n)\}^{1/2}$ for $i  \in \{n + 1, \ldots, n^2\}$.
    Then $\sum_{i=1}^p \kappa_{i,n}^2 = 1$ and $\lim_{n\to \infty} \kappa_{i,n} = \kappa_i$, $i= 1, 2, \ldots$.
    It is straightforward to show that
        $
        \left(
             \bxi_p^\myT
         \bPsi_n
         \bxi_p
            - 
            \mytr(\bPsi_n)
        \right)
        /{
        \{
            2\mytr(
            \bPsi_n^2
        )
    \}^{1/2}
    }
    $
    converges weakly to 
    $
            \mathcal L
            \left(
                (1- \sum_{i=1}^\infty \kappa_i^2)^{1/2} \xi_0
                +
    2^{-1/2} 
            \sum_{i=1}^\infty \kappa_{i} (\zeta_i^2 - 1)
            \right)
    $.
    This completes the proof.

\section{Proof of Theorem \ref{thm:final_thm}}
In this section, we provide the proof of Theorem \ref{thm:final_thm}.

If the conclusion holds for the case that $\epsilon_{1,1}^*$ is a standard normal random variable, 
    then Lemma \ref{lemma:universality_rad} implies that it also holds for the case of Rademacher random variable.
Hence without loss of generality, we assume that $\epsilon_{1,1}^*$ is a standard normal random variable.

%We begin with some concepts and notations that are useful for our proof.
We begin with some basic notations and facts that are useful for our proof.
Denote by $\bPsi_n = \BU \bLambda \BU^\myT$ the spectral decomposition of $\bPsi_n$ where $\bLambda = \mydiag\{\lambda_{1}(
    \bPsi_n
), \ldots, \lambda_{p}(
\bPsi_n
)\}$ and $\BU$ is an orthogonal matrix.
Let $\BU_{m} \in \mathbb R^{p\times m}$ be the first $m$ columns of $\BU$ and
$\tilde \BU_{m} \in \mathbb R^{p \times (p-m)}$ be the last $p-r$ columns of $\BU$.
Then we have
$\BU_m^\myT \tilde \BU_m = \BO_{m \times (p-m)}$
and
$
\BU_{m}
\BU_{m}^{\myT}
+
\tilde \BU_{m}
\tilde \BU_{m}^{\myT}
= \BI_p
$.
%Then
%the matrix
%$
%\tilde \BU_{m}
%\tilde \BU_{m}^{\myT}
%$
%is the projection matrix onto the orthogonal complement of the column space of $\BU_{m}$.
For $k =1, 2$,
we have
{\small
\begin{align}\label{eq:bound_mp1_eigenvalue}
    \lambda_1 (\tilde \BU_m^{\myT} \bar \bSigma_k \tilde \BU_m)
    \leq
    n_k
    \lambda_1 (
      \tilde \BU_m^{\myT}
      \bPsi_n
\tilde \BU_m
)
    =
    n_k
    \lambda_{m+1}(
    \bPsi_n
    )
    .
\end{align}
}%

To prove the conclusion, we apply the subsequence trick.
That is, for any subsequence of $\{n\}$, we prove that there is a further subsequence along which the conclusion holds.
For any subsequence of $\{n\}$,
by Cantor's diagonalization trick (see, e.g., \cite{Simon2015RealAnalysis}, Section 1.5),
there exists a further subsequence along which
{\small
\begin{align}\label{eq:subsequence_limit}
    \frac{
        \lambda_{i}
        \left(
            \bPsi_n
    \right)
}
    {
        \left\{
        \mytr
        \left(
            \bPsi_n^2
        \right)
    \right\}^{1/2}
    }
    \to \kappa_{i},
    \quad
    i = 1, 2, \ldots
    .
\end{align}
}%
Thus, we only need to prove the conclusion with the additional condition \eqref{eq:subsequence_limit}.
Without loss of generality, we assume that \eqref{eq:subsequence_limit} holds for the full sequence $\{n\}$.

From Fatou's lemma, we have $\sum_{i=1}^\infty \kappa_{i}^2 \leq 1$.
Now we claim that there exists a sequence  $\{r_n^*\}$ of non-decreasing integers which tends to infinity such that 
{\small
\begin{align}\label{eq:wahaha}
    \frac{
        \sum_{i=1}^{r_n^*} 
        \lambda_{i}^2
        \left(
            \bPsi_n
    \right)
    }
    {
        \mytr
        \left(
            \bPsi_n^2
        \right)
    }
    \to \sum_{i=1}^{\infty} \kappa_{i}^2
    .
\end{align}
}%
We defer the proof of this fact to Lemma \ref{lemma:low_sequence} in Section \ref{sec:deferred}.

Fix a positive integer $r$.
%Without loss of generality, we assume $r_n^* > r$.
Then for large $n$, we have $r_n^* > r$.
Let $\check \BU_{r}$ be the $(r+1)$th to the $r_n^*$th columns of $\BU$.
Then $\check \BU_{r} \in \mathbb R^{p \times (r_n^* - r)}$ is a column orthogonal matrix such that
$
\check \BU_{r}
\check \BU_{r}^{\myT}
=
\BU_{r_n^*}
\BU_{r_n^*}^{\myT}
-
\BU_{r}
\BU_{r}^{\myT}
$.
We can decompose the identity matrix into the sum of three mutually orthogonal projection matrices as $\BI_p = \BU_{r} \BU_{r}^{\myT}
+
\check \BU_{r} \check \BU_{r}^{\myT}
+
\tilde \BU_{r_n^*} \tilde \BU_{r_n^*}^{\myT}
$.
Then we have
$
    T_{\mathrm{CQ}}(
    E^*
    ;
    \tilde \BX_1, \tilde \BX_2
    )
    =
    T_{1,r} + T_{2,r} + T_{3,r}
$, where
{\small
\begin{align*}
    T_{1,r} = &
    T_{\mathrm{CQ}}(
    E^*
    ;
    \tilde \BX_1 \BU_r, \tilde \BX_2 \BU_r
    ),
    \quad
    T_{2,r} = 
    T_{\mathrm{CQ}}(
    E^*
    ;
    \tilde \BX_1 \check \BU_r, \tilde \BX_2 \check \BU_r
    ),
    \quad
    T_{3,r} = 
    T_{\mathrm{CQ}}(
    E^*
    ;
    \tilde \BX_1 \tilde \BU_{r_n^*}, \tilde \BX_2 \tilde \BU_{r_n^*}
    )
    .
\end{align*}
}%

Let $\{\xi_i\}_{i=0}^\infty$ be a sequence of independent standard normal random variables.
We claim that 
%for any fixed $r$,
{\small
\begin{align}\label{eq:conclusion_T13}
    \left\|\mathcal L \left(
        \frac{T_{1,r} + T_{3,r}}{
            \sigma_{T,n}
        }
        \mid
        \tilde \BX_1,
        \tilde \BX_2
    \right) - 
    \mathcal L 
    \left(
            (1 - \sum_{i=1}^{\infty} \kappa_{i}^2)^{1/2}
            \xi_0
    +
    2^{-1/2}\sum_{i=1}^r \kappa_i (\xi_i^2-1)
    \right)
    \right\|_3
    \xrightarrow{P} 0 .
\end{align}
}%
The above quantity is a continuous function of $\tilde \BX_1$ and $\tilde \BX_2$ and is thus measurable.
We defer the proof of \eqref{eq:conclusion_T13} to Lemma \ref{lemma:T13} in Section \ref{sec:deferred}.
It is known that for random variables in $\mathbb R$, convergence in probability is metrizable; see, e.g., \cite{dudleyProbability}, Section 9.2.
%As a consequence,
As a standard property of metric space,
\eqref{eq:conclusion_T13} also holds when $r$ is replaced by certain $r_n$ with $r_n \to \infty$ and $ r_n \leq r_n^* $.
Also note that L\'evy's equivalence theorem and three-series theorem (see, e.g., \cite{dudleyProbability}, Theorem 9.7.1 and Theorem 9.7.3) implies that
$
    \sum_{i=1}^{r_n} \kappa_i (\xi_i^2-1)
$
converges weakly to 
$
    \sum_{i=1}^{\infty} \kappa_i (\xi_i^2-1)
$.
Thus,
%there exists a sequence $\{r_n\}$ such that $r_n \to \infty$, $ r_n \leq r_n^* $ and  
{\small
\begin{align}\label{eq:conclusion_T13_2}
    \left\|\mathcal L \left(
        \frac{T_{1,r_n} + T_{3,r_n}}{
            \sigma_{T,n}
        }
        \mid \tilde \BX_1, \tilde \BX_2
    \right) - 
    \mathcal L 
    \left(
            (1 - \sum_{i=1}^{\infty} \kappa_{i}^2)^{1/2}
            \xi_0
    +
    2^{-1/2}
    \sum_{i=1}^{\infty} \kappa_i (\xi_i^2-1)
    \right)
    \right\|_3
    \xrightarrow{P} 0 .
\end{align}
}%

Now we deal with $T_{2, r_n}$.
%derive an upper bound for the approximation error caused by removing the second term.
From Lemma \ref{lemma:variance},
we have
{\small
\begin{align*}
        \myE
        \left(
    \frac{
        T_{2, r_n}^2
        }
        {
            \sigma_{T,n}^2
        }
        \mid
        \tilde \BX_1,
        \tilde \BX_2
        \right)
        =&
    \frac{
        2 \mytr\left\{\left(\check \BU_{r_n}^{\myT}\bPsi_n \check \BU_{r_n}\right)^2 \right\}
    }{
    \sigma_{T,n}^2
}
+o_P(1)
=
        \sum_{i=1}^{r_n^*} 
    \frac{
        \lambda_{i}^2
        \left(\bPsi_n \right)
    }
    {
        \mytr
        \left(
            \bPsi_n^2
        \right)
    }
-
        \sum_{i=1}^{r_n} 
    \frac{
        \lambda_{i}^2
        (\bPsi_n)
    }
    {
        \mytr
        \left(
            \bPsi_n^2
        \right)
    }
+o_P(1)
.
\end{align*}
}%
From Fatou's lemma,
{\small
\begin{align*}
    \liminf_{n\to\infty}
        \sum_{i=1}^{r_n} 
    \frac{
        \lambda_{i}^2
        (\bPsi_n)
    }
    {
        \mytr
        \left(
            \bPsi_n^2
        \right)
    }
    \geq 
    \sum_{i=1}^\infty
    \liminf_{n\to\infty}
    \frac{
        \lambda_{i}^2
        (\bPsi_n)
    }
    {
        \mytr
        \left(
            \bPsi_n^2
        \right)
    }
    = \sum_{i=1}^{\infty} \kappa_{i}^2
    .
\end{align*}
}%
Combining the above inequality and \eqref{eq:wahaha} leads to
{\small
\begin{align*}
        \myE
        \left(
    \frac{
        T_{2, r_n}^2
        }
        {
            \sigma_{T,n}^2
        }
        \mid
        \tilde \BX_1,
        \tilde \BX_2
        \right)
        =&
o_P(1)
.
\end{align*}
}%
Then from Lemma \ref{lemma:add_error},
{\small
\begin{align*}
    \left\|
    \mathcal L \left(
        \frac{T_{\mathrm{CQ}}( E^* ; \tilde \BX_1, \tilde \BX_2 ) }{
            \sigma_{T,n}
        }
        \mid \tilde \BX_1, \tilde \BX_2
    \right) 
    -
    \mathcal L \left(
        \frac{T_{1,r_n} + T_{3,r_n}}{
            \sigma_{T,n}
        }
    \right) 
    \right\|_3
    \xrightarrow{P} 0 .
\end{align*}
}%
Combining the above equality and \eqref{eq:conclusion_T13_2} leads to
{\small
\begin{align}
    \label{eq:wuwuwuwu}
    \left\|
    \mathcal L \left(
        \frac{T_{\mathrm{CQ}}( E^* ; \tilde \BX_1, \tilde \BX_2 ) }{
            \sigma_{T,n}
        }
        \mid \tilde \BX_1, \tilde \BX_2
    \right) 
    -
    \mathcal L 
    \left(
            (1 - \sum_{i=1}^{\infty} \kappa_{i}^2)^{1/2}
            \xi_0
    +
    2^{-1/2}
    \sum_{i=1}^{\infty} \kappa_i (\xi_i^2-1)
    \right)
    \right\|_3
    \xrightarrow{P} 0 .
\end{align}
}%
Then the conclusion follows from \eqref{eq:wuwuwuwu} and Lemma \ref{lemma:713aa}.

\section{Proof of Corollary \ref{corollary:the2}}

    Using the subsequence trick, it suffices to prove the conclusion for a subsequence of $\{n\}$.
    Then from Corollary \ref{corollary:the},
    we can assume without loss of generality that
    {\small
    \begin{align}\label{eq:7142}
        \left\|
    \mathcal L \left(
        \frac{
         \bxi_p^\myT
         \bPsi_n
         \bxi_p
            - 
            \mytr(\bPsi_n)
        }{
        \left\{
            2\mytr\left(
            \bPsi_n^2
        \right)
    \right\}^{1/2}
    }
\right)
-
    \mathcal L 
    \left(
            (1 - \sum_{i=1}^{\infty} \kappa_{i}^2)^{1/2}
            \xi_0
    +
    2^{-1/2}
    \sum_{i=1}^{\infty} \kappa_i (\xi_i^2-1)
    \right)
    \right\|_3
    \to 0,
    \end{align}
}%
    where
$\{\xi_i \}_{i=0}^\infty$ is a sequence of independent standard normal random variables,
$\{\kappa_i\}_{i=1}^\infty$ is a sequence of positive numbers such that $\sum_{i=1}^\infty \kappa_i^2 \in [0, 1]$.
Let $\tilde F(\cdot)$ denote the cumulative distribution function of
$
            (1 - \sum_{i=1}^{\infty} \kappa_{i}^2)^{1/2}
            \xi_0
    +
    2^{-1/2}
    \sum_{i=1}^{\infty} \kappa_i (\xi_i^2-1)
$.
We claim that $\tilde F(\cdot)$
is continuous and strictly increasing on the interval $\{x \in \mathbb R: \tilde F(x) > 0 \}$.
We defer the proof of this fact to Lemma \ref{lemma:714} in Section \ref{sec:deferred}.
This fact, combined with \eqref{eq:7142}, leads to
{\small
\begin{align}\label{eq:714OhOh}
    \sup_{x\in \mathbb R}|G_n(x) - \tilde F(x) | = o(1),
    \quad
    G_n^{-1}(1-\alpha) = \tilde F^{-1} (1-\alpha) + o(1).
\end{align}
}%
Furthermore,
in view of Theorem \ref{thm:final_thm} and \eqref{eq:7142}, we have
{\small
\begin{align}\label{eq:714Oh}
    \frac{\hat F_{\mathrm{CQ}}^{-1} (1 - \alpha)}{\sigma_{T,n}}
=\tilde F^{-1} (1-\alpha) + o_P(1)
.
\end{align}
}%

    We have
    {\small
    \begin{align*}
        &
    \Pr \left\{
    T_{\mathrm{CQ}} (\BX_1, \BX_2) > \hat F_{\mathrm{CQ}}^{-1} (1-\alpha)
\right\}
\\
        =&
    \Pr \left\{
        \frac{
        T_{\mathrm{CQ}} (\BX_1, \BX_2) - \|\mu_1 - \mu_2\|^2
        }{
        \sigma_{T,n}
    }
    +
    \tilde{F}^{-1} (1-\alpha)
    -
    \frac{
        \hat F_{\mathrm{CQ}}^{-1} (1-\alpha) 
    }
    {\sigma_{T,n}}
    >
    \tilde{F}^{-1} (1-\alpha)
    -
    \frac{
     \|\mu_1 - \mu_2\|^2
    }{
        \sigma_{T,n}
}
\right\}
    \end{align*}
}%
    Note that
    {\small
    \begin{align*}
        \frac{
        T_{\mathrm{CQ}} (\BX_1, \BX_2)
        - 
        \|\mu_1 - \mu_2\|^2
        }{
        \sigma_{T,n}
    }
        =
        \frac{
        T_{\mathrm{CQ}} (\BY_1, \BY_2)
        }{
        \sigma_{T,n}
    }
        +
        \frac{
        2(\mu_1 - \mu_2)^\myT (\bar Y_1 - \bar Y_2)
        }{
        \sigma_{T,n}
    }
    =
        \frac{
        T_{\mathrm{CQ}} (\BY_1, \BY_2)
        }{
        \sigma_{T,n}
    }
        +
        o_P(1)
        ,
    \end{align*}
}%
    where the last equality holds since
    {\small
\begin{align*}
    \myVar\left(
        \frac{2(\mu_1 - \mu_2)^\myT (\bar Y_1 - \bar Y_2)}{\sigma_{T,n}}
    \right)
    =
    (1+o(1))
    \frac{
    2 (\mu_1 - \mu_2)^\myT \bPsi_n (\mu_1 - \mu_2)
    }{
    \mytr(\bPsi_n^2)
}
=o(1)
.
\end{align*}
}%
Then it follows from Theorem \ref{thm:universality_TCQ}, equality \eqref{eq:714Oh} and the fact that $\tilde F(\cdot)$ is continuous that
{\small
\begin{align*}
    \sup_{x\in \mathbb R}
    \left|
    \Pr\left(
        \frac{
        T_{\mathrm{CQ}} (\BX_1, \BX_2) - \|\mu_1 - \mu_2\|^2
        }{
        \sigma_{T,n}
    }
    +
    \tilde{F}^{-1} (1-\alpha)
    -
    \frac{
        \hat F_{\mathrm{CQ}}^{-1} (1-\alpha) 
    }
    {\sigma_{T,n}}
    \leq x
\right)
-\tilde F (x)
\right|
= o(1).
\end{align*}
}%
Therefore,
{\small
\begin{align*}
    \Pr \left\{
    T_{\mathrm{CQ}} (\BX_1, \BX_2) > \hat F_{\mathrm{CQ}}^{-1} (1-\alpha)
\right\}
=
&
1 - \tilde F\left(
    \tilde F^{-1}(1-\alpha) - \frac{\|\mu_1 - \mu_2\|^2}{\sigma_{T,n}}
\right)
+o(1)
\\
=
&
1 - G_n\left(
    G_n^{-1}(1-\alpha) - \frac{\|\mu_1 - \mu_2\|^2}{
        \left\{ 2 \mytr (\bPsi_n^2) \right\}^{1/2}
}
\right)
+o(1)
,
\end{align*}
}%
where the last equality follows from \eqref{eq:714OhOh}.
This completes the proof.

\section{Deferred proofs} \label{sec:deferred}
In this section, we provide proofs of some intermediate results in our proofs of main results.
Some results in this section are also used in the main text.
\begin{lemma}\label{lemma:low_sequence}
    Suppose the conditions of Theorem \ref{thm:final_thm} hold.
    Furthermore, suppose the condition \eqref{eq:subsequence_limit} holds.
    Then there exists
a sequence 
$\{r_n^*\}$
of non-decreasing integers which tends to infinity such that 
\eqref{eq:wahaha} holds.
\end{lemma}
\begin{proof}
    For any fixed positive integer $m$,
    we have
    {\small
    \begin{align*}
    \frac{
    \sum_{i=1}^{m}
    \lambda_{i}^2(\bPsi_n)
    }
    {
        \mytr(\bPsi_n^2)
    }
    \to \sum_{i=1}^{m} \kappa_{i}^2
    .
    \end{align*}
}%
Therefore, there exists an $n_m$ such that for any $n > n_m$,
{\small
    \begin{align*}
    \left|
    \frac{
    \sum_{i=1}^{m}
    \lambda_{i}^2(\bPsi_n)
    }
    {
        \mytr(\bPsi_n^2)
    }
    - \sum_{i=1}^{m} \kappa_{i}^2
    \right|
    <\frac{ 1}{m}
    .
    \end{align*}
}%
We can without loss of generality and assume $n_1 < n_2 <\cdots$ since otherwise we can enlarge some $n_m$.
    Define $r_n^* = m$ for $ n_m < n \leq n_{m+1} $, $m =1, 2, \ldots$ and $r_n^* = 1$ for $n \leq n_1$.
    By definition, $r_n^*$ is non-decreasing and $\lim_{n\to \infty} r_n^* = \infty$.
    Also, for any $n > n_1$,
    {\small
    \begin{align*}
    \left|
    \frac{
    \sum_{i=1}^{ r_n^*}
    \lambda_{i}^2(\bPsi_n)
    }
    {
        \mytr(\bPsi_n^2)
    }
    - \sum_{i=1}^{r_n^*} \kappa_{i}^2
    \right|
    <\frac{ 1}{r_n^*}
    .
    \end{align*}
}%
    Thus,
    {\small
    \begin{align*}
        \lim_{n \to \infty}
    \left|
    \frac{
    \sum_{i=1}^{r_n^*}
    \lambda_{i}^2(\bPsi_n)
    }
    {
        \mytr(\bPsi_n^2)
    }
    -
    \sum_{i=1}^{r_n^*} \kappa_{i}^2
    \right|
    =
    0
    .
    \end{align*}
}%
The conclusion follows from the above limit and the fact that
$
    \sum_{i=1}^{r_n^*} \kappa_{i}^2  \to
    \sum_{i=1}^{\infty} \kappa_{i}^2 
$.
\end{proof}

\begin{lemma}\label{lemma:small_lemma}
Suppose Assumption \ref{assumption:wangwang} holds.
Then for any $p\times p$ positive semi-definite matrix $\BB$ and $k =1 , 2$,
{\small
%\begin{align*}
%    \myE \{(\tilde X_{k, i}^\myT \BB \tilde X_{k,i})^2\}
%    \leq
%    \frac{\tau}{2}
%    \{\mytr(\BB \bSigma_{k, 2i-1}) \}^2
%    +
%    \frac{\tau}{2}
%    \{\mytr(\BB \bSigma_{k, 2i}) \}^2
%    .
%\end{align*}
}%
    {\small
\begin{align*}
    \myE \{(\tilde X_{k, i}^\myT \BB \tilde X_{k,i})^2 \}
    \leq 
    \tau
\{
\myE (\tilde X_{k, i}^\myT \BB \tilde X_{k,i})
\}^2
    .
\end{align*}
}%
\end{lemma}
\begin{proof}
    We have
    {\small
    \begin{align*}
        \myE \{(\tilde X_{k, i}^\myT \BB \tilde X_{k, i})^2 \}
    %=&
    %\frac{1}{4}
    %\myE ((Y_{k, 2i} - Y_{k, 2i-1})^\myT \BB (Y_{k, 2i} - Y_{k, 2i-1}) )^2
    %\\
    =&
    \frac{1}{16}
        \myE \{(Y_{k, 2i}^\myT \BB Y_{k, 2i} )^2\}
    +
    \frac{1}{16}
        \myE \{(Y_{k, 2i-1}^\myT \BB Y_{k, 2i-1} )^2\}
        \\
     &
    +
    \frac{1}{8}
        \myE \{
            (Y_{k, 2i}^\myT \BB Y_{k, 2i} )
            (Y_{k, 2i-1}^\myT \BB Y_{k, 2i-1} )
        \}
    +
    \frac{1}{4}
         \myE \{(Y_{k, 2i}^\myT \BB Y_{k, 2i-1} )^2\}
    \\
    \leq &
    \frac{\tau}{16}
    \{\mytr(\BB \bSigma_{k, 2i}) \}^2
    +
    \frac{\tau}{16}
    \{\mytr(\BB \bSigma_{k, 2i-1}) \}^2
        \\
     &
    +
    \frac{1}{8}
        \mytr ( \BB \bSigma_{k, 2i} )
        \mytr ( \BB \bSigma_{k, 2i-1} )
    +
    \frac{1}{4}
    \mytr (
    \BB \bSigma_{k, 2i} 
    \BB \bSigma_{k,2i-1} 
    )
    ,
    \end{align*}
}%
    where the last inequality follows from Assumption \ref{assumption:wangwang}.
    %From Cauchy-Schwarz inequality,
    %we have
    Note that
    {\small
    \begin{align*}
    \mytr (
    \BB \bSigma_{k, 2i} 
    \BB \bSigma_{k,2i-1} 
    )
    =&
    \mytr \{
        (\BB^{1/2} \bSigma_{k, 2i} \BB^{1/2})
        (\BB^{1/2} \bSigma_{k,2i-1} \BB^{1/2})
\}
\\
\leq&
    \mytr(\BB^{1/2} \bSigma_{k, 2i} \BB^{1/2}) 
    \mytr(\BB^{1/2} \bSigma_{k, 2i-1}\BB^{1/2}) 
    \\
=&
    \mytr(\BB \bSigma_{k, 2i}) 
    \mytr(\BB \bSigma_{k, 2i-1}) 
    .
    \end{align*}
}%
    It follows from the above two inequalities and the condition $\tau \geq 3$ that
    {\small
    \begin{align*}
        \myE \{(\tilde X_{k, i}^\myT \BB \tilde X_{k,i})^2 \}
        \leq 
        \frac{\tau}{16}
    \{\mytr(\BB \bSigma_{k, 2i}) 
    +
    \mytr(\BB \bSigma_{k, 2i-1}) \}^2
    =
    \tau
    \{
    \myE (\tilde X_{k, i}^\myT \BB \tilde X_{k,i})
    \}^2
        .
    \end{align*}
}%
    This completes the proof.
\end{proof}

\begin{lemma}\label{lemma:variance}
    Suppose Assumptions \ref{assumption:n}, \ref{assumption:wangwang} and \ref{assumption7} hold, and $\sigma_{T,n}^2 > 0$ for all $n$.
Let $\{\BB_n \}$  be a sequence of matrices where $\BB_n \in \mathbb R^{p \times m_n }$ is column orthogonal and the column number $m_n \leq p$.
Let $E^* = (\epsilon^*_{1,1}, \ldots, \epsilon^*_{1, m_1}, \epsilon^*_{2,1}, \ldots, \epsilon^*_{2, m_2})^\myT$, where $\epsilon^*_{k,i}$, $i = 1, \ldots, m_k$, $k = 1, 2$, are independent random variables with $\myE (\epsilon^*_{k,i}) = 0 $ and $\myVar (\epsilon^*_{k,i}) = 1$.
%Suppose the singular values of $\BB_n$ is bounded by an absolute constant, $n =1 ,2 , \ldots$.
Then as $n \to \infty$,
{\small
\begin{align*}
    &
    \frac{
    \myVar\{
T_{\mathrm{CQ}}(
E^* ;
\tilde \BX_1 \BB_n, \tilde \BX_2 \BB_n
)
    \mid \tilde \BX_1, \tilde \BX_2 \}
    }{
    \sigma_{T,n}^2
}
    =
    \frac{2 \mytr\{(\BB_n^\myT \bPsi_n \BB_n)^2\}}{ \sigma_{T,n}^2 }
    +o_P(1).
\end{align*}
}%
\end{lemma}
\begin{proof}
    We have
    {\small
    \begin{align*}
    &
    \myVar\{
T_{\mathrm{CQ}}(
E^*;
\tilde \BX_1 \BB_n, \tilde \BX_2 \BB_n
) 
    \mid \tilde \BX_1, \tilde \BX_2 \}
    \\
    =&
    \sum_{k=1}^2
        \frac{4 
            \sum_{i=1}^{m_k}
            \sum_{j=i+1}^{m_k}
            (\tilde X_{k, i}^\myT \BB_n \BB_n^\myT \tilde X_{k,j})^2
        }{m_k^2 (m_k - 1)^2 
     } 
     +
\frac{4
            \sum_{i=1}^{m_1}
            \sum_{j=1}^{m_2}
            (\tilde X_{1, i}^\myT \BB_n \BB_n^\myT \tilde X_{2,j })^2
        }{m_1^2 m_2^2  
    }
.
\end{align*}
}%

Fisrt we deal with $
\sum_{i=1}^{m_k}
\sum_{j=i+1}^{m_k}
(\tilde X_{k, i}^\myT 
\BB_n \BB_n^\myT
\tilde X_{k,j})^2
$, $k = 1, 2$.
We have
{\small
\begin{align*}
    &
    \myE 
    \left\{
            \sum_{i=1}^{m_k}
            \sum_{j=i+1}^{m_k}
            (\tilde X_{k, i}^\myT 
\BB_n \BB_n^\myT
            \tilde X_{k,j})^2
    \right\}
    \\
    =&
    \frac{1}{16}
            \sum_{i=1}^{m_k}
            \sum_{j=i+1}^{m_k}
            \mytr (
            \BB_n^\myT (\bSigma_{k, 2j-1}+\bSigma_{k, 2j})
            \BB_n \BB_n^\myT (\bSigma_{k, 2i-1}+ \bSigma_{k, 2i})
            \BB_n 
            )
            \\
    =&
    \frac{1}{32}
        \left(
            \sum_{i=1}^{2m_k}
            \sum_{j=1}^{2m_k}
            \mytr (
            \BB_n^\myT \bSigma_{k, j}
            \BB_n \BB_n^\myT \bSigma_{k, i}
            \BB_n 
            )
            -
            \sum_{i=1}^{m_k}
            \mytr\{ (
            \BB_n^\myT (\bSigma_{k, 2i-1}+ \bSigma_{k, 2i})
            \BB_n 
            )^2
        \}
        \right)
.
\end{align*}
}%
Hence
{\small
\begin{align}
    &
    \left|
    \myE 
    \left\{
            \sum_{i=1}^{m_k}
            \sum_{j=i+1}^{m_k}
            (\tilde X_{k, i}^\myT 
\BB_n \BB_n^\myT
            \tilde X_{k,j})^2
    \right\}
    -
    \frac{1}{32}
    n_k^2
    \mytr \{ (\BB_n^\myT \bar \bSigma_k \BB_n )^2 \}
    \right|
    \notag
    \\
    \leq&
    \frac{1}{16}
     n_k \mytr(\BB_n^\myT \bSigma_{k,n_k} \BB_n \BB_n^\myT \bar \bSigma_{k} \BB_n)
    +
    \frac{1}{16}
\sum_{i=1}^{n_k}\mytr \{ (\BB_n^\myT \bSigma_{k,i}  \BB_n)^2 \}
\notag
\\
    \leq&
    \frac{1}{16}
    \left\{
n_k^2 \mytr (\bar \bSigma_{k}^2)
\mytr (\bSigma_{k,n_k}^2)
\right\}^{1/2}
    +
    \frac{1}{16}
\sum_{i=1}^{n_k}\mytr (\bSigma_{k,i}^2) 
\notag
\\
    =& o\left(
        n_k^2 \mytr ( \bar \bSigma_{k}^2 )
    \right)
    ,
    \label{eq:71}
\end{align}
}%
where the last equality follows from Assumption \ref{assumption7}.

Now we compute the variance of $\sum_{i=1}^{m_k}
            \sum_{j=i+1}^{m_k}
            (\tilde X_{k, i}^\myT 
\BB_n \BB_n^\myT
            \tilde X_{k,j})^2
            $.
Note that
{\small
\begin{align*}
    &
    \left[
        \myE
    \left\{
            \sum_{i=1}^{m_k}
            \sum_{j=i+1}^{m_k}
            (\tilde X_{k, i}^\myT \BB_n \BB_n^\myT \tilde X_{k,j})^2
        \right\}
    \right]^2
    \\
        =&
            \sum_{i=1}^{m_k}
            \sum_{j=i+1}^{m_k}
            [
            \myE
            \{
            (\tilde X_{k, i}^\myT
                \BB_n \BB_n^\myT
            \tilde X_{k,j})^2
        \}
    ]^2
            +
            2
            \sum_{i=1}^{m_k}
            \sum_{j=i+1}^{m_k}
            \sum_{\ell=j+1}^{m_k}
            \myE \{(\tilde X_{k, i}^\myT
                \BB_n \BB_n^\myT
            \tilde X_{k,j})^2\}
            \myE \{(\tilde X_{k, i}^\myT
                \BB_n \BB_n^\myT
            \tilde X_{k,\ell})^2\}
            \\
    &
            +
            2
            \sum_{i=1}^{m_k}
            \sum_{j=i+1}^{m_k}
            \sum_{\ell=j+1}^{m_k}
            \myE \{(\tilde X_{k, i}^\myT
                \BB_n \BB_n^\myT
            \tilde X_{k,j})^2\}
            \myE \{(\tilde X_{k, j}^\myT
                \BB_n \BB_n^\myT
            \tilde X_{k,\ell})^2\}
            \\
    &
            +
            2
            \sum_{i=1}^{m_k}
            \sum_{j=i+1}^{m_k}
            \sum_{\ell=j+1}^{m_k}
            \myE \{(\tilde X_{k, i}^\myT
                \BB_n \BB_n^\myT
            \tilde X_{k,\ell})^2\}
            \myE \{(\tilde X_{k, j}^\myT
                \BB_n \BB_n^\myT
            \tilde X_{k,\ell})^2\}
        \\
    & 
            +
            2
            \sum_{i=1}^{m_k}
            \sum_{j=i+1}^{m_k}
            \sum_{\ell=j+1}^{m_k}
            \sum_{r=\ell+1}^{m_k}
            \myE \{(\tilde X_{k, i}^\myT 
                \BB_n \BB_n^\myT
            \tilde X_{k,j})^2\}
            \myE \{(\tilde X_{k, \ell}^\myT 
                \BB_n \BB_n^\myT
            \tilde X_{k,r})^2\}
            \\
    &
            +2
            \sum_{i=1}^{m_k}
            \sum_{j=i+1}^{m_k}
            \sum_{\ell=j+1}^{m_k}
            \sum_{r=\ell+1}^{m_k}
            \myE \{(\tilde X_{k, i}^\myT
                \BB_n \BB_n^\myT
            \tilde X_{k,\ell})^2\}
            \myE \{(\tilde X_{k, j}^\myT
                \BB_n \BB_n^\myT
            \tilde X_{k,r})^2\}
            \\
    &
            +2
            \sum_{i=1}^{m_k}
            \sum_{j=i+1}^{m_k}
            \sum_{\ell=j+1}^{m_k}
            \sum_{r=\ell+1}^{m_k}
            \myE \{(\tilde X_{k, i}^\myT 
                \BB_n \BB_n^\myT
            \tilde X_{k,r})^2\}
            \myE \{(\tilde X_{k, j}^\myT 
                \BB_n \BB_n^\myT
            \tilde X_{k,\ell})^2\}
            .
\end{align*}
}%
We denote the above $7$ terms by $C_{1}, \ldots, C_{7}$.
On the other hand,
{\small
\begin{align*}
    &
    \left\{
            \sum_{i=1}^{m_k}
            \sum_{j=i+1}^{m_k}
            (\tilde X_{k, i}^\myT \BB_n \BB_n^\myT \tilde X_{k,j})^2
        \right\}^2
        \\
    =&
            \sum_{i=1}^{m_k}
            \sum_{j=i+1}^{m_k}
            (\tilde X_{k, i}^\myT
                \BB_n \BB_n^\myT
            \tilde X_{k,j})^4
            +
            2
            \sum_{i=1}^{m_k}
            \sum_{j=i+1}^{m_k}
            \sum_{\ell=j+1}^{m_k}
            (\tilde X_{k, i}^\myT
                \BB_n \BB_n^\myT
            \tilde X_{k,j})^2
            (\tilde X_{k, i}^\myT
                \BB_n \BB_n^\myT
            \tilde X_{k,\ell})^2
            \\
    &
            +
            2
            \sum_{i=1}^{m_k}
            \sum_{j=i+1}^{m_k}
            \sum_{\ell=j+1}^{m_k}
            (\tilde X_{k, i}^\myT
                \BB_n \BB_n^\myT
            \tilde X_{k,j})^2
            (\tilde X_{k, j}^\myT
                \BB_n \BB_n^\myT
            \tilde X_{k,\ell})^2
            \\
    &
            +
            2
            \sum_{i=1}^{m_k}
            \sum_{j=i+1}^{m_k}
            \sum_{\ell=j+1}^{m_k}
            (\tilde X_{k, i}^\myT
                \BB_n \BB_n^\myT
            \tilde X_{k,\ell})^2
            (\tilde X_{k, j}^\myT
                \BB_n \BB_n^\myT
            \tilde X_{k,\ell})^2
        \\
    & 
            +
            2
            \sum_{i=1}^{m_k}
            \sum_{j=i+1}^{m_k}
            \sum_{\ell=j+1}^{m_k}
            \sum_{r=\ell+1}^{m_k}
            (\tilde X_{k, i}^\myT 
                \BB_n \BB_n^\myT
            \tilde X_{k,j})^2
            (\tilde X_{k, \ell}^\myT 
                \BB_n \BB_n^\myT
            \tilde X_{k,r})^2
            \\
    &
            +2
            \sum_{i=1}^{m_k}
            \sum_{j=i+1}^{m_k}
            \sum_{\ell=j+1}^{m_k}
            \sum_{r=\ell+1}^{m_k}
            (\tilde X_{k, i}^\myT
                \BB_n \BB_n^\myT
            \tilde X_{k,\ell})^2
            (\tilde X_{k, j}^\myT
                \BB_n \BB_n^\myT
            \tilde X_{k,r})^2
            \\
    &
            +2
            \sum_{i=1}^{m_k}
            \sum_{j=i+1}^{m_k}
            \sum_{\ell=j+1}^{m_k}
            \sum_{r=\ell+1}^{m_k}
            (\tilde X_{k, i}^\myT 
                \BB_n \BB_n^\myT
            \tilde X_{k,r})^2
            (\tilde X_{k, j}^\myT 
                \BB_n \BB_n^\myT
            \tilde X_{k,\ell})^2
            .
\end{align*}
}%
We denote the above $7$ terms by $T_{1}, \ldots, T_{7}$.
It can be seen that for $i =5, 6, 7$, $\myE (T_i) = C_i$.
Thus,
{\small
\begin{align*}
    \myVar
    \left\{
            \sum_{i=1}^{m_k}
            \sum_{j=i+1}^{m_k}
            (\tilde X_{k, i}^\myT \BB_n \BB_n^\myT \tilde X_{k,j})^2
        \right\}
        =
        \sum_{i=1}^4 \myE(T_i)
        -
        \sum_{i=1}^4 C_i
        \leq
        \sum_{i=1}^4 \myE(T_i).
\end{align*}
}%
Note that for $k = 1, 2$ and $i,j, \ell \in \{1, \ldots,  n_k \} $,
{\small
\begin{align*}
    \myE
    \left\{
        (\tilde X_{k, i}^\myT 
\BB_n \BB_n^\myT
        \tilde X_{k,j})^2
        (\tilde X_{k, i}^\myT
\BB_n \BB_n^\myT
        \tilde X_{k,\ell})^2
    \right\}
            \leq&
            \left[
                \myE
                \left\{
            (\tilde X_{k, i}^\myT
\BB_n \BB_n^\myT
            \tilde X_{k,j})^4
        \right\}
            \myE 
            \left\{
            (\tilde X_{k, i}^\myT
\BB_n \BB_n^\myT
            \tilde X_{k,\ell})^4
        \right\}
    \right]^{1/2}
    .
\end{align*}
}%
Consequently,
{\small
\begin{align*}
    \myVar
    \left\{
            \sum_{i=1}^{m_k}
            \sum_{j=i+1}^{m_k}
            (\tilde X_{k, i}^\myT \BB_n \BB_n^\myT \tilde X_{k,j})^2
        \right\}
    \leq
    \sum_{i=1}^4 \myE(T_i)
    \leq
    \sum_{i=1}^{m_k}
    \left[
        \sum_{j = 1}^{m_k}
            \left[
                \myE
                \left\{
            (\tilde X_{k, i}^\myT
\BB_n \BB_n^\myT
            \tilde X_{k,j})^4
        \right\}
    \right]^{1/2}
        \mathbf 1_{\{ j \neq i\}}
    \right]^2
    .
\end{align*}
}%
From Lemma \ref{lemma:small_lemma}, for $k = 1,2$ and distinct $i, j \in \{1, \ldots, n_k\}$,
{\small
\begin{align*}
    \myE
    \{
    (\tilde X_{k, i}^\myT 
\BB_n \BB_n^\myT
    \tilde X_{k,j})^4
\}
    \leq &
    \myE
    \{
    ( \tilde X_{k,j}^\myT
\BB_n \BB_n^\myT
    \tilde X_{k, i} \tilde X_{k, i}^\myT
\BB_n \BB_n^\myT
    \tilde X_{k,j})^2
\}
\\
    \leq &
    \frac{\tau^2}{256}
    \Big\{
        \mytr (
\BB_n^\myT
\bSigma_{k, 2i-1}
\BB_n \BB_n^\myT
\bSigma_{k, 2j-1}
\BB_n 
)
+
        \mytr (
\BB_n^\myT
\bSigma_{k, 2i-1}
\BB_n \BB_n^\myT
\bSigma_{k, 2j}
\BB_n
)
\\
         &+
        \mytr (
\BB_n^\myT
\bSigma_{k, 2i}
\BB_n \BB_n^\myT
\bSigma_{k, 2j-1}
\BB_n 
)
+
        \mytr (
\BB_n^\myT
\bSigma_{k, 2i}
\BB_n \BB_n^\myT
\bSigma_{k, 2j} 
\BB_n 
)
\Big\}^2
    .
\end{align*}
}%
Thus,
{\small
\begin{align*}
    \myVar
    \left\{
            \sum_{i=1}^{m_k}
            \sum_{j=i+1}^{m_k}
            (\tilde X_{k, i}^\myT \BB_n \BB_n^\myT \tilde X_{k,j})^2
        \right\}
        \leq&
    \frac{\tau^2}{128}
    n_k^2
    \sum_{i=1}^{n_k}
    \left\{
        \mytr (
\BB_n^\myT
\bSigma_{k, i}
\BB_n \BB_n^\myT
\bar \bSigma_{k}
\BB_n 
)
\right\}^2
\\
\leq&
    \frac{\tau^2}{128}
    n_k^2
    \sum_{i=1}^{n_k}
        \mytr \{ (
\BB_n^\myT
\bSigma_{k, i}
\BB_n 
)^2 \}
\mytr \{(
\BB_n^\myT
\bar \bSigma_{k}
\BB_n 
)^2\}
\\
= & 
o\left(
    n_k^4
    \left(
        \mytr (
\bar \bSigma_{k}
^2)
\right)^2
\right)
    ,
\end{align*}
}%
where the last equality follows from Assumption \ref{assumption7}.
Combining the above bound and \eqref{eq:71} leads to
{\small
\begin{align}\label{eq:72}
\sum_{i=1}^{m_k}
\sum_{j=i+1}^{m_k}
(\tilde X_{k, i}^\myT \BB_n \BB_n^\myT \tilde X_{k,j})^2
=
    \frac{1}{32}
    n_k^2
    \mytr \{ (\BB_n^\myT \bar \bSigma_k \BB_n )^2 \}
    +
    o
    \left(n_k^2 \mytr(\bar \bSigma_k^2)\right)
    .
\end{align}
}%

Now we deal with
$
    \sum_{i=1}^{m_1} \sum_{j=1}^{m_2} (\tilde X_{1,i}^\myT
\BB_n \BB_n^\myT
    \tilde X_{2,j})^2
    $.
We have
{\small
\begin{align*}
    \myE\left(
        \sum_{i=1}^{m_1} \sum_{j=1}^{m_2} (\tilde X_{1,i}^\myT
\BB_n \BB_n^\myT
    \tilde X_{2,j})^2
\right)
=
\frac{1}{16}
\sum_{i=1}^{m_1}
\sum_{j=1}^{m_2}
\mytr (
\BB_n^\myT (\bSigma_{1, 2i-1}+\bSigma_{1, 2i})\BB_n 
\BB_n^\myT (\bSigma_{2, 2j-1}+\bSigma_{2, 2j})\BB_n
)
.
\end{align*}
}%
Hence from Assumption \ref{assumption7},
{\small
\begin{align*}
    &
    \left|
    \myE\left(
        \sum_{i=1}^{m_1} \sum_{j=1}^{m_2} (\tilde X_{1,i}^\myT
\BB_n \BB_n^\myT
    \tilde X_{2,j})^2
\right)
-
\frac{1}{16}
n_1 n_2 \mytr( \BB_n^\myT \bar \bSigma_1 \BB_n 
\BB_n^\myT \bar \bSigma_2 \BB_n
)
\right|
\\
\leq &
\frac{1}{16}
n_2
\mytr( \BB_n^\myT \bSigma_{1,n_1} \BB_n  \BB_n^\myT \bar \bSigma_{2} \BB_n )
+
\frac{1}{16}
n_1
\mytr( \BB_n^\myT \bSigma_{2,n_2} \BB_n  \BB_n^\myT \bar \bSigma_{1} \BB_n )
\\
\leq &
\frac{1}{16}
\left\{
    n_2^2
    \mytr\{(\BB_n^\myT \bSigma_{1, n_1} \BB_n)^2\}
    \mytr\{(\BB_n^\myT \bar \bSigma_{2} \BB_n)^2\}
\right\}^{1/2}
+
\frac{1}{16}
\left\{
    n_1^2
    \mytr\{(\BB_n^\myT \bSigma_{2, n_2} \BB_n)^2\}
    \mytr\{(\BB_n^\myT \bar \bSigma_{1} \BB_n)^2\}
\right\}^{1/2}
\\
=&
o\left[
\left\{
    n_1^2
    n_2^2
    \mytr(\bar \bSigma_{1}^2 )
    \mytr(\bar \bSigma_{2}^2 )
\right\}^{1/2}
\right]
.
\end{align*}
}%
Now we compute the variance of 
$
        \sum_{i=1}^{m_1} \sum_{j=1}^{m_2} (\tilde X_{1,i}^\myT
\BB_n \BB_n^\myT
    \tilde X_{2,j})^2
$.
Note that
{\small
\begin{align*}
    &
    \left(
        \sum_{i=1}^{m_1} \sum_{j=1}^{m_2} (\tilde X_{1,i}^\myT
\BB_n \BB_n^\myT
    \tilde X_{2,j})^2
\right)^2
\\
    =&
\sum_{i=1}^{m_1}
\sum_{j=1}^{m_2}
(\tilde X_{1,i}^\myT
\BB_n \BB_n^\myT
    \tilde X_{2,j})^4
    +
    2
\sum_{i=1}^{m_1}
\sum_{\ell=1}^{m_2}
\sum_{r=\ell+1}^{m_2}
(\tilde X_{1,i}^\myT
\BB_n \BB_n^\myT
\tilde X_{2,\ell})^2
(\tilde X_{1,i}^\myT
\BB_n \BB_n^\myT
\tilde X_{2,r})^2
\\
     &+
    2
\sum_{i=1}^{m_1}
\sum_{j=i+1}^{m_1}
\sum_{\ell=1}^{m_2}
(\tilde X_{1,i}^\myT
\BB_n \BB_n^\myT
\tilde X_{2,\ell})^2
(\tilde X_{1,j}^\myT
\BB_n \BB_n^\myT
\tilde X_{2,\ell})^2
\\
     &+
    4
\sum_{i=1}^{m_1}
\sum_{j=i+1}^{m_1}
\sum_{\ell=1}^{m_2}
\sum_{r=\ell+1}^{m_2}
(\tilde X_{1,i}^\myT
\BB_n \BB_n^\myT
\tilde X_{2,\ell})^2
(\tilde X_{1,j}^\myT
\BB_n \BB_n^\myT
\tilde X_{2,r})^2
.
\end{align*}
}%
Hence
{\small
\begin{align*}
    &
    \myVar
    \left(
        \sum_{i=1}^{m_1} \sum_{j=1}^{m_2} (\tilde X_{1,i}^\myT
\BB_n \BB_n^\myT
    \tilde X_{2,j})^2
\right)
\\
\leq
&
\sum_{i=1}^{m_1}
\sum_{j=1}^{m_2}
\myE\{
(\tilde X_{1,i}^\myT
\BB_n \BB_n^\myT
    \tilde X_{2,j})^4
\}
    +
    2
\sum_{i=1}^{m_1}
\sum_{\ell=1}^{m_2}
\sum_{r=\ell+1}^{m_2}
\myE\{
(\tilde X_{1,i}^\myT
\BB_n \BB_n^\myT
\tilde X_{2,\ell})^2
(\tilde X_{1,i}^\myT
\BB_n \BB_n^\myT
\tilde X_{2,r})^2
\}
\\
     &+
    2
\sum_{i=1}^{m_1}
\sum_{j=i+1}^{m_1}
\sum_{\ell=1}^{m_2}
    \myE\{
(\tilde X_{1,i}^\myT
\BB_n \BB_n^\myT
\tilde X_{2,\ell})^2
(\tilde X_{1,j}^\myT
\BB_n \BB_n^\myT
\tilde X_{2,\ell})^2
\}
\\
    \leq&
    \sum_{i=1}^{m_1} 
    \left[\sum_{j=1}^{m_2} 
        \left[
\myE\left\{
(\tilde X_{1,i}^\myT
\BB_n \BB_n^\myT
    \tilde X_{2,j})^4
\right\}
\right]^{1/2}
    \right]^2
    +
    \sum_{j=1}^{m_2} 
    \left[\sum_{i=1}^{m_1}
        \left[
\myE\left\{
(\tilde X_{1,i}^\myT
\BB_n \BB_n^\myT
    \tilde X_{2,j})^4
\right\}
\right]^{1/2}
    \right]^2
    \\
    \leq&
    \frac{\tau^2}{128}
    n_2^2
    \sum_{i=1}^{n_1}
    \left\{
        \mytr(\BB_n^\myT \bSigma_{1,i} \BB_n \BB_n^\myT \bar \bSigma_{2} \BB_n)
    \right\}^2
    +
    \frac{\tau^2}{128}
    n_1^2
    \sum_{j=1}^{n_2}
    \left\{
        \mytr(\BB_n^\myT \bar \bSigma_{1} \BB_n \BB_n^\myT \bSigma_{2, j} \BB_n)
    \right\}^2
    \\
    \leq&
    \frac{\tau^2}{128}
    n_2^2
    \sum_{i=1}^{n_1}
        \mytr (\bSigma_{1,i}^2) 
        \mytr (
    \bar \bSigma_{2}^2) 
    +
    \frac{\tau^2}{128}
    n_1^2
    \sum_{j=1}^{n_2}
        \mytr (\bSigma_{2,i}^2) 
        \mytr (
    \bar \bSigma_{1}^2) 
    \\
    =&
    o\left(
        n_1^2 n_2^2
        \mytr(\bar \bSigma_1^2)
        \mytr(\bar \bSigma_2^2)
    \right)
    .
\end{align*}
}%
Thus,
{\small
\begin{align*}
        \sum_{i=1}^{m_1} \sum_{j=1}^{m_2} (\tilde X_{1,i}^\myT
\BB_n \BB_n^\myT
    \tilde X_{2,j})^2
    =&
\frac{1}{16}
n_1 n_2 \mytr( \BB_n^\myT \bar \bSigma_1 \BB_n 
\BB_n^\myT \bar \bSigma_2 \BB_n)
+
o_P\left[
\left\{
    n_1^2
    n_2^2
    \mytr(\bar \bSigma_{1}^2 )
    \mytr(\bar \bSigma_{2}^2 )
\right\}^{1/2}
\right]
.
\end{align*}
}%
It follows from \eqref{eq:72} and the above equality that
{\small
\begin{align*}
    &
    \sum_{k=1}^2
        \frac{4 
            \sum_{i=1}^{m_k}
            \sum_{j=i+1}^{m_k}
            (\tilde X_{k, i}^\myT \BB_n \BB_n^\myT \tilde X_{k,j})^2
        }{m_k^2 (m_k - 1)^2 
     } 
     +
\frac{4
            \sum_{i=1}^{m_1}
            \sum_{j=1}^{m_2}
            (\tilde X_{1, i}^\myT \BB_n \BB_n^\myT \tilde X_{2,j })^2
        }{m_1^2 m_2^2  
    }
    \\
    =&
    (1+o(1))
    \left[
    \sum_{k=1}^2
        \frac{
            2 \mytr\{(\BB_n^\myT \bar \bSigma_k \BB_n)^2\}
        }{ n_k^2 
     } 
     +
\frac{4
    \mytr(\BB_n^\myT \bar \bSigma_1 \BB_n  \BB_n^\myT \bar \bSigma_2 \BB_n ) 
        }{n_1 n_2
    }
\right]
\\
&
+
o_P\left[
    \sum_{k=1}^2
        \frac{
            \mytr(\bar \bSigma_k^2)
        }{ n_k^2 
     } 
    +
    \frac{
\left\{
    \mytr(\bar \bSigma_{1}^2 )
    \mytr(\bar \bSigma_{2}^2 )
\right\}^{1/2}
}{n_1 n_2}
\right]
.
\end{align*}
}%
Then the conclusion follows.
\end{proof}

\begin{lemma}\label{lemma:universality_rad}
    Suppose the conditions of Theorem \ref{thm:final_thm} hold.
    %Suppose Assumptions \ref{assumption:n}, \ref{assumption:wangwang} and \ref{assumption7} hold, and $\sigma_{T,n}^2 > 0$ for all $n$.
    Let $E = (\epsilon_{1,1}, \ldots, \epsilon_{1, m_1}, \epsilon_{2,1}, \ldots, \epsilon_{2, m_2})^\myT$, where $\epsilon_{k,i}$, $i = 1, \ldots, m_k$, $k = 1, 2$, are independent Rademacher random variables.
    Let $E_k^* = (\epsilon^*_{1,1}, \ldots, \epsilon^*_{1, m_1}, \epsilon^*_{2,1}, \ldots, \epsilon^*_{2, m_2})^\myT$, where $\epsilon^*_{k,i}$, $i = 1, \ldots, m_k$, $k = 1, 2$, are independent standard normal random variables.
    Then as $n \to \infty$,
    {\small
\begin{align*}
    \left\|
    \mathcal L \left(
        \frac{T_{\mathrm{CQ}}( E^* ; \tilde \BX_1, \tilde \BX_2 ) }{
            \sigma_{T,n}
        }
        \mid \tilde \BX_1, \tilde \BX_2
    \right) 
    -
    \mathcal L \left(
        \frac{T_{\mathrm{CQ}}( E ; \tilde \BX_1, \tilde \BX_2 ) }{
            \sigma_{T,n}
        }
        \mid \tilde \BX_1, \tilde \BX_2
    \right) 
    \right\|_3
    \xrightarrow{P} 0 .
\end{align*}
}%
\end{lemma}
\begin{proof}
%to replace Rademacher random variables by standard normal random variables.
%From Lemma \ref{lemma:variance},
    {\small
%\begin{align*}
%    &
%    \myVar\{
%T_{\mathrm{CQ}}(
%E_1, E_2;
%\tilde \BX_1, \tilde \BX_2
%)
%    \mid \tilde \BX_1, \tilde \BX_2 \}
%    =
%    \myVar\{
%T_{\mathrm{CQ}}(
%E_1^*, E_2^*;
%\tilde \BX_1, \tilde \BX_2
%)
%    \mid \tilde \BX_1, \tilde \BX_2 \}
%    =(1+o_P(1)) \sigma_{T,n}^2
%.
%\end{align*}
    }%
We apply Theorem \ref{thm:universality_GQF} conditioning on $\tilde \BX_1$ and $\tilde \BX_2$.
%And the randomness comes from $E_k$ and $E_k^*$ as random variables.
%In Theorem \ref{thm:universality_GQF},
Define 
{\small
\begin{align*}
\xi_{i} = 
\left\{
\begin{array}{ll}
    \epsilon_{1,i}& \text{for } i = 1, \ldots, m_1,
    \\
    \epsilon_{2, i-m_1}& \text{for }i = m_1 + 1, \ldots, m_1 + m_2,
\end{array}
\right.
\quad
\text{and}
\quad
\eta_{i} = 
\left\{
\begin{array}{ll}
    \epsilon_{1,i}^*& \text{for } i = 1, \ldots, m_1,
    \\
    \epsilon_{2, i-m_1}^*& \text{for }i = m_1 + 1, \ldots, m_1 + m_2.
\end{array}
\right.
\end{align*}
}%
Define
{\small
\begin{align*}
w_{i,j}(a, b) = 
\left\{
    \begin{array}{ll}
        \frac{2 a b \tilde X_{1, i}^\myT \tilde X_{1,j} }{m_1 (m_1 - 1)
    %\left\{\myVar(T_{\mathrm{CQ}}) \right\}^{1/2}
            \sigma_{T,n}
    } 
&
\text{for } 1\leq i< j \leq m_1
,
\\
\frac{-2 a b\tilde X_{1, i}^\myT \tilde X_{2,j - m_1 }}{m_1 m_2 
    %\left\{\myVar(T_{\mathrm{CQ}}) \right\}^{1/2}
    \sigma_{T,n}
} 
&
\text{for } 1\leq i \leq m_1,  m_1+1 \leq j \leq m_1 + m_2
,
\\
\frac{2ab\tilde X_{2, i - m_1 }^\myT \tilde X_{2,j - m_1 }}{m_2 (m_2 - 1) 
    %\left\{\myVar(T_{\mathrm{CQ}}) \right\}^{1/2}
    \sigma_{T,n}
} 
&
\text{for } m_1 + 1\leq i< j \leq m_1 + m_2
.
    \end{array}
\right.
\end{align*}
}%
%Then in Theorem \ref{thm:universality_GQF},
With the above definitions,
we have
$W(\xi_1, \ldots, \xi_n) = 
T_{\mathrm{CQ}}(E; \tilde \BX_1, \tilde \BX_2) 
/ \sigma_{T,n}
$
and
$W(\eta_1, \ldots, \eta_n) = 
T_{\mathrm{CQ}}(E^*; \tilde \BX_1, \tilde \BX_2) 
/ \sigma_{T,n}
$.
It can be easily seen that
Assumptions \ref{assumption:oh1} and \ref{assumption:oh2} hold.
By direct calculation, we have
{\small
\begin{align*}
    \sigma_{i,j}^2=
    \left\{
    \begin{array}{ll}
        \frac{4 
            (\tilde X_{1, i}^\myT \tilde X_{1,j})^2
        }{m_1^2 (m_1 - 1)^2 
%\myVar(T_{\mathrm{CQ}})
        \sigma_{T,n}^2
     } 
&
\text{for } 1\leq i< j \leq m_1
,
\\
\frac{4
            (\tilde X_{1, i}^\myT \tilde X_{2,j -m_1 })^2
        }{m_1^2 m_2^2  
%\myVar(T_{\mathrm{CQ}})
        \sigma_{T,n}^2
    }
&
\text{for } 1\leq i \leq m_1,  m_1+1 \leq j \leq m_1 + m_2
,
\\
\frac{4 
            (\tilde X_{2, i - m_1 }^\myT \tilde X_{2,j -m_1 })^2
    }{m_2^2 (m_2 - 1)^2 
%\myVar(T_{\mathrm{CQ}})
    \sigma_{T,n}^2
}
&
\text{for } m_1 + 1\leq i< j \leq m_1 + m_2
.
    \end{array}
        \right.
\end{align*}
}%
Hence
{\small
\begin{align*}
    \mathrm{Inf}_i
    %=
    %\sum_{i=1}^{k-1}
    %\sigma_{i,k}^2
    %+
    %\sum_{j=k+1}^n
    %\sigma_{k,j}^2
    =
    \left\{
    \begin{array}{ll}
        \frac{4 
            \sum_{j=1}^{m_1} (\tilde X_{1, i}^\myT \tilde X_{1,j})^2
            \mathbf 1_{\{ j \neq i\}}
        }{m_1^2 (m_1 - 1)^2 
        %\myVar(T_{\mathrm{CQ}})
\sigma_{T,n}^2
    } 
        +
\frac{4
        \sum_{j =1}^{m_2} (\tilde X_{1, i}^\myT \tilde X_{2, j})^2
}{m_1^2 m_2^2 
%\myVar(T_{\mathrm{CQ}})
\sigma_{T,n}^2
}
&
\text{for } 1\leq i \leq m_1
,
\\
        \frac{4 
            \sum_{j = 1}^{m_2} (\tilde X_{2, i-m_1}^\myT \tilde X_{2,j})^2
            \mathbf 1_{\{ j \neq i\} }
        }{m_2^2 (m_2 - 1)^2
    %\myVar(T_{\mathrm{CQ}})
\sigma_{T,n}^2
} 
        +
\frac{4
        \sum_{j =1}^{m_1} (\tilde X_{1, j}^\myT \tilde X_{2, i - m_1 })^2
}{m_1^2 m_2^2 
%\myVar(T_{\mathrm{CQ}})
\sigma_{T,n}^2
}
&
\text{for } m_1+1 \leq i \leq m_1 + m_2
.
    \end{array}
        \right.
\end{align*}
}%
It can be easily seen that $\rho_n = 9$ for the above defined random variables.
From Theorem \ref{thm:universality_GQF}, it suffices to prove that $\sum_{i=1}^{m_1 + m_2} \mathrm{Inf}_i^{3/2} \xrightarrow{P} 0$.
We have
{\small
\begin{align*}
    \sum_{i=1}^{m_1 + m_2} \mathrm{Inf}_i^{3/2} 
\leq
\left( \max_{i\in\{1, \ldots, m_1 + m_2 \}} \mathrm{Inf}_i \right)^{1/2}
\left(\sum_{i=1}^{m_1 + m_2} \mathrm{Inf}_i\right)
\leq
\left( 
\sum_{i=1}^{m_1 + m_2} \mathrm{Inf}_i^2 
\right)^{1/4}
\left(\sum_{i=1}^{m_1 + m_2} \mathrm{Inf}_i\right)
.
\end{align*}
}%
But
{\small
\begin{align*}
    \sum_{i=1}^{m_1 + m_2} \mathrm{Inf}_i
    =&
\sum_{k=1}^2
        \frac{4 
            \sum_{i=1}^{m_k} \sum_{j=1}^{m_k} (\tilde X_{k, i}^\myT \tilde X_{k, j})^2
            \mathbf 1_{\{ j \neq i\}}
        }{m_k^2 (m_k - 1)^2 
        %\myVar(T_{\mathrm{CQ}})
\sigma_{T,n}^2
    } 
        +
\frac{8
        \sum_{i =1}^{m_1}
        \sum_{j =1}^{m_2}
        (\tilde X_{2, i}^\myT \tilde X_{1, j})^2
}{m_1^2 m_2^2 
%\myVar(T_{\mathrm{CQ}})
\sigma_{T,n}^2
}
\\
=&
2
    \frac{
    \myVar\{
T_{\mathrm{CQ}}(
E^* ;
\tilde \BX_1 , \tilde \BX_2 
)
    \mid \tilde \BX_1, \tilde \BX_2 \}
    }{
    \sigma_{T,n}^2
}
\\
=&
2+o_P(1),
\end{align*}
}%
where the last equality follows from Lemma \ref{lemma:variance}.
Hence it suffices to prove that
$
\sum_{i=1}^{m_1 + m_2} \mathrm{Inf}_i^2  \xrightarrow{P} 0
$.
We have
{\small
\begin{align*}
    \myE \left(
        \sum_{i=1}^{m_1}
        \mathrm{Inf}_i^2
    \right)
    \leq&
        \frac{32
        }{m_1^4 (m_1 - 1)^4
\sigma_{T,n}^4
    } 
    \myE
    \left[
        \sum_{i=1}^{m_1}
        \left\{
            \sum_{j=1}^{m_1} (\tilde X_{1, i}^\myT \tilde X_{1,j})^2
            \mathbf 1_{\{ j \neq i\}}
        \right\}^2
    \right]
        %\myVar(T_{\mathrm{CQ}})
    \\
    &
        +\frac{32}{m_1^4 m_2^4 \sigma_{T, n}^4}
        \myE
        \left[
        \sum_{i=1}^{m_1}
        \left\{
        \sum_{j =1}^{m_2} (\tilde X_{1, i}^\myT \tilde X_{2, j})^2
    \right\}^2
\right]
.
\end{align*}
}%
We have
{\small
\begin{align*}
    &
    \myE
    \left[
        \sum_{i=1}^{m_1}
        \left\{
            \sum_{j=1}^{m_1} (\tilde X_{1, i}^\myT \tilde X_{1,j})^2
            \mathbf 1_{\{ j \neq i\}}
        \right\}^2
    \right]
    \\
    \leq&
        \sum_{i=1}^{m_1}
            \sum_{j=1}^{m_1}
            \sum_{\ell=1}^{m_1}
            \left[
    \myE
    \left\{
            (\tilde X_{1, i}^\myT \tilde X_{1,j})^4
        \right\}
    \myE
    \left\{
            (\tilde X_{1, i}^\myT \tilde X_{1,\ell})^4
        \right\}
        \right]^{1/2}
            \mathbf 1_{\{ j \neq i\}}
            \mathbf 1_{\{ \ell \neq i\}}
            \\
    \leq&
    \frac{\tau^2}{256}
        \sum_{i=1}^{m_1}
            \sum_{j=1}^{m_1}
            \sum_{\ell=1}^{m_1}
    \left\{
        \sum_{i' = 2i-1}^{2i}
        \sum_{j' = 2j-1}^{2j}
        \mytr (\bSigma_{1,i'}\bSigma_{1,j'})
    \right\}
    \left\{
        \sum_{i' = 2i-1}^{2i}
        \sum_{\ell' = 2\ell-1}^{2\ell}
        \mytr (\bSigma_{1,i'}\bSigma_{1,\ell'})
        \right\}
        \\
    \leq&
    \frac{\tau^2}{256}
    n_1^2
        \sum_{i=1}^{m_1}
    \left\{
        \sum_{i' = 2i-1}^{2i}
        \mytr (\bSigma_{1,i'}\bar \bSigma_{1})
    \right\}^2
        \\
    \leq&
    \frac{\tau^2}{128}
    n_1^2
        \sum_{i=1}^{n_1}
        \mytr (\bSigma_{1,i}^2)
        \mytr (\bar \bSigma_{1}^2)
        \\
    =&
    o
    \left[
        n_1^4
    \left\{
        \mytr(\bar \bSigma_1^2)
    \right\}^2
\right]
,
\end{align*}
}%
where the second inequality follows from Lemma \ref{lemma:small_lemma}
and the last equality follows from Assumption \ref{assumption7}.
On the other hand,
{\small
\begin{align*}
\myE
\left[
\sum_{i=1}^{m_1}
\left\{
\sum_{j =1}^{m_2} (\tilde X_{1, i}^\myT \tilde X_{2, j})^2
\right\}^2
\right]
\leq&
\sum_{i=1}^{m_1}
\sum_{j =1}^{m_2} 
\sum_{\ell =1}^{m_2} 
\left[
\myE
\left\{
(\tilde X_{1, i}^\myT \tilde X_{2, j})^4
\right\}
\myE
\left\{
(\tilde X_{1, i}^\myT \tilde X_{2, \ell})^4
\right\}
\right]^{1/2}
\\
\leq&
\frac{\tau^2}{128}
n_2^2 \sum_{i=1}^{n_1} \mytr(\bSigma_{1,i}^2) \mytr(\bar \bSigma_{2}^2)
\\
=&
o\left\{
    n_1^2 n_2^2
    \mytr(\bar \bSigma_1^2)
    \mytr(\bar \bSigma_2^2)
\right\}.
\end{align*}
}%
Thus,
{\small
\begin{align*}
    \sum_{i=1}^{m_1} \mathrm{Inf}_i^2
    =o_P
    \left[
        \frac{
            \left\{\mytr(\bar \bSigma_1^2)\right\}^2
        }{n_1^4 \sigma_{T,n}^4}
        +
        \frac{
    \mytr(\bar \bSigma_1^2)
    \mytr(\bar \bSigma_2^2)
        }
        {n_1^2 n_2^2 \sigma_{T,n}^4}
    \right]
    =o_P
    \left[
        \frac{
            \left\{\mytr(\bar \bSigma_1^2)\right\}^2
        }{n_1^4 \sigma_{T,n}^4}
        +
        \frac{
            \left\{\mytr(\bar \bSigma_2^2)\right\}^2
        }{n_2^4 \sigma_{T,n}^4}
    \right]
    =o_P(1)
    .
\end{align*}
}%
Similarly,
{\small
\begin{align*}
    \sum_{i=m_1 + 1}^{ m_1 + m_2 } \mathrm{Inf}_i^2
    =o_P(1)
    .
\end{align*}
}%
This completes the proof.
\end{proof}

\begin{lemma}\label{lemma:T13}
Suppose the conditions of Theorem \ref{thm:final_thm} hold.
Furthermore, suppose the condition \eqref{eq:subsequence_limit} holds.
Then \eqref{eq:conclusion_T13} holds.
\end{lemma}
\begin{proof}
    Let $E_1^* = (e_{1,1}^*, \ldots, e_{1,m_1}^*)^\myT$ and 
$E_2^* = (e_{2,1}^*, \ldots, e_{2,m_2}^*)^\myT$.
Define
{\small
\begin{align*}
T_{1,r}^* =
     \left\|
     \frac{1}{m_1}
\BU_r^{\myT}
\tilde \BX_1^\myT
E_1^{*}
-
\frac{1}{m_2}
\BU_r^{\myT}
\tilde \BX_{2}^\myT
E_2^*
     \right\|^2
     -
     \frac{1}{m_1^2}
     \|\tilde \BX_1 \BU_r\|_F^2
     -
     \frac{1}{m_2^2} 
     \|\tilde \BX_2 \BU_r\|_F^2
     .
\end{align*}
}%
%{\color{red}From Lemma \ref{lemma:T13},}
From Lemma \ref{lemma:add_error}
it suffices to prove
{\small
\begin{align}\label{eq:gugu1}
    \myE
    \left\{
        \frac{
    (
    T_{1,r} -
    T_{1,r}^*
    )^2
    }{
    \sigma_{T,n}^2
}
    \mid
    \tilde \BX_1, \tilde \BX_2
    \right\}
    \xrightarrow{P} 0
    ,
\end{align}
}%
and
{\small
\begin{align}\label{eq:gugu2}
    \left\|
    \mathcal L
    \left(
        \frac{
T_{1,r}^*
+
T_{3,r}
}{
\sigma_{T,n}
}
\mid
    \tilde \BX_1, \tilde \BX_2
\right)
-
    \mathcal L 
    \left(
            (1 - \sum_{i=1}^{\infty} \kappa_{i}^2)^{1/2}
            \xi_0
    +
    2^{-1/2}\sum_{i=1}^r \kappa_i (\xi_i^2-1)
    \right)
\right\|_3
\xrightarrow{P} 0
.
\end{align}
}%

To prove \eqref{eq:gugu1}, we only need to show that 
$\myE \{ (T_{1,r} - T_{1,r}^*)^2\} = o(\sigma_{T,n}^2)$ where the expectation is unconditional.
It is straightforward to show that
{\small
\begin{align*}
     T_{1,r}
     -T_{1,r}^*
     =
     &
     \sum_{k=1}^2
     \frac{
\left\|
\BU_r^{\myT}
\tilde \BX_k^\myT
E_k^{*}
\right\|^2
-
     \left\|\tilde \BX_k \BU_r \right\|_F^2
     }{m_k^2 (m_k - 1)}
-
\sum_{k=1}^2
\sum_{i=1}^{m_k}
\frac{
 (\epsilon_{k,i}^{*2}-1) \left\| \BU_r^{\myT} \tilde X_{k,i} \right\|^2
}{m_k (m_k - 1)}
.
\end{align*}
}%
For $k = 1, 2$, we have
{\small
\begin{align*}
    \myE
    \left\{
        \left(
\sum_{i=1}^{m_k}
            \frac{
 (\epsilon_{k,i}^{*2}-1)
\left\| \BU_r^{\myT} \tilde X_{k,i} \right\|^2
}{
m_k (m_k - 1)
}
    \right)^2
\right\}
    =&
    \frac{
2
}
{m_k^2 (m_k - 1)^2}
\sum_{i=1}^{m_k}
\myE 
\{
( \tilde X_{k,i}^\myT \BU_r \BU_r^{\myT} \tilde X_{k,i})^2
\}
.
\end{align*}
}%
But
{\small
\begin{align}\label{eq:miaomiao}
\sum_{i=1}^{m_k}
\myE 
\{
( \tilde X_{k,i}^\myT \BU_r \BU_r^{\myT} \tilde X_{k,i})^2
\}
\leq&
\frac{
\tau
}{8}
\sum_{i=1}^{n_k} \{\mytr(\BU_r^{\myT} \bSigma_{k,i} \BU_r )\}^2
\leq
\frac{\tau r}{8}
\sum_{i=1}^{n_k} \mytr\{(\BU_r^{\myT} \bSigma_{k,i} \BU_r )^2\}
,
\end{align}
}%
where the first inequality follows from Lemma \ref{lemma:small_lemma} and the second inequality follows from Cauchy-Schwarz inequality.
Note that $\mytr\{(\BU_r^{\myT} \bSigma_{k,i} \BU_r )^2\} \leq \mytr(\bSigma_{k,i}^2)$.
Thus,
{\small
\begin{align*}
    \myE
    \left\{
        \left(
\sum_{i=1}^{m_k}
            \frac{
 (\epsilon_{k,i}^{*2}-1)
\left\| \BU_r^{\myT} \tilde X_{k,i} \right\|^2
}{
m_k (m_k - 1)
}
    \right)^2
\right\}
    =&
O
    \left(
        \frac{1}{n_k^4}
    \sum_{i=1}^{n_k} \mytr( \bSigma_{k,i}^2 )
\right)
=
o\left(
    \frac{1}{n_k^2}
        \mytr
        (\bar \bSigma_k^2
        )
\right)
=
o(
%\myVar(T_{\mathrm{CQ}})
\sigma_{T,n}^2
)
,
\end{align*}
}%
where the second equality follows from Assumption \ref{assumption7}.
On the other hand, for $k = 1, 2$, we have
{\small
\begin{align*}
     \myE
     \left\{
     \left(
         \frac{
\left\|
\BU_r^{\myT}
\tilde \BX_k^\myT
E_k^{*}
\right\|^2
-
     \|\tilde \BX_k \BU_r\|_F^2
 }{
 m_k^2 (m_k - 1)
 }
 \right)^2
 \right\}
 =
 \frac{2}{m_k^4 (m_k - 1)^2}
 \myE
 \left[
 \mytr 
 \left\{
 (
\BU_r^{\myT} \tilde \BX_k^\myT
 \tilde \BX_k \BU_r
)^2
\right\}
\right]
.
\end{align*}
}%
But
{\small
\begin{align*}
    \myE
    \left[
    \mytr
    \left\{
 (
\BU_r^{\myT} \tilde \BX_k^\myT
 \tilde \BX_k \BU_r
)^2
\right\}
\right]
=&
 \sum_{i=1}^{m_k}
    \myE
    \left[
    \mytr
    \left\{
 (
 \BU_r^{\myT} \tilde X_{k, i}
 \tilde X_{k, i}^\myT \BU_r
)^2
\right\}
\right]
\\
 &
+
\frac{1}{8}
\sum_{i=1}^{m_k}
\sum_{j=i+1}^{m_k}
    \mytr
    \left\{
 (
 \BU_r^{\myT} 
 (
 \bSigma_{k, 2i-1}
 +
 \bSigma_{k, 2i}
 )
 \BU_r
)(
 \BU_r^{\myT} 
 (
 \bSigma_{k, 2j-1}
 +
 \bSigma_{k, 2j}
 )
 \BU_r
)
\right\}
\\
\leq&
\frac{
\tau r
}{8}
\sum_{i=1}^{n_k}
    \mytr
\{
    (
 \BU_r^{\myT} 
 \bSigma_{k,i}
 \BU_r
    )^2
\}
+
\frac{1}{8}
\sum_{i=1}^{n_k}
\sum_{j=i+1}^{n_k}
\mytr
    \{
 (
 \BU_r^{\myT} 
 \bSigma_{k,i}
 \BU_r 
)
 (
 \BU_r^{\myT} 
 \bSigma_{k,j}
 \BU_r
)
\}
\\
\leq&
\frac{
    \tau r
}{8}
    n_k^2
\mytr\{
    (
    \BU_r^{\myT}
    \bar \bSigma_k
    \BU_r
)^2\}
\\
\leq&
\frac{\tau r }{8}
    n_k^2
    \mytr(
    \bar \bSigma_k^2)
%\\
%=&
%O(n_k^4 \myVar(T_{\mathrm{CQ}}) )
,
\end{align*}
}%
where the first inequality follows from \eqref{eq:miaomiao}.
It follows that
{\small
\begin{align*}
     \myE
     \left\{
     \left(
         \frac{
\left\|
\BU_r^{\myT}
\tilde \BX_k^\myT
E_k^{*}
\right\|^2
-
     \|\tilde \BX_k \BU_r\|_F^2
 }{
 m_k^2 (m_k - 1)
 }
 \right)^2
 \right\}
 =o\left(
     \sigma_{T,n}^2
 \right)
 .
\end{align*}
}%
Thus, \eqref{eq:gugu1} holds.

Now we apply Lemma \ref{lemma:630} to prove \eqref{eq:gugu2}.
In Lemma \ref{lemma:630}, we take
$\bZeta_n = E^*$,
and
{\small
\begin{align*}
    \BA_n
    =
    \begin{pmatrix}
        \BA_{1,1} \quad & \BA_{1,2}
        \\
        \BA_{1,2}^\myT \quad & \BA_{2,2}
    \end{pmatrix}
    ,
\quad
    \BB_n = 
    \sigma_{T,n}^{-1/2}
    \left(
        \frac{1}{m_1}
\BU_r^{\myT}
\tilde \BX_1^\myT
,
\quad
-\frac{1}{m_2}
\BU_r^{\myT}
\tilde \BX_{2}^\myT
\right)
    ,
\end{align*}
}%
where
{\small
\begin{align*}
    \BA_{1,1}
    =&
    \frac{1}{
        \sigma_{T,n} m_1 (m_1 - 1)
    }
    \begin{pmatrix}
        0 & 
        \tilde X_{1,1}^\myT \tilde \BU_{r_n^*} \tilde \BU_{r_n^*}^{ \myT} \tilde X_{1,2}
          &\cdots &
          \tilde X_{1,1}^\myT \tilde \BU_{r_n^*} \tilde \BU_{r_n^*}^{\myT} \tilde X_{1,m_1}
\\
    \tilde X_{1,2}^\myT \tilde \BU_{r_n^*} \tilde \BU_{r_n^*}^{\myT} \tilde X_{1,1}
    & 
    0
          &\cdots &
          \tilde X_{1,2}^\myT \tilde \BU_{r_n^*} \tilde \BU_{r_n^*}^{\myT} \tilde X_{1,m_1}
\\
\vdots & \vdots & \ddots & \vdots 
\\
\tilde X_{1,m_1}^\myT \tilde \BU_{r_n^*} \tilde \BU_{r_n^*}^{\myT} \tilde X_{1,1}
    & 
    \tilde X_{1,m_1}^\myT \tilde \BU_{r_n^*} \tilde \BU_{r_n^*}^{\myT} \tilde X_{1,2}
          &\cdots &
          0
    \end{pmatrix}
    ,
    \\
    \BA_{1,2}
    =&
    -
    \frac{1}{
        \sigma_{T,n}
        m_1 m_2
    }
    \begin{pmatrix}
        \tilde X_{1, 1}^\myT \tilde \BU_{r_n^*} \tilde \BU_{r_n^*}^{\myT} \tilde X_{2,1 } 
&
\tilde X_{1, 1}^\myT \tilde \BU_{r_n^*} \tilde \BU_{r_n^*}^{\myT} \tilde X_{2,2 } 
& \cdots &
 \tilde X_{1, 1}^\myT \tilde \BU_{r_n^*} \tilde \BU_{r_n^*}^{\myT} \tilde X_{2,m_2 } 
\\
\tilde X_{1, 2}^\myT \tilde \BU_{r_n^*} \tilde \BU_{r_n^*}^{\myT} \tilde X_{2,1 } 
&
\tilde X_{1, 2}^\myT \tilde \BU_{r_n^*} \tilde \BU_{r_n^*}^{ \myT} \tilde X_{2,2 } 
& \cdots &
\tilde X_{1, 2}^\myT \tilde \BU_{r_n^*} \tilde \BU_{r_n^*}^{\myT} \tilde X_{2,m_2 } 
\\
\vdots & \vdots & \ddots & \vdots 
\\
\tilde X_{1, m_1}^\myT \tilde \BU_{r_n^*} \tilde \BU_{r_n^*}^{\myT} \tilde X_{2,1 } 
&
\tilde X_{1, m_1}^\myT \tilde \BU_{r_n^*} \tilde \BU_{r_n^*}^{\myT} \tilde X_{2,2 } 
& \cdots &
\tilde X_{1, m_1}^\myT \tilde \BU_{r_n^*} \tilde \BU_{r_n^*}^{\myT} \tilde X_{2,m_2 } 
    \end{pmatrix}
    ,
    \\
    \BA_{2,2}
    =&
    \frac{1}{
        \sigma_{T,n}
        m_2 (m_2 - 1)
    }
    \begin{pmatrix}
        0 & 
        \tilde X_{2,1}^\myT \tilde \BU_{r_n^*} \tilde \BU_{r_n^*}^{\myT} \tilde X_{2,2}
          &\cdots &
          \tilde X_{2,1}^\myT \tilde \BU_{r_n^*} \tilde \BU_{r_n^*}^{\myT} \tilde X_{2,m_2 }
\\
\tilde X_{2,2}^\myT \tilde \BU_{r_n^*} \tilde \BU_{r_n^*}^{\myT} \tilde X_{2,1}
    & 
    0
          &\cdots &
          \tilde X_{2,2}^\myT \tilde \BU_{r_n^*} \tilde \BU_{r_n^*}^{\myT} \tilde X_{2,m_2 }
\\
\vdots & \vdots & \ddots & \vdots 
\\
\tilde X_{2,m_2 }^\myT \tilde \BU_{r_n^*} \tilde \BU_{r_n^*}^{\myT} \tilde X_{2,1}
    & 
    \tilde X_{2,m_2 }^\myT \tilde \BU_{r_n^*} \tilde \BU_{r_n^*}^{\myT} \tilde X_{2,2}
          &\cdots &
          0
    \end{pmatrix}
    .
\end{align*}
}%
Then we have
{\small
\begin{align*}
    \bzeta_n^\myT \BA_n \bzeta_n
    -\mytr(\BA_n)
    =
    \frac{
T_{\mathrm{CQ}}(
E^*;
\tilde \BX_1 \tilde \BU_{r_n^*}, \tilde \BX_2 \tilde \BU_{r_n^*}
) 
}{\sigma_{T,n}}
=
\frac{T_{3,r}}{
    \sigma_{T,n}
}
,
\end{align*}
}%
and
{\small
\begin{align*}
    \bzeta_n^\myT
    \BB_n^\myT \BB_n
    \bzeta_n
    -
    \mytr(\BB_n^\myT \BB_n)
    =
    \frac{T_{1,r}^*}
    {\sigma_{T,n}}
    .
\end{align*}
}%

From Lemma \ref{lemma:variance},
{\small
\begin{align}
    2
\mytr(\BA_n^2) =
&
\frac{
    \myVar\{
T_{\mathrm{CQ}}(
E^*;
\tilde \BX_1 \tilde \BU_{r_n^*}, \tilde \BX_2 \tilde \BU_{r_n^*}
) 
\mid \tilde \BX_1 , \tilde \BX_2
\}
}
{\sigma_{T,n}^2}
\notag
\\
=&
\frac{ \mytr\{(
            \tilde \BU_{r_n^*}^{\myT}
            \bPsi_n 
\tilde \BU_{r_n^*}
)^2\} }{ 
        \mytr
        (
        \bPsi_n^2
            )
}
    +o_P(1)
\notag
\\
=&
1-
    \frac{
        \sum_{i=1}^{r_n^*} \lambda_{i}^2
        (
        \bPsi_n
        )
    }
    {
        \mytr
        \left(
            \bPsi_n^2
    \right)
    }
    +o_P(1)
\notag
    \\
    =& 
    1 - \sum_{i=1}^{\infty} \kappa_{i}^2
    + o_P(1)
    ,
    \label{eq:79c1}
\end{align}
}%
where the last equality follows from \eqref{eq:wahaha}.
On the other hand, 
%from the inequality $ \mytr(\BB_1 + \BB_2)^4 \leq 8 (\mytr(\BB_1^4) + \mytr(\BB_2^4))$, we have
{\small
\begin{align*}
    \mytr(\BA_n^4)
    \leq &
    8
    \left\{
    \mytr
    \left\{
    \begin{pmatrix}
        \BA_{1,1}
        & \BO_{n_1 \times n_2}
        \\
            \BO_{n_2 \times n_1}
        &
        \BA_{2,2}
    \end{pmatrix}^4
\right\}
    +
    \mytr
    \left\{
    \begin{pmatrix}
        \BO_{n_1 \times n_1}
        & \BA_{1,2}
        \\
        \BA_{1,2}^\myT
        &
        \BO_{n_2 \times n_2}
    \end{pmatrix}^4
\right\}
\right\}
\\
    =
    &
        8\mytr (\BA_{1,1}^4)
        +
        8\mytr (\BA_{2,2}^4)
        +
        16\mytr \{(\BA_{1,2} \BA_{1,2}^\myT )^2\}
    .
\end{align*}
}%
For $i,j = 1, \ldots, m_1$, let $w_{i,j} = 
            \tilde X_{1,i}^\myT \tilde \BU_{r_n^*} \tilde \BU_{r_n^*}^{\myT} \tilde X_{1, j}
$.
We have
{\small
\begin{align*}
    \sigma_{T,n}^4
    m_1^4 (m_1 -1)^4
    \mytr(\BA_{1,1}^4)
    =
    &
    \sum_{i=1}^{m_1}
    \sum_{j=1}^{m_1}
    \sum_{k=1}^{m_1}
    \sum_{\ell=1}^{m_1}
    w_{i,j}
    w_{j,k}
    w_{k,\ell}
    w_{\ell,i}
    \mathbf 1_{\{i \neq j \}}
    \mathbf 1_{\{j \neq k \}}
    \mathbf 1_{\{k \neq \ell \}}
    \mathbf 1_{\{\ell \neq i \}}
    .
\end{align*}
}%
We split the above sum into the following four cases: $k = i, \ell = j$; $k = i, \ell \neq j$; $k \neq i, \ell = j$; $k \neq i, \ell \neq j$.
The second and the third cases result in the same sum.
Then we have
{\small
\begin{align}
    \sigma_{T,n}^4
    m_1^4 (m_1 -1)^4
    \mytr(\BA_{1,1}^4)
    =
    &
    \sum_{i=1}^{m_1}
    \sum_{j=1}^{m_1}
    w_{i,j}^4
    \mathbf 1_{\{i \neq j \}}
    +
    2
    \sum_{i=1}^{m_1}
    \sum_{j=1}^{m_1}
    \sum_{k=1}^{m_1}
    w_{i,j}^2
    w_{i,k}^2
    \mathbf 1_{\{i, j, k \text{ are distinct} \}}
    \notag
    \\
    &
    +
    \sum_{i=1}^{m_1}
    \sum_{j=1}^{m_1}
    \sum_{k=1}^{m_1}
    \sum_{\ell=1}^{m_1}
    w_{i,j}
    w_{j,k}
    w_{k,\ell}
    w_{\ell,i}
    \mathbf 1_{\{i, j, k, \ell \text{ are distinct} \}}
    .
    \label{eq:threeterms}
\end{align}
}%

First we deal with the first two terms of \eqref{eq:threeterms}.
    For distinct $i,j \in \{1, \ldots, m_1\}$,
    two applications of Lemma \ref{lemma:small_lemma} yield
    {\small
    \begin{align*}
        \myE( w_{i,j}^4)
    =&
    \myE
    \left(
        \tilde X_{1,i}^\myT \tilde \BU_{r_n^*} \tilde \BU_{r_n^*}^{ \myT} \tilde X_{1,j}
        \tilde X_{1,j}^\myT \tilde \BU_{r_n^*} \tilde \BU_{r_n^*}^{\myT} \tilde X_{1, i}
    \right)^2
    \\
    \leq&
    \frac{\tau^2}{256}
    \Big\{
        \mytr(
        \tilde \BU_{r_n^*}^{\myT}
        \bSigma_{1,2i-1}
        \tilde \BU_{r_n^*}
        \tilde \BU_{r_n^*}^{\myT}
        \bSigma_{1,2j-1}
        \tilde \BU_{r_n^*}
        )
        +
        \mytr(
        \tilde \BU_{r_n^*}^{\myT}
        \bSigma_{1,2i-1}
        \tilde \BU_{r_n^*}
        \tilde \BU_{r_n^*}^{\myT}
        \bSigma_{1,2j}
        \tilde \BU_{r_n^*}
        )
        \\
        &
        \quad
        \quad
        +
        \mytr(
        \tilde \BU_{r_n^*}^{\myT}
        \bSigma_{1,2i}
        \tilde \BU_{r_n^*}
        \tilde \BU_{r_n^*}^{\myT}
        \bSigma_{1,2j-1}
        \tilde \BU_{r_n^*}
        )
        +
        \mytr(
        \tilde \BU_{r_n^*}^{\myT}
        \bSigma_{1,2i}
        \tilde \BU_{r_n^*}
        \tilde \BU_{r_n^*}^{\myT}
        \bSigma_{1,2j}
        \tilde \BU_{r_n^*}
        )
    \Big\}^2
    .
    \end{align*}
}%
    Consequently,
    {\small
    \begin{align}
        &
        \myE
        \left\{
    \sum_{i=1}^{m_1}
    \sum_{j=1}^{m_1}
    w_{i,j}^4
    \mathbf 1_{\{i \neq j \}}
    +
    2
    \sum_{i=1}^{m_1}
    \sum_{j=1}^{m_1}
    \sum_{k=1}^{m_1}
    w_{i,j}^2
    w_{i,k}^2
    \mathbf 1_{\{i, j, k \text{ are distinct} \}}
\right\}
    %\\
    %\leq
        %&
    %\sum_{i=1}^{m_1}
    %\sum_{j=1}^{m_1}
    %    \myE
    %    (
    %w_{i,j}^4
    %)
    %\mathbf 1_{\{i \neq j \}}
    %+
    %2
    %\sum_{i=1}^{m_1}
    %\sum_{j=1}^{m_1}
    %\sum_{k=1}^{m_1}
    %\{
    %\myE(
    %w_{i,j}^4
    %)
    %\myE(
    %w_{i,k}^2
    %)
%\}^{1/2}
%    \mathbf 1_{\{i, j, k \text{ are distinct} \}}
\notag
    \\
        \leq&
        2
    \sum_{i=1}^{m_1}
    \left[
    \sum_{j=1}^{m_1}
    \{
        \myE
        (
    w_{i,j}^4
    )
\}^{1/2}
    \mathbf 1_{\{i \neq j \}}
\right]^2
\notag
    \\
        \leq&
        \frac{\tau^2}{128}
        n_1^2
    \sum_{i=1}^{m_1}
    \left\{
        \mytr(\tilde \BU_{r_n^*}^{\myT} \bSigma_{1, 2i-1} \tilde \BU_{r_n^*} \tilde \BU_{r_n^*}^{\myT} \bar \bSigma_{1} \tilde \BU_{r_n^*})
        +
        \mytr(\tilde \BU_{r_n^*}^{\myT} \bSigma_{1, 2i} \tilde \BU_{r_n^*} \tilde \BU_{r_n^*}^{\myT} \bar \bSigma_{1} \tilde \BU_{r_n^*})
    \right\}^2
\notag
    \\
        \leq&
        \frac{\tau^2}{64}
        n_1^2
    \sum_{i=1}^{n_1}
    \left\{
        \mytr(\tilde \BU_{r_n^*}^{\myT} \bSigma_{1, i} \tilde \BU_{r_n^*} \tilde \BU_{r_n^*}^{\myT} \bar \bSigma_{1} \tilde \BU_{r_n^*})
    \right\}^2
\notag
    \\
        \leq&
        \frac{\tau^2}{64}
        n_1^2
    \mytr
        (
        \bar \bSigma_{1}^2
        )
    \sum_{i=1}^{n_1}
        \mytr
            (\bSigma_{1, i}^2)
\notag
            \\
        =&o\left[
        n_1^4
    \left\{
        \mytr
            (\bar \bSigma_{1}^2)
    \right\}^2
    \right],
    \label{eq:mumu}
    \end{align}
}%
    where the last equality follows from Assumption \ref{assumption7}.
Now we deal with the third term of \eqref{eq:threeterms}.
For distinct $i, j, k, \ell \in \{1, \ldots, m_1\}$, we have
{\small
\begin{align*}
    &
    \myE
    (
    w_{i,j}
    w_{j,k}
    w_{k, \ell}
    w_{\ell, i}
    )
    \\
    =&
    \frac{1}{256}
     \mytr
     \left\{
        \tilde \BU_{r_n^*}^{ \myT}
        \left(\sum_{i^\dagger = 2i-1}^{2i}
        \bSigma_{1, i^\dagger}
        \right)
        \tilde \BU_{r_n^*}
    \tilde \BU_{r_n^*}^{\myT}
        \left(
    \sum_{j^\dagger = 2j-1}^{2j}
        \bSigma_{1, j^\dagger}
        \right)
    \tilde \BU_{r_n^*}
    \tilde \BU_{r_n^*}^{\myT}
    \left(
    \sum_{k^\dagger = 2k-1}^{2k}
        \bSigma_{1, k^\dagger}
    \right)
    \tilde \BU_{r_n^*}
    \tilde \BU_{r_n^*}^{\myT}
    \left(
    \sum_{\ell^\dagger = 2\ell - 1}^{2\ell}
        \bSigma_{1, \ell^\dagger}
    \right)
    \tilde \BU_{r_n^*}
\right\}
    .
\end{align*}
}%
Then from Lemma \ref{lemma:final_matrix_1}, we have
{\small
\begin{align*}
    &
    \myE
    \left\{
    \sum_{i=1}^{m_1}
    \sum_{j=1}^{m_1}
    \sum_{k=1}^{m_1}
    \sum_{\ell=1}^{m_1}
    w_{i,j}
    w_{j,k}
    w_{k,\ell}
    w_{\ell,i}
    \mathbf 1_{\{i, j, k, \ell \text{ are distinct} \}}
\right\}
\\
    =&
    O\left[
            n_1^4
            \mytr\left\{
                (
                \tilde \BU_{r_n^*}^{\myT}
                    \bar \bSigma_{1}
                \tilde \BU_{r_n^*}
                )^4
            \right\}
            +
            n_1^2
            \mytr\left(
                    \bar \bSigma_{1}^2
                \right)
        \sum_{i=1}^{n_1}
            \mytr\left(
                \bSigma_{1,i}^2
            \right)
            +
            \left\{
                \sum_{i=1}^{n_1}
            \mytr
                (
                    \bSigma_{1,i}^2
                )
            \right\}^2
        \right]
\\
= &
O\left[
                n_1^6
                \left\{
                    \lambda_{r_n^* + 1}
                (
                \bPsi_n
    )
\right\}^2
                \mytr
                (
        \bar \bSigma_{1}^2
                )
\right]
+
o\left[ n_1^4 
\{
                \mytr
                (
        \bar \bSigma_{1}^2
                )
            \}^2
        \right]
                ,
\end{align*}
}%
where the last equality follows from \eqref{eq:bound_mp1_eigenvalue} and Assumption \ref{assumption7}.
It follows from \eqref{eq:mumu} and the above bound that
{\small
    \begin{align*}
        \mytr(\BA_{1,1}^4) =
        O_P
        \left\{
            \frac{
                \left\{
                \lambda_{r_n^* + 1}
                \left(
                    \bPsi_n
                \right)
            \right\}^2
                }{
                \sigma_{T,n}^2
        }
            \frac{
            \mytr
            (
            \bar \bSigma_{1}^2
            )
            }{
                n_1^2
                \sigma_{T,n}^2
        }
    \right\}
        +
        o_P
        \left[
            \left\{
                \frac{\mytr(\bar \bSigma_1^2)}{n_1^2 \sigma_{T,n}^2}
            \right\}^2
        \right]
        =o_P(1)
        .
    \end{align*}
}%
    Similarly, we have
        $\mytr(\BA_{2,2}^4) =
        o_P
        \left(
            1
        \right)
$.

    Now we deal with $\mytr\{(\BA_{1,2} \BA_{1,2}^\myT)^2\}$.
    For $i = 1, \ldots, m_1$, $j =1, \ldots, m_2$,
    let $w_{i,j}^* = \tilde X_{1,i}^\myT \tilde \BU_{r_n^*} \tilde \BU_{r_n^*}^{\myT} \tilde X_{2, j} $.
    Then
    {\small
    \begin{align}
        \sigma_{T,n}^4
        m_1^4 
        m_2^4 
        \mytr\{(\BA_{1,2} \BA_{1,2}^\myT)^2\}
        %=
        %\sum_{i=1}^{m_1}
        %\sum_{j=1}^{m_1}
        %\sum_{k=1}^{m_2}
        %\sum_{\ell=1}^{m_2}
        %w_{i,k}^*
        %w_{j,k}^*
        %w_{i,\ell}^*
        %w_{j,\ell}^*
        =
        &
        \sum_{i=1}^{m_1}
        \sum_{k=1}^{m_2}
        w_{i,k}^{*4}
        +
        \sum_{i=1}^{m_1}
        \sum_{j=1}^{m_1}
        \sum_{k=1}^{m_2}
        w_{i,k}^{*2}
        w_{j,k}^{*2}
        \mathbf 1_{\{i \neq j\}}
        +
        \sum_{i=1}^{m_1}
        \sum_{k=1}^{m_2}
        \sum_{\ell=1}^{m_2}
        w_{i,k}^{*2}
        w_{i,\ell}^{*2}
        \mathbf 1_{\{k \neq \ell\}}
        \notag
        \\
        &
        +
        \sum_{i=1}^{m_1}
        \sum_{j=1}^{m_1}
        \sum_{k=1}^{m_2}
        \sum_{\ell=1}^{m_2}
        w_{i,k}^*
        w_{j,k}^*
        w_{i,\ell}^*
        w_{j,\ell}^*
        \mathbf 1_{\{i \neq j\}}
        \mathbf 1_{\{k \neq \ell\}}
        .
        \label{eq:threeterms2}
    \end{align}
}%
    First we deal with the first three terms of \eqref{eq:threeterms2}.
    From Lemma \ref{lemma:small_lemma},
for $i = 1, \ldots, m_1$ and $ j = 1, \ldots, m_2$,
{\small
\begin{align*}
    \myE(w_{i,j}^{*4})
    =&
    \myE
    \{
    (
\tilde X_{1,i}^\myT \tilde \BU_{r_n^*} \tilde \BU_{r_n^*}^{\myT} \tilde X_{2, j} 
\tilde X_{2,j}^\myT \tilde \BU_{r_n^*} \tilde \BU_{r_n^*}^{\myT} \tilde X_{1, i} 
)^2
\}
\\
    \leq &
\frac{\tau^2}{256}
\Big\{
    \mytr(
        \tilde \BU_{r_n^*}^{\myT} 
        \bSigma_{1,2i-1}
        \tilde \BU_{r_n^*} 
        \tilde \BU_{r_n^*}^{\myT} 
        \bSigma_{2,2j-1}
        \tilde \BU_{r_n^*} 
    )
    +
    \mytr(
        \tilde \BU_{r_n^*}^{\myT} 
        \bSigma_{1,2i-1}
        \tilde \BU_{r_n^*} 
        \tilde \BU_{r_n^*}^{\myT} 
        \bSigma_{2,2j}
        \tilde \BU_{r_n^*} 
    )
    \\
         & \quad \quad
    +
    \mytr(
        \tilde \BU_{r_n^*}^{\myT} 
        \bSigma_{1,2i}
        \tilde \BU_{r_n^*} 
        \tilde \BU_{r_n^*}^{\myT} 
        \bSigma_{2,2j-1}
        \tilde \BU_{r_n^*} 
    )
    +
    \mytr(
        \tilde \BU_{r_n^*}^{\myT} 
        \bSigma_{1,2i}
        \tilde \BU_{r_n^*} 
        \tilde \BU_{r_n^*}^{\myT} 
        \bSigma_{2,2j}
        \tilde \BU_{r_n^*} 
    )
\Big\}^2
    .
\end{align*}
}%
Consequently,
{\small
\begin{align}
    &
    \myE \left\{
        \sum_{i=1}^{m_1}
        \sum_{k=1}^{m_2}
        w_{i,k}^{*4}
        +
        \sum_{i=1}^{m_1}
        \sum_{j=1}^{m_1}
        \sum_{k=1}^{m_2}
        w_{i,k}^{*2}
        w_{j,k}^{*2}
        \mathbf 1_{\{i \neq j\}}
        +
        \sum_{i=1}^{m_1}
        \sum_{k=1}^{m_2}
        \sum_{\ell=1}^{m_2}
        w_{i,k}^{*2}
        w_{i,\ell}^{*2}
        \mathbf 1_{\{k \neq \ell\}}
    \right\}
%    \\
%    \leq&
%    \myE \left\{
%        \sum_{i=1}^{m_1}
%        \sum_{k=1}^{m_2}
%        \myE(
%        w_{i,k}^{*4}
%        )
%        +
%        \sum_{i=1}^{m_1}
%        \sum_{j=1}^{m_1}
%        \sum_{k=1}^{m_2}
%        \left\{
%        \myE
%        (
%        w_{i,k}^{*4}
%        )
%        \myE
%        (
%        w_{j,k}^{*4}
%        )
%    \right\}^{1/2}
%        \mathbf 1_{\{i \neq j\}}
%        +
%        \sum_{i=1}^{m_1}
%        \sum_{k=1}^{m_2}
%        \sum_{\ell=1}^{m_2}
%        \left\{
%            \myE(
%        w_{i,k}^{*4}
%        )
%        \myE(
%        w_{i,\ell}^{*4}
%        )
%    \right\}^{1/2}
%        \mathbf 1_{\{k \neq \ell\}}
%    \right\}
    \notag
    \\
    \leq&
        \sum_{i=1}^{m_1}
        \left[
        \sum_{j=1}^{m_2}
        \left\{
        \myE(
        w_{i,j}^{*4}
        )
    \right\}^{1/2}
\right]^2
        +
        \sum_{j=1}^{m_2}
        \left[
        \sum_{i=1}^{m_1}
        \left\{
        \myE(
        w_{i,j}^{*4}
        )
    \right\}^{1/2}
\right]^2
    \notag
\\
    \leq&
    \frac{\tau^2 n_2^2}{128}
    \sum_{i=1}^{n_1}
    \left\{
    \mytr(
    \tilde \BU_{r_n^*}^{\myT} \bSigma_{1,i} \tilde \BU_{r_n^*}
    \tilde \BU_{r_n^*}^{\myT} \bar \bSigma_{2} \tilde \BU_{r_n^*}
    ) 
    \right\}^2
    +
    \frac{\tau^2 n_1^2}{128}
    \sum_{j=1}^{n_2}
    \left\{
    \mytr(
    \tilde \BU_{r_n^*}^{\myT} \bar \bSigma_{1} \tilde \BU_{r_n^*}
    \tilde \BU_{r_n^*}^{\myT} \bSigma_{2,j} \tilde \BU_{r_n^*}
    ) 
    \right\}^2
    \notag
\\
    \leq&
    \frac{\tau^2 n_2^2}{128}
 \mytr(
     \bar \bSigma_{2}^2
     )
    \sum_{i=1}^{n_1}
    \mytr(
 \bSigma_{1,i}^2
 )
    +
    \frac{\tau^2 n_1^2}{128}
 \mytr(
     \bar \bSigma_{1}^2
     )
    \sum_{j=1}^{n_2}
    \mytr(
 \bSigma_{2,j}^2
 )
    \notag
 \\
    =&
    o\left(
        n_1^2 n_2^2
        \mytr(\bar \bSigma_1^2)
        \mytr(\bar \bSigma_2^2)
    \right)
    \notag
 \\
    =&
    o\left[
        n_1^4 n_2^4
        \left\{\mytr(\bPsi_n^2)\right\}^2
    \right]
    ,
    \label{eq:ohohoh}
\end{align}
}%
where the second last equality follows from
Assumption \ref{assumption7}.
Now we deal with the fourth term of \eqref{eq:threeterms2}.
    For distinct $i, j \in \{1, \ldots, m_1 \}$ and 
    distinct
    $ k, \ell \in \{1, \ldots, m_2 \}$, 
we have
{\small
\begin{align*}
    &
        \myE
        (
        w_{i,k}^*
        w_{j,k}^*
        w_{i,\ell}^*
        w_{j,\ell}^*
        )
        \\
        =&
    \frac{1}{256}
        \mytr
        \left\{
            \tilde \BU_{r_n^*}^{\myT}
            \left(\sum_{i^\dagger = 2i-1}^{2i}
            \bSigma_{1,i^\dagger} \right) \tilde \BU_{r_n^*}
            \tilde \BU_{r_n^*}^{\myT}
            \left(\sum_{k^\dagger = 2k-1}^{2k}
            \bSigma_{2, k^\dagger} \right) \tilde \BU_{r_n^*}
            \tilde \BU_{r_n^*}^{\myT}
            \left(\sum_{j^\dagger = 2j-1}^{2j}
            \bSigma_{1,j^\dagger} \right) \tilde \BU_{r_n^*}
            \tilde \BU_{r_n^*}^{\myT}
            \left(\sum_{\ell^\dagger = 2\ell-1}^{2\ell}
            \bSigma_{2,\ell^\dagger} \right) \tilde \BU_{r_n^*}
    \right\}
        .
\end{align*}
}%
Then from Lemma \ref{lemma:final_matrix_2}, 
we have
{\small
\begin{align*}
    &
    \myE
    \left(
        \sum_{i=1}^{m_1}
        \sum_{j=1}^{m_1}
        \sum_{k=1}^{m_2}
        \sum_{\ell=1}^{m_2}
        w_{i,k}^*
        w_{j,k}^*
        w_{i,\ell}^*
        w_{j,\ell}^*
        \mathbf 1_{\{i \neq j\}}
        \mathbf 1_{\{k \neq \ell\}}
    \right)
    \\
    =&
    O\Bigg[
        n_2^4
        \mytr \left\{
        \left(
        \tilde \BU_{r_n^*}^{\myT} \bar \bSigma_1 \tilde \BU_{r_n^*}
    \right)^4
        \right\}
    +
        n_1^4
        \mytr \left\{
        \left(
        \tilde \BU_{r_n^*}^{\myT} \bar \bSigma_2 \tilde \BU_{r_n^*}
    \right)^4
        \right\}
    \\
     &
     \quad
     \quad
    +
            n_2^2
            \mytr ( \bar \bSigma_2^2 ) 
            \sum_{i=1}^{n_1}
            \mytr ( \bSigma_{1,i}^2 )
            +
            n_1^2
            \mytr (\bar \bSigma_1^2)
            \sum_{i=1}^{n_2}
            \mytr ( \bSigma_{2,i}^2)
        +
        \left\{
            \sum_{i=1}^{n_1}
            \mytr ( \bSigma_{1,i}^2 )
        \right\}
        \left\{
            \sum_{i=1}^{n_2}
            \mytr ( \bSigma_{2,i}^2 )
        \right\}
    \Bigg]
    \\
    =&
    O\Bigg[
    n_1^4
    n_2^4
    \left\{
        \lambda_{r_n^* + 1} (\bPsi_n)
    \right\}^2
    \mytr\left(
        \bPsi_n^2
    \right)
    \Bigg]
    +
    o\left[
        n_1^4 n_2^4
        \left\{
            \mytr(\bPsi_n^2)
        \right\}^2
    \right]
    ,
\end{align*}
}%
where the last equality follows from
\eqref{eq:bound_mp1_eigenvalue} and
Assumption \ref{assumption7}.
It follows from the above inequality and \eqref{eq:ohohoh} that
{\small
\begin{align*}
    \mytr\{(\BA_{1,2} \BA_{1,2}^\myT)^2\}
    =
    &
    O_P\left[
        \frac{
            \left\{
    \lambda_{r_n^* + 1}
    (\bPsi_n)
\right\}^2
        }
        { \sigma_{T,n}^2 }
    \right]
    + 
    o_P(1)
    %\left[
    %    \frac{
    %    \left\{\mytr(\bPsi_n^2)\right\}^2
    %    }{\sigma_{T,n}^4}
    %\right]
    =o_P(1)
    .
\end{align*}
}%
    Thus, we have 
    {\small
    \begin{align}
    \mytr(\BA^4) = o_P(1).
    \label{eq:79c2}
    \end{align}
}%

    Now we deal with $\BB_n \BB_n^\myT$.
We have
{\small
\begin{align*}
    \BB_n \BB_n^\myT
    =
    \sigma_{T,n}^{-1}
    \left(
        \frac{1}{m_1^2}
        \BU_r^{\myT} \tilde \BX_1^\myT \tilde \BX_1 \BU_r
        +
        \frac{1}{m_2^2}
        \BU_r^{\myT} \tilde \BX_2^\myT \tilde \BX_2 \BU_r
    \right)
    .
\end{align*}
}%
For $k = 1,2$, we have
{\small
\begin{align*}
    \myE (\BU_r^{\myT} \tilde \BX_k^\myT \tilde \BX_k \BU_r)
=
\frac{1}{4}
\sum_{i=1}^{m_k}
\BU_r^{\myT} 
(\bSigma_{k, 2i-1} + \bSigma_{k, 2i}) \BU_r
    .
\end{align*}
}%
Consequently,
{\small
\begin{align*}
    \myE 
    \left\|
    \BU_r^{\myT} \tilde \BX_k^\myT \tilde \BX_k \BU_r
    -
\frac{n_k}{4}
    \BU_r^{\myT} \bar \bSigma_k \BU_r
    \right\|_F^2
    \leq&
\sum_{i=1}^{m_k}
    \myE 
    \left\|
    \BU_r^{\myT} \tilde X_{k,i} \tilde X_{k,i}^\myT \BU_r
    -
\frac{1}{4}
\BU_r^{\myT} 
(\bSigma_{k, 2i-1} + \bSigma_{k, 2i}) \BU_r
    \right\|_F^2
    \\
        &
    +
    \frac{1}{16}\|\BU_r^\myT \bSigma_{k, n_k} \BU_r\|_F^2
    \\
    \leq&
\sum_{i=1}^{m_k}
    \myE 
    \left\|
    \BU_r^\myT
    \tilde X_{k,i}
    \right\|^4
    +
    \frac{1}{16} \mytr(\bSigma_{k,n_k}^2)
    %\\
    %\leq&
%\frac{\tau r}{2}
%\sum_{i=1}^{n_k} \mytr\{(\BU_r^{\myT} \bSigma_{k,i} \BU_r )^2\}
    %+
    %\frac{1}{16} \mytr(\bSigma_{k,n_k}^2)
    \\
    \leq&
\tau r
\sum_{i=1}^{n_k} \mytr( \bSigma_{k,i}^2 )
\\
    =& 
        o\left\{
    n_k^2
        \mytr
        (\bar \bSigma_k^2
        )
    \right\}
    .
\end{align*}
}%
where the third inequality follows from \eqref{eq:miaomiao} and the last equality follows from Assumption \ref{assumption7}.
Thus, 
{\small
\begin{align*}
\BU_r^{\myT} \tilde \BX_k^\myT \tilde \BX_k \BU_r = 
\frac{n_k}{4} \BU_r^{ \myT} \bar \bSigma_k \BU_r
+
o_p\left[
n_k
\left\{
 \mytr (\bar \bSigma_k^2)
\right\}^{1/2}
\right]
.
\end{align*}
}%
It follows that
{\small
\begin{align}
    \BB_n \BB_n^\myT = 
    \sigma_{T,n}^{-1}
\BU_r^{\myT}
\bPsi_n
\BU_r
    +
    o_P\left(1\right)
=
2^{-1/2}
\mydiag(\kappa_1, \ldots, \kappa_r)
    +
    o_P\left(1\right)
    .
    \label{eq:79c3}
\end{align}
}%

Note that we have proved that \eqref{eq:79c1}, \eqref{eq:79c2} and \eqref{eq:79c3} hold in probability.
Then for every subsequence of $\{n\}$, there is a further subsequence along which these three equalities hold almost surely, and consequently \eqref{eq:gugu2} holds almost surely by Lemma \ref{lemma:630}.
That is, \eqref{eq:gugu2} holds in probability.
This completes the proof.
\end{proof}

\begin{lemma}\label{lemma:714}
Suppose $\{\xi_i \}_{i=0}^\infty$ is a sequence of independent standard normal random variables
and
$\{\kappa_i\}_{i=1}^\infty$ is a sequence of positive numbers such that $\sum_{i=1}^\infty \kappa_i^2 \in [0, 1]$.
Then the cumulative distribution function $F(\cdot)$ of
$
            (1 - \sum_{i=1}^{\infty} \kappa_{i}^2)^{1/2}
            \xi_0
    +
    2^{-1/2}
    \sum_{i=1}^{\infty} \kappa_i (\xi_i^2-1)
$
is continuous and strictly increasing on the interval $\{x \in \mathbb R: F(x) > 0 \}$.
\end{lemma}
\begin{proof}
If $\kappa_i = 0$, $i =1, 2, \ldots$, then 
the conclusion holds since
$
            (1 - \sum_{i=1}^{\infty} \kappa_{i}^2)^{1/2}
            \xi_0
    +
    2^{-1/2}
    \sum_{i=1}^{\infty} \kappa_i (\xi_i^2-1)
    = \xi_0
$ is a standard normal random variable.
Otherwise, 
we can assume ithout loss of generality that $\kappa_1 > 0$.
Let $\zeta = 
            (1 - \sum_{i=1}^{\infty} \kappa_{i}^2)^{1/2}
            \xi_0
            -2^{-1/2} \kappa_1
    +
    2^{-1/2}
    \sum_{i=2}^{\infty} \kappa_i (\xi_i^2-1)
$.
Then
$
            (1 - \sum_{i=1}^{\infty} \kappa_{i}^2)^{1/2}
            \xi_0
    +
    2^{-1/2}
    \sum_{i=1}^{\infty} \kappa_i (\xi_i^2-1)
    =
    2^{-1/2}
    \kappa_1 \xi_1^2
    + \zeta
$.
Let $f_1(\cdot)$ denote the probability density function of $
    2^{-1/2}
    \kappa_1 \xi_1^2
$.
Let $F_{\zeta}(\cdot)$ denote the cumulative distribution function of $\zeta$.
Then
$
    2^{-1/2}
    \kappa_1 \xi_1^2
    + \zeta
$
has density function
{\small
\begin{align*}
    f(x)  =   \int_{-\infty}^{+\infty}f_1 (x - t) \, \mathrm d F_{\zeta}(t).
\end{align*}
}%
As a result, $F(\cdot)$ is continuous.

Now we prove that $F(\cdot)$ is 
strictly increasing on the interval $\{x \in \mathbb R: F(x) > 0 \}$.
Let $c$ be a point in the support of $\mathcal L (\zeta)$.
For any real number $a$ such that $a > c$ and for any $\delta > 0$,
{\small
\begin{align*}
    &
    \Pr \left\{
    2^{-1/2}
    \kappa_1 \xi_1^2
    + \zeta
    \in (a, a + \delta)
\right\}
\\
\geq&
    \Pr \left\{
    2^{-1/2}
    \kappa_1 \xi_1^2
    \in ( a - c + \delta/4, a- c + 3\delta/4 )
\right\}
    \Pr \left\{
    \zeta
    \in ( c - \delta/4, c + \delta/4 )
\right\}
\end{align*}
}%
Since $c$ is in the support of $\mathcal L (\zeta)$, we have $\Pr( c - \delta / 4 < \zeta < c + \delta / 4 ) > 0$; see, e.g., \cite{Cohn2013Measure}, Section 7.4.
Thus,
$
    \Pr \left\{
    2^{-1/2}
    \kappa_1 \xi_1^2
    + \zeta
    \in (a, a + \delta)
\right\}
>0
$.
Therefore, $F(\cdot)$ is strictly increasing on the interval $(c, + \infty)$.
Then the conclusion follows from the fact that $c$ is an arbitrary point in the support of $\mathcal L (\zeta)$.
\end{proof}

\end{document}